\documentclass[11pt,notitlepage,twoside,a4paper]{amsart}
\usepackage{mathrsfs}

\usepackage{amsmath,amssymb,enumerate}
\usepackage{epsfig,fancyhdr,color}
\usepackage{CJK}
\usepackage{epstopdf}
\usepackage{amssymb}
\usepackage{amsmath,amsthm}
\usepackage{latexsym}
\usepackage{amscd}
\usepackage{psfrag}
\usepackage{graphicx}
\usepackage{epsf}
\usepackage[latin1]{inputenc}
\usepackage[all]{xy}
\usepackage{tikz}
\usepackage{prettyref}
\usepackage{subfigure}
\usepackage{float}
\usepackage{color}
\usepackage{array}
\usepackage{cite}
\usepackage{graphicx}
\usepackage[totalwidth=17cm,totalheight=24cm]{geometry}

\tikzset{elegant/.style={smooth,thick,samples=30}}
\tikzset{eaxis/.style={->,>=stealth}}

\DeclareMathOperator{\sech}{sech}
\DeclareMathOperator{\csch}{csch}
\DeclareMathOperator{\arcsinh}{arcsinh}
\DeclareMathOperator{\arcoth}{arcoth}
\DeclareMathOperator{\arccosh}{arccosh}

\newtheorem{theorem}{\rm\bf Theorem}[section]

\newtheorem{proposition}[theorem]{\rm\bf Proposition}
\newtheorem{lemma}[theorem]{\rm\bf Lemma}
\newtheorem{corollary}[theorem]{\rm\bf Corollary}

\newtheorem*{proposition 1}{\rm\bf Proposition 1}
\newtheorem*{proposition 2}{\rm\bf Proposition 2}

\theoremstyle{definition}
\newtheorem{definition}[theorem]{\rm\bf Definition}

\theoremstyle{remark}
\newtheorem{remark}[theorem]{\rm\bf Remark}

\newtheorem{example}[theorem]{\rm\bf Example}

\newtheorem{question}[theorem]{\rm\bf Question}

\newtheorem*{case 1}{\rm\bf Case 1}
\newtheorem*{case 2}{\rm\bf Case 2}
\newtheorem*{case 3}{\rm\bf Case 3}
\newtheorem*{case 4}{\rm\bf Case 4}

\newtheorem*{subcase 1}{\rm\bf Subcase 1}
\newtheorem*{subcase 2}{\rm\bf Subcase 2}
\newtheorem*{subcase 3}{\rm\bf Subcase 3}
\newtheorem*{subcase 4}{\rm\bf Subcase 4}

\def\interieur#1{\mathord{\mathop{\kern 0pt #1}\limits^\circ}}

\newcounter{ClaimProofCounter}
\newcounter{SthProofCounter}
\newcounter{CaseCounter}
\newcounter{SubcaseCounter}
\title[Teichm\"uller space of infinite type]{The arc metric on Teichm\"uller spaces of surfaces of infinite type with boundary}

\author{Qiyu Chen}
\address{Qiyu Chen: School of Mathematics, Sun Yat-sen
University, 510275, Guangzhou, P. R. China}
\email{chenqy0121@gmail.com}

\author{Lixin Liu}
\address{Lixin Liu: School of Mathematics, Sun Yat-sen University, 510275, Guangzhou, P. R. China}
\email{mcsllx@mail.sysu.edu.cn}

\date{\today}

\begin{document}

	\begin{abstract}
      Let $X_{0}$ be a complete hyperbolic surface of infinite type with geodesic boundary which admits a countable pair of pants decomposition. As an application of the Basmajian identity for complete bordered hyperbolic surfaces of infinite type with limit sets of 1-dimensional measure zero, we define an asymmetric metric (which is called arc metric) on the quasiconformal Teichm\"uller space $\mathcal{T}(X_{0})$ provided that $X_{0}$ satisfies a geometric condition. Furthermore, we construct several examples of hyperbolic surfaces of infinite type satisfying the geometric condition and discuss the relation between the Shiga's condition and the geometric condition.


\thanks{The work is partially supported by NSFC, No: 11271378}
	\end{abstract}

	\maketitle
	\section{Introduction}
     The Thurston metric was originally defined by Thurston \cite{Thurston1998} as an asymmetric metric to solve the extremal problem of finding the best Lipschitz map in the homotopy class of homeomorphisms between two hyperbolic surfaces of finite type without boundary.
     For surfaces of finite type with boundary, a modification of the Thurston metric, the so-called arc metric (see more precisely below) was studied
     \cite{ALPS1, LPST, LTBoundary,PG2010}.

      It's a natural problem whether the Thurston metric is well-defined in Teichm\"uller spaces of surfaces of infinite type (see \cite{Su} for example).
      In this paper, we define the arc metric, a modification of the Thurston metric, on the quasiconformal Teichm\"uller space $\mathcal{T}(X_{0})$ of a complete hyperbolic surface $X_{0}$ of infinite type with geodesic boundary provided that $X_{0}$ satisfies the geometric condition $(\star)$ (see the definition below).

      Let $X_{0}$ be a complete hyperbolic surface of infinite type with geodesic boundary which admits a countable pair of pants decomposition. The completeness for hyperbolic surfaces or hyperbolic structures that we consider throughout this paper means that each boundary component of this surface is a closed geodesic and each puncture of this surface has a neighbourhood which is isometric to a cusp, that is, a surface isometric to the quotient of the region $\{z=x+iy: y>a\}$ of the upper half-plane $\mathbb{H}^{2}$, for some $a>0$, by the isometric group generated by $z\rightarrow z+1$.

      Denote the boundary of $X_{0}$ by $\partial X_{0}$ and denote the set of boundary components of $X_{0}$ by $\mathcal{B}(X_{0})=\{\beta_{1}, \beta_{2}, ... , \beta_{k}, ... \}$.
      Note that the number of boundary components of $X_{0}$ and the number of cusps of $X_{0}$ can be countably infinite.

      $(X,f)$ is said to be  a \emph{marked hyperbolic surface} if $X$ is a complete hyperbolic surface of infinite type and $f:X_{0}\rightarrow X$ is a quasiconformal mapping which leaves each puncture and each boundary component setwise fixed. Two marked hyperbolic surfaces $(X_{1},f_{1})$ and $(X_{2},f_{2})$ are said to be \emph{equivalent} if $f_{2}\circ f_{1}^{-1}$ is homotopic to an isometry from $X_{1}$ to $X_{2}$. Denote the equivalence class of  $(X,f)$ by  $[X,f]$.
       We denote by $\mathcal{T}(X_{0})$  the \emph{reduced quasiconformal Teichm\"uller space} of $X_{0}$ (see \cite{Earle,LP,Thurston-notes}), which is the set of equivalence classes of marked hyperbolic surfaces. It deserves to  mention that all Teichm\"uller spaces that we consider here are reduced, which means that homotopies do not necessarily fix $\partial X_{0}$ pointwise.

       For the sake of simplicity, we shall call $\mathcal{T}(X_{0})$ the Teichm\"uller space of $X_{0}$ for short and denote a marked hyperbolic surface $(X,f)$ or its equivalence class $[X,f]$ by  $X$, without explicit reference to the marking.

               The Teichm\"uller space $\mathcal{T}(X_{0})$ has a complete distance $d_{T}$ called \emph{the Teichm\"uller distance} which is defined by
        \begin{equation*}
				d_{T}([X_{1}, f_{1}], [ X_{2}, f_{2}]) = \frac{1}{2} \log \inf_{g\simeq
               f_{2}\circ f_{1}^{-1}}K[g],
		\end{equation*}
        where the infimum is taken over all quasiconformal mappings $g:X_{1}\rightarrow X_{2}$ homotopic to $ f_{2}\circ f_{1}^{-1}$ and $K[g]$ is the maximal dilation of $g$.

          Recall that a \emph{pair of pants} is a surface whose interior is homeomorphic to a sphere with three disjoint closed disks removed whose boundary is a (possibly empty) disjoint union of circles. A \emph{generalized hyperbolic pair of pants} is a pair of pants equipped with a convex hyperbolic metric in which every topological hole corresponds to either a closed boundary geodesic or a cusp. In particular, a \emph{hyperbolic pair of pants} is a generalized hyperbolic pair of pants with three closed geodesic boundary components.

       A \emph{pair of pants decomposition} of a hyperbolic surface $X$ is a system of pairwise disjoint simple closed geodesics $\mathcal{P}=\{C_{i}\}_{i\in I}$  (for convenience, we ignore the degenerated ones which are homotopic to punctures) such that  $X\setminus (\cup_{i\in I} C_{i})$ is a disjoint union of the interior of generalized hyperbolic pairs of pants. Moreover, if $\mathcal{P}$ is countably infinite, we say that $X$ admits a \emph{countable pair of pants decomposition}. Note that the hyperbolic surfaces of infinite type in this paper are assumed to admit a countable pair of pants decomposition.

  \subsection{Related definitions and notations.}
       Let $S$ be a surface with negative Euler characteristic. A simple closed curve on $S$ is said to be \emph{interior} if it is contained in the interior of $S$. It is said to be \emph{peripheral} if it is homotopic to a 
       puncture. It is said to be \emph{essential} if it is 
       neither peripheral nor isotopic to a point.
       Let $\mathcal{S}(S)$ denote the set of homotopy classes of essential simple closed curves on $S$.

        If $S$ has non-empty boundary $\partial S$, we denote the set of boundary components of $S$ by $\mathcal{B}(S)$. An \emph{arc} on $S$ is the image of a compact interval, which is immersed in $S$, with its interior (possibly with self-intersections) contained in the interior of $S$ and its endpoints lying on $\partial S$. In particular, a \emph{simple arc} is an arc without self-intersections. An arc is said to be \emph{essential} if it is not isotopic (relative to $\partial S$) to a subset of $\partial S$. Note that we do not require the homotopies to fix $\partial S$ pointwise.
       Denote by $\mathcal{A}(S)$ the set of homotopy classes of essential arcs on $S$ and by $\mathcal{A}'(S)$ the subset of $\mathcal{A}(S)$ consisting of homotopy classes of essential simple arcs on $S$.

       For any $\alpha \in \mathcal{S}(S)\cup\mathcal{A}(S)$ and any hyperbolic structure $X$ on $S$, we denote by $\ell_{\alpha}(X)$ the \emph{hyperbolic length} of $\alpha$ on $X$, that is, the length of the (unique) geodesic representative of $\alpha$ on the hyperbolic surface $X$.

        If $S$ is a surface of finite type without boundary, the \emph{Thurston metric} $d_{Th}$ (see \cite{Thurston1998}) is an asymmetric metric on the Teichm\"uller space $\mathcal{T}(S)$ defined by
       \begin{equation*}
       d_{Th}(X,Y)=\log\sup_{\alpha\in\mathcal{S}(S)}
       \frac{\ell_{\alpha}(Y)}{\ell_{\alpha}(X)},
       \end{equation*}
       for all $X,Y\in\mathcal{T}(S)$.

       If $S$ is a surface of finite type with boundary, the \emph{arc metric} $d_{A}$ (see \cite{ LPST, LTBoundary}) on $\mathcal{T}(S)$, as a modification of the Thurston metric, is defined by

         \begin{equation*}
       d_{A}(X,Y)=\log\sup_{\alpha\in\mathcal{A}'(S)\cup\mathcal{S}(S)}
       \frac{\ell_{\alpha}(Y)}{\ell_{\alpha}(X)},
       \end{equation*}
       for all $X,Y\in\mathcal{T}(S)$. It's essential to consider the union of closed curves and arcs in the definition of $d_{A}$ for surfaces of finite type with boundary, since there exist two distinct hyperbolic structures $X,Y$ (see \cite{Parlier,PG2010}) on $S$ such that $\ell_{Y}(\alpha)<\ell_{X}(\alpha)$ for all $\alpha\in\mathcal{S}(S)$. This implies that
         \begin{equation*}
     \log\sup_{\alpha\in\mathcal{S}(S)}
       \frac{\ell_{\alpha}(Y)}{\ell_{\alpha}(X)}\leq 0.
       \end{equation*}
       Moreover, it was shown in \cite{LPST} that
       \begin{equation*}
     \log\sup_{\alpha\in\mathcal{A}'(S)\cup\mathcal{S}(S)} \frac{\ell_{\alpha}(Y)}{\ell_{\alpha}(X)}
     =\log\sup_{\alpha\in\mathcal{A}'(S)\cup\mathcal{B}(S)} \frac{\ell_{\alpha}(Y)}{\ell_{\alpha}(X)}.
       \end{equation*}
       Therefore, the arc metric $d_{A}$ can be also defined by the following formula
         \begin{equation*}
       d_{A}(X,Y)=\log\sup_{\alpha\in\mathcal{A}'(S)\cup\mathcal{B}(S)}
       \frac{\ell_{\alpha}(Y)}{\ell_{\alpha}(X)}.
       \end{equation*}

        Recall that a \emph{Fuchsian group} is a torsion-free discrete group of orientation-preserving isometries on $\mathbb{H}^{2}$. Let $R$ be a hyperbolic Riemann surface. Denote by $\Gamma_{R}$ the \emph{Fuchsian group of $R$}, which is the Fuchsian group such that $R$ is the quotient of $\mathbb{H}^{2}$ by $\Gamma_{R}$.

       Denote by $\Lambda(\Gamma)$ the \emph{limit set} of a Fuchsian group $\Gamma$ acting on the upper half-plane $\mathbb{H}^{2}$, which is a set of points on $\widehat{\mathbb{R}}=\mathbb{R}\cup \{\infty\}$ where the orbit by $\Gamma$ accumulates. Moreover, the complement of $\Lambda(\Gamma)$ in $\widehat{\mathbb{R}}$ is said to be the \emph{set of discontinuity}, which is denoted by $\Omega(\Gamma)$.
        $\Gamma$ is said to be of the \emph{first kind} if $\Omega(\Gamma)$ is empty, otherwise it is said to be of the \emph{second kind}. Note that the Fuchsian group of a bordered Riemann surface is of the second kind. The Fuchsian groups we consider in this paper are of the second kind and infinitely generated unless otherwise indicated.

        Let $C(\Lambda(\Gamma))$ be the convex hull in $\mathbb{H}^{2}$ of the limit set $\Lambda(\Gamma)$ and let $\partial C(\Lambda(\Gamma))$ be the boundary of $C(\Lambda(\Gamma))$ in $\mathbb{H}^{2}$. The \emph{convex core} $C_{R}$ of a hyperbolic Riemann surface $R$ is  the quotient  of $C(\Lambda(\Gamma_{R}))$ by $\Gamma_{R}$, which is the smallest closed convex subregion of $R$ such that its inclusion map induces a  homotopy equivalence. For a hyperbolic surface $X$, we denote by $\Gamma_{X}$ the Fuchsian group of the Riemann surface with convex core $X$.

        \begin{definition}\label{definition of a removable set}
        For a Fuchsian group $\Gamma$, we say that a disjoint union of regions $A=\cup_{n\in \mathbb{N}}A_{n}$ in $\mathbb{H}^{2}$ is \emph{removable for $\Gamma$} (see \cite{Matsuzaki}) if it satisfies the following conditions:
        \begin{enumerate}[(1)]
          \item Each $A_{n}$ is a simply connected open set in $\mathbb{H}^{2}$ which is either a hyperbolic disk, a horodisk tangent to $\widehat{\mathbb{R}}$ or an $r$-neighbourhood of a complete geodesic in $\mathbb{H}^2$ for some $r>0$ (note that the radius $r$ depends on the choice of the complete geodesic and not necessarily uniformly bounded).
          \item The set $A$ is invariant under the action of $\Gamma$.
        \end{enumerate}
       \end{definition}

     \begin{definition}
     We say that  $X_{0}$ satisfies \emph{the geometric condition $(\star)$} (see \cite{Matsuzaki}) if there is a positive constant $L$ and a removable set $A$ for $\Gamma_{X_{0}}$ such that 
     all points of $C(\Lambda(\Gamma_{X_0}))\setminus A$ lie within a distance $L$ of $\partial C(\Lambda(\Gamma_{X_0}))$.
     \end{definition}

      \subsection{Main theorems.}

 \newtheorem*{thm4.7}{\rm\bf Theorem 4.7}
\begin{thm4.7}\label{Thurston metric}
  {\it Let $X_{0}$ be a complete hyperbolic surface of infinite type with boundary which satisfies the geometric condition $(\star)$. Then the following two functions $d$ and $\overline{d}$ on $\mathcal{T}(X_{0})\times \mathcal{T}(X_{0})$ are asymmetric metrics, where
   \begin{equation*}
   \begin{split}
   d(X,Y)=\log\sup_{\alpha \in \mathcal{A}(X_{0})\cup\mathcal{S}(X_{0})}\frac{\ell_{\alpha}(Y)}{\ell_{\alpha}(X)},\\
     \overline{d}(X,Y)=\log\sup_{\alpha \in \mathcal{A}(X_{0})\cup\mathcal{S}(X_{0})}\frac{\ell_{\alpha}(X)}{\ell_{\alpha}(Y)},
     \end{split}
    \end{equation*}
  for all $X, Y\in \mathcal{T}(X_{0})$.}
\end{thm4.7}

The asymmetric metric $d$ is an analogue, for surfaces of infinite type with boundary, of the arc metrics defined for surfaces of finite type with boundary. We also call $d$ the \emph{arc metric} on $\mathcal{T}(X_{0})$.

\newtheorem*{thm4.10}{\rm\bf Theorem 4.10}
\begin{thm4.10}\label{asymmetric arc metrics}
     {\it  Let $X_{0}$ be a complete hyperbolic surface of infinite type with boundary, then the following equality still holds for all $X, Y\in \mathcal{T}(X_{0})$.
     \begin{equation*}
  \sup_{\alpha \in \mathcal{A}(X_{0})\cup\mathcal{B}(X_{0})}\frac{\ell_{\alpha}(Y)}{\ell_{\alpha}(X)}
  =\sup_{\gamma \in \mathcal{A}(X_{0})\cup\mathcal{S}( X_{0})}\frac{\ell_{\gamma}(Y)}{\ell_{\gamma}(X)}.
   \end{equation*}
  In particular, if $X_{0}$ satisfies the geometric condition $(\star)$, then the following equality defines the same asymmetric metric.
   \begin{equation*}
   \log\sup_{\alpha \in \mathcal{A}(X_{0})\cup\mathcal{B}(X_{0})}\frac{\ell_{\alpha}(Y)}{\ell_{\alpha}(X)}
    =\log\sup_{\gamma \in \mathcal{A}(X_{0})\cup\mathcal{S}( X_{0})}\frac{\ell_{\gamma}(Y)}{\ell_{\gamma}(X)}.
   \end{equation*}}
   \end{thm4.10}

         \subsection{Outline of the paper.} In Section 2 we give the Basmajian identity and the generalized McShane identity for complete bordered hyperbolic surfaces of infinite type with limit sets of 1-dimensional measure zero. In Section 3, we consider the geometric condition $(\star)$ and discuss its properties. In Section 4, we define an asymmetric metric on $\mathcal{T}(X_{0})$ and give the proofs of Theorem \ref{Thurston metric} and
         Theorem \ref{asymmetric arc metrics}.  In the last section, we construct several examples of hyperbolic surfaces of infinite type satisfying the geometric condition $(\star)$ and discuss the relation between the Shiga's condition and the geometric condition $(\star)$.

         \subsection{Acknowledgements} The authors would like to thank Professors Katsuhiko Matsuzaki for his help and suggestions.

      \section{Basmajian identity and generalized McShane identity for complete bordered hyperbolic surfaces of infinite type}

    In this section we present the Basmajian identity and the generalized McShane identity for a complete bordered hyperbolic surface $X$ of infinite type with the limit set $\Lambda(\Gamma_{X})$ of 1-dimensional measure zero. The Basmajian identity is a direct result of the orthogonal spectrum theorem given by Basmajian\cite{Basmajian1} if the limit set of the Fuchsian group $\Gamma_{X}$ has 1-dimensional measure zero. We sketch the proof of the generalized McShane identity and refer to \cite{Bridgeman2} for details.

       \subsection{Basmajian identity for complete bordered hyperbolic surfaces of infinite type.}
       For the convenience of the exposition of the orthogonal spectrum theorem given by Basmajian, we introduce the related notations and terminology (see \cite{Basmajian1}).

       Let $M^{n}$ be an orientable hyperbolic manifold of dimension $n\geq2$. A \emph{hypersurface} $S$ in $M^{n}$ is a codimension one complete submanifold endowed with the induced metric. $S$ is said to be \emph{totally geodesic} if every geodesic on $S$ is a geodesic in $M^{n}$.

       Let $S_{1}$ be a totally geodesic hypersurface which is either disjoint from $S$ or equal to $S$. Two paths from $S$ to $S_{1}$ are said to be \emph{freely homotopic relative to $S$ and $S_{1}$} if there is a homotopy in $M^{n}$ between them which keeps the initial point in $S$ and the terminal point in $S_{1}$. The equivalence class of a path $\alpha$ is called the \emph{relative free homotopy class of $\alpha$} and it is said to be \emph{trival} if $S=S_{1}$ and $\alpha$ is homotopic to a single point in $S$.

      Hypersurfaces $S$ and $S_1$ are called \emph{asymptotic} if there exists a path from $S$ to $S_1$ such that its relative free homotopy class is nontrival and contains paths of arbitrary short length. In this case, the length of the homotopy class is defined to be zero. If $S$ and $S_1$ are not asymptotic, then each nontrival relative free homotopy class of a path $\alpha$ from $S$ to $S_1$ contains a shortest path which is the unique common orthogonal in the class $[\alpha]$. The length of this homotpy class $[\alpha]$ is defined to be the length of the common orthogonal in $[\alpha]$.

       Let $\mathcal{C}$ be a (possibly infinite) set of mutually disjoint embedded totally geodesic hypersurfaces in $M^{n}$. For each non-negative integer $k$, the \emph{$k$-th orthogonal spectrum of $M^{n}$ related to $S$ and $\mathcal{C}$} is denoted by $\mathcal{O}_{k}(M^{n};S,\mathcal{C})$, which is the ordered nondecreasing sequence of lengths of nontrival relative free homotopy classes of paths which start in $S$ and go in the direction of the normal to $S$, cross $\mathcal{C}$ along the way $k$ times, and end in a hypersurface contained in $\mathcal{C}$ perpendicularly. Note that the direction of the normal to $S$ here is chosen appropriately on one side, such that the lifts starting from $\tilde{S}$ of those paths lie to the same side of $\tilde{S}$ for a fixed connected component $\tilde{S}$ of a lift of $S$.

     Denote by $m_{h}$ the hyperbolic measure on $S$ inherited from the volume element on $M^{n}$, and by $V_{n}(r)$ the hyperbolic volume of the $n$-dimensional ball of radius $r$.

      \begin{theorem}\label{Original Basmajian theorem}\emph{(Basmajian \cite{Basmajian1}, The Orthogonal Spectrum Theorem)}
      Let $\mathcal{C}$ be a disjoint set of embedded totally geodesic hypersurfaces in the hyperbolic manifold $M^{n}$ and let $S$ be an embedded oriented hypersurface which is totally geodesic. Suppose further that $S$ is either disjoint from $\mathcal{C}$ or one of the hypersurfaces in $\mathcal{C}$, and that no nontrival relative free homotopy class from $S$ to $\mathcal{C}$ has length zero. Then the $k$-th orthogonal spectrum,
      \begin{equation*}
      \mathcal{O}_{k}(M^{n};S,\mathcal{C})=\{d_{i}\},
      \end{equation*}
      satisfies:
       \begin{equation}
      Vol_{n-1}(S)=m_{h}(F_{k})+\sum_{i=1}^{\infty} V_{n-1}(r(d_{i})),
      \end{equation}
      where $F_{k}$  is the subset of $S$ consisting of all points whose corresponding oriented normal ray to $S$ 
       intersects $\mathcal{C}$ at most $k$ times, and $r(x)=\log\coth(\frac{x}{2})$.
      \end{theorem}
      Applying Theorem \ref{Original Basmajian theorem} and the method introduced by Basmajian for the proof of Corollary 1.2 in \cite{Basmajian1}, we have the following proposition.

       \begin{proposition}
       \label{Basmajian identity theorem}
     Let $X$ be a complete bordered hyperbolic surface of infinite type with the limit set $\Lambda(\Gamma_{X})$ of 1-dimensional measure zero. Then for any $\beta_{j}\in\mathcal{B}(X)$, we have
      \begin{equation}\label{Basmajian identity}
        \ell_{\beta_{j}}(X)=\sum_{i=1}^{\infty} 2\log \coth(\frac{d_{i}^{j}(X)}{2}),
        \end{equation}
      where
      $\{d_{i}^{j}(X)\}^{\infty}_{i=1}$ denotes the 0-th orthogonal spectrum $\mathcal{O}_{0}(X;\beta_{j}, \mathcal{B}(X))$ of $X$ related to $\beta_{j}$ and $\mathcal{B}(X)$.
      \end{proposition}

        \begin{proof}
        As in Theorem \ref{Original Basmajian theorem}, we let $M^{n}=X$, $\mathcal{C}=\mathcal{B}(X)=\{\beta_{1}, \beta_{2}, ... , \beta_{k}, ... \}$, $S=\beta_{j}$. Consider the orthogonal spectrum $\mathcal{O}_{0}(X;\beta_{j}, \mathcal{B}(X))=\{d_{i}^{j}(X)\}^{\infty}_{i=1}$ and it follows from Theorem \ref{Original Basmajian theorem} that
        \begin{equation*}
        \ell_{\beta_{j}}(X)=m_{h}(F_{0}^{j})+\sum_{i=1}^{\infty} 2\log \coth(\frac{d_{i}^{j}(X)}{2}),
        \end{equation*}
        where $F_{0}^{j}$ is the subset of $\beta_{j}$ consisting of the points from which the oriented geodesics starting perpendicularly never hit $\partial X$.

         Denote by $\mathcal{G}_{j}$ the set of all the complete geodesics which start perpendicularly from $\beta_{j}$ and never hit $\partial X$. It is not hard to see that for any geodesic $g\in \mathcal{G}_{j}$, the endpoint at infinity of a lift of $g$ must lie on the limit set of the Fuchsian group $\Gamma_{X}$. Fix a connected component $\widetilde{\beta_{j}}$ of a lift of $\beta_{j}$ and denote by $V_{j}$ the set of the endpoints at infinity of the lifts starting from $\widetilde{\beta_{j}}$ of all the geodesics in $\mathcal{G}_{j}$. It is clear that $V_{j}\subset \Lambda(\Gamma_{X})$.

          Observe that the endpoints of $\widetilde{\beta_{j}}$ divides the circle at infinity $S_{\infty}^{1}$ into two disjoint open components. We endow $S_{\infty}^{1}$ with 1-dimensional Lebesgue measure and let $\mathcal{R}^{+}$ be the open component for which the normal to $\widetilde{\beta_{j}}$ points. Consider the map $p_{j}: \mathcal{R}^{+} \rightarrow \beta_{j}$ given by orthogonal projection to $\widetilde{\beta_{j}}$ followed by the covering map into the quotient surface $X$. Then $F_{0}^{j}$ is exactly $p_{j}(V_{j})$. By the assumption that the limit set $\Lambda(\Gamma_{X})$ has $1$-dimensional measure zero and by the fact that $p_{j}$ preserves sets of measure zero (see Proposition 3.3 in \cite{Basmajian1}), we derive that $m_{h}(F_{0}^{j})=m_{h}(p_{j}(V_{j}))=0$.
       Hence,
          \begin{equation*}
        \ell_{\beta_{j}}(X)=\sum_{i=1}^{\infty} 2\log \coth(\frac{d_{i}^{j}(X)}{2}).
        \end{equation*}
       \end{proof}

         \subsection{Generalized McShane identity for complete bordered hyperbolic surfaces of infinite type.}
                 The  generalized McShane identity for bordered hyperbolic surfaces of finite type is given by Mirzakhani \cite{Mirzakhani}. To generalize it to the case of a complete hyperbolic surface $X$ of infinite type with boundary, we apply the method given by Bridgeman and Tan \cite{Bridgeman2}. The way is to consider the boundary flows on the surface $X$.

        Indeed, let $T_{1}(X)$ be the unit tangent bundle of $X$ and $\pi:T_{1}(X)\rightarrow X$ be the projective map. Fix a boundary component $\beta_{1}$ of $X$ and denote by $W$ the subset of $T_{1}(X)$ consisting of the vectors with basepoints on $\beta_{1}$ which are perpendicular to $\beta_{1}$ and point to the interior of $X$. It is obvious that $\pi$ is a bijection from $W$ to $\beta_{1}$. We identify $W$ with $\beta_{1}$ under $\pi$ and define the measure $\mu$ on $W$ to be the pull back of the 1-dimensional Lebesgue measure on $\beta_{1}$ under $\pi$. In particular, $\mu(W)=\ell_{\beta_{1}}(X)$. Then we consider the geodesic $g_{v}$ starting at $p=\pi(v)\in\beta_{1}$ obtained by exponentiating $v$, where $g_{v}$ is assumed to stop when it hits itself or the boundary $\partial X$.

        Let $Z\subset W$ be the set of vectors in which $g_{v}$ starts has infinite length. It is not hard to see that for every $v\in W\setminus Z$, $g_{v}$ is a geodesic arc contained in a unique generalized hyperbolic pair of pants embedded in $X$ bounded by $\beta_{1}$ and a pair of simple closed curves $\gamma_{1}$ and $\gamma_{2}$ (where either $\gamma_{1}$, $\gamma_{2}$ are both interior simple closed geodesics, or exactly one of them is an interior simple closed geodesic while the other is a geodesic boundary component or a cusp distinct from $\beta_{1}$). Denote by $\mathcal{P}$ the set of all such pairs of pants embedded in $X$. For each $P\in\mathcal{P}$, let $X_{P}=\{v\in W\setminus Z: g_{v}\subset P\}$, then $W=Z\cup(\cup_{P\in\mathcal{P}}X_{P})$.
        Hence, $\ell_{\beta_{1}}(X)=\sum\limits_{P\in\mathcal{P}}\mu(X_{P})+\mu(Z)$.

        If $\gamma_{1}$ and $\gamma_{2}$ are both interior simple closed geodesics, it can be computed by elementary hyperbolic geometry that $\mu(X_{P})=\mathcal{D}(\ell_{\beta_{1}}(X), \ell_{\gamma_{1}}(X), \ell_{\gamma_{2}}(X))$. Otherwise, assume that $\gamma_{1}$ is a geodesic boundary component (may be a cusp) and $\gamma_{2}$ is an interior simple closed geodesic. It can be computed that $\mu(X_{P})=\mathcal{R}(\ell_{\beta_{1}}(X), \ell_{\gamma_{1}}(X), \ell_{\gamma_{2}}(X))$. Here the functions  $\mathcal{D}$ and $ \mathcal{R}$ are respectively defined by
      \begin{equation*}
       \mathcal{D}(x_{1},x_{2},x_{3}) = 2 \log \left(\frac{e^{\frac{x_{1}}{2}}+e^{\frac{x_{2}+x_{3}}{2}}}
      {e^{\frac{-x_{1}}{2}}+e^{\frac{x_{2}+x_{3}}{2}}}\right),
      \end{equation*}
      \begin{equation*}
      \mathcal{R}(x_{1},x_{2},x_{3}) = x_{1}-\log\left(\frac{\cosh\frac{x_{2}}{2}+
      \cosh\frac{x_{1}+x_{3}}{2}} {\cosh\frac{x_{2}}{2}+\cosh\frac{x_{1}-x_{3}}{2}}\right).
      \end{equation*}

        The difficulty is how to ensure that $\mu(Z)=0$. However, if the limit set of the Fuchsian group $\Gamma_{X}$ has 1-dimensional measure zero, it is true that $\mu(Z)=0$. This proof is similar to the proof for $m_{h}(F_{0}^{j})=0$ in Proposition
        \ref{Basmajian identity theorem}.

       Therefore, the generalized McShance identity still holds for $X$ if $\Gamma_{X}$ has 1-dimensional measure zero.
       Then we have the following proposition.

         \begin{proposition}
       \label{McShane identity theorem}
      Let $X$ be a complete bordered hyperbolic surface of infinite type with the limit set $\Lambda(\Gamma_{X})$ of 1-dimensional measure zero. Let $\beta_{1}$ be a boundary component of $X_{0}$ with $\ell_{\beta_{1}}(X)>0$. Then we have
       \begin{equation}\label{McShane identity}
      \begin{split}
      \sum_{\{\gamma_{1},\gamma_{2}\}\in \mathcal{F}_{1}} \mathcal{D}(L_{1},\ell_{\gamma_{1}},\ell_{\gamma_{2}})+
      \sum_{i=2}^{\infty}\sum_{\gamma\in\mathcal{F}_{1,i}}
      \mathcal{R}(L_{1},L_{i},\ell_{\gamma})=L_{1}.
      \end{split}
      \end{equation}
      Here $L_{i}=\ell_{\beta_{i}}(X), \ell_{\gamma_{i}}=\ell_{\gamma_{i}}(X)$ and $\mathcal{B}(X)=\{\beta_{1}, \beta_{2}, ... , \beta_{k}, ... \}$. In particular, we include the cusps as geodesic boundary components of length zero in $\mathcal{B}(X)$.
      $\mathcal{F}_{1}$ denotes the set of all the unordered pairs of isotopy classes of interior simple closed curves which bound a pair of pants with $\beta_{1}$. $\mathcal{F}_{1,i}$ denotes the set of all the isotopy classes of interior simple closed curves which bound a pair of pants with $\beta_{1}$ and $\beta_{i}$.
       \end{proposition}

  Recall that a class $\mathcal{O}$ of Fuchsian groups is \emph{quasiconformally invariant} \cite{Matsuzaki1} if it satisfies that for any Fuchsian group $\Gamma\in\mathcal{O}$, if there is a quasiconformal homeomorphism $f$ of $\mathbb{H}^{2}$ such that $\Gamma'=f\Gamma f^{-1}$ is Fuchsian, then $\Gamma'$ belongs to $\mathcal{O}$.

  \begin{remark}\label{Matsuzaki}
  It was remarked in \cite{Matsuzaki1} by Matsuzaki that the class of Fuchsian groups whose limit set has vanishing 1-dimensional measure is not quasiconformally invariant (see Example 2 in \cite{Taniguchi} and Theorem 3 in  \cite{BA}). Thus it's possible that the Basmajian identity and the generalized McShane identity fail to hold for the Teichm\"uller space of a hyperbolic surface of infinite type with the limit set of 1-dimensional measure zero. To overcome this difficulty, we consider the geometric condition $(\star)$ in the next section.
   \end{remark}

   \section{A geometric condition}

   In this section, we aim to show that the Basmajian identity and the generalized McShane identity hold for $\mathcal{T}(X_{0})$ provided that $X_0$ satisfies the geometric condition $(\star)$. The key is to show the 1-dimensional measure of the limit set $\Lambda(\Gamma_{X})$ of each $X\in\mathcal{T}(X_{0})$ is zero.

   First we discuss some properties of the geometric condition $(\star)$. 
    To state and verify the related results, we fix some terminology and notations first.

    We say that a 
    map
     $f:(X_{1}, d_{1})\rightarrow (X_{2}, d_{2})$ between two metric spaces is \emph{bi-Lipschitz} if there exists a real number $L\geq 1$ satisfying
    \begin{equation*}
    \frac{1}{L} d_{1}(x,y)\leq d_{2}(x,y) \leq L d_{1}(x,y)
    \end{equation*}
    for any $x, y \in X_{1}$.
    The real number $L$ is called a \emph{bi-Lipschitz constant of f}. Two metric spaces are said to be \emph{bi-Lipschitz equivalent} if there exists a bi-Lipschitz homeomorphism between them.

    Let $X_{0}$ be a complete hyperbolic surface of infinite type with geodesic boundary.
    We denote by $\mathcal{T}_{bL}(X_{0})$ the \emph{bi-Lipschitz Teichm\"uller space} of $X_{0}$, which is the set of equivalence classes of pairs $(X,f)$, where $X$ is a complete hyperbolic surface of infinite type and $f: X_{0}\rightarrow X$ is a bi-Lipschitz homeomorphism with respect to the hyperbolic metrics which leaves each puncture and each boundary component setwise fixed. Here two pairs $(X_{1},f_{1})$ and $(X_{2},f_{2})$ are said to be \emph{equivalent} if $f_{2}\circ f_{1}^{-1}$ is homotopic to an isometry from $X_{1}$ to $X_{2}$. Denote the equivalence class of  $(X,f)$ by  $[X,f]$. It deserves to  mention that the homotopies do not necessarily fix $\partial X_{0}$ pointwise.

    In $\mathcal{T}_{bL}(X_{0})$, we consider the \emph{bi-Lipschitz metric} $d_{bL}$ (see \cite{LP}) which is defined by
    \begin{equation*}
    d_{bL}([X_{1},f_{1}], [X_{2}, f_{2}])=  \frac{1}{2} \log \inf_{g\simeq
               f_{2}\circ f_{1}^{-1}}L(g),
    \end{equation*}
    where the infimum is taken over all bi-Lipschitz homeomorphisms $g:X_{1}\rightarrow X_{2}$ homotopic to $ f_{2}\circ f_{1}^{-1}$ and $L(g)$ is the bi-Lipschitz constant of $g$.

  \begin{theorem}\label{Matsuzaki 1}
  \emph{(Matsuzaki \cite{Matsuzaki})}
  Let $\Gamma$ be a Fuchsian group acting on the upper half-plane $\mathbb{H}^{2}$. If there is a positive constant $L$ and a removable set $A$ for $\Gamma$ such that all points of $C(\Lambda(\Gamma))\setminus A$ lie within a distance $L$ of $\partial C(\Lambda(\Gamma))$, then there is a constant $\alpha \in (0,1)$ depending only on $L$ such that the Hausdorff dimension of the limit set of $\Gamma$ satisfies $\dim\Lambda(\Gamma)\leq \alpha <1$.
       \end{theorem}

       \begin{remark}
      The condition in Theorem \ref{Matsuzaki 1} is exactly the geometric condition  ($\star$) in Definition \ref{definition of a removable set}. 
      In the estimate of the Hausdorff dimension $\dim\Lambda(\Gamma)$ in Theorem \ref{Matsuzaki 1} (see \cite[Theorem 1]{Matsuzaki}), Matsuzaki aimed to show that only the depth of the convex core $C(\Lambda(\Gamma))/\Gamma$ without the removable set is important. For a removable set $A$ for $\Gamma$, the components as horodisks and neighbourhoods of complete geodesics in $\mathbb{H}^{2}$ are used to deal with the thin parts of $C(\Lambda(\Gamma))/\Gamma$, while the components as hyperbolic disks are used to deal with the thick parts of $C(\Lambda(\Gamma))/\Gamma$. We give the corresponding examples in Section 5, see Example \ref{ex: horodisks}, Example \ref{ex: neighbourhoods of a geodesic} and Example  \ref{ex: hyperbolic disks} respectively.
       \end{remark}

       \begin{definition}\label{definition of a weakly removable set}
        For a Fuchsian group $\Gamma$, we say that a disjoint union of regions $A=\cup_{n\in \mathbb{N}}A_{n}$ in $\mathbb{H}^{2}$ is \emph{weakly removable for $\Gamma$} (see \cite{Matsuzaki}) if it satisfies the following conditions:
        \begin{enumerate}[(1)]
         \item Each $A_{n}$ is an open set in $\mathbb{H}^{2}$ whose euclidean closure intersects $\widehat{\mathbb{R}}$ with the set of 1-dimensional measure zero.
         \item The set $A$ is invariant under the action of $\Gamma$.
       \end{enumerate}
       \end{definition}

       \begin{definition}\label{Def of weak geometric condition}
       We say that $X_{0}$ satisfies \emph{the weak geometric condition $(\diamond)$} if there is a positive constant $L$ and a weakly removable set $A$ for $\Gamma_{X_{0}}$ such that all points of $C(\Lambda(\Gamma_{X_0}))\setminus A$ lie within a distance $L$ of $\partial C(\Lambda(\Gamma_{X_0}))$.
       \end{definition}

        \begin{remark}\label{the weak geometric condition theorem}
        It was proved in \cite[Theorem 5]{Matsuzaki} that if $X_0$ satisfies the weak geometric condition $(\diamond)$, then the 1-dimensional measure of $\Lambda(\Gamma_{X_0})$ is zero. In other words, if the conclusion in Theorem \ref{Matsuzaki 1} is weaken to be that the 1-dimensional measure of $\Lambda(\Gamma)$ is zero, it suffices to consider a weakly removable set for $\Gamma$ instead of a removable set for $\Gamma$.



       \end{remark}

        \begin{theorem}\emph{(Matsuzaki \cite{Matsuzaki})} \label{Matsuzaki 2}Let $N_{\Gamma}$ be a hyperbolic surface of infinite topological type and let $\{c_{n}\}_{n=1,2,...}$ be the components of the boundary of the convex core $\partial C_{\Gamma}\subset N_{\Gamma}$. If the hyperbolic lengths $\ell(c_{n})$ satisfy
        \begin{equation*}
        \sum_{n}\ell(c_{n})^{\frac{1}{2}}<\infty,
        \end{equation*}
         then the Hausdorff dimension of the limit set of $\Gamma$ is equal to $1$.
         \end{theorem}

      \begin{theorem}\emph{(Liu, Papadopoulos \cite{LP})}\label{bi-Lip and qc map}
     For every complete hyperbolic surface $X_{0}$ of infinite type, we have the set-theoretic equality
     \begin{equation*}
    \mathcal{T}(X_{0}) = \mathcal{T}_{bL}(X_{0}),
     \end{equation*}
      and there exists a constant $C$ such that for every $X$ and $Y$ in $\mathcal{T}(X_{0})$, we have
     \begin{equation}
     d_{T}(X,Y)\leq d_{bL}(X,Y)\leq Cd_{T}(X,Y).
     \end{equation}
      \end{theorem}

      It is an alternative statement of Theorem 4.3 in \cite{LP}. The idea was originally introduced by Thurston (see \cite{Thurston-notes} p. 268).

       \begin{lemma}\label{the cover of a removable set}
        Let $X_{0}$ be a complete hyperbolic surface of infinite type with boundary. Let $X\in\mathcal{T}(X_0)$ and let $A=\cup_{n\in\mathbb{N}}A_n$ be a removable set for $\Gamma_X$. Then for any subsurface $\Sigma$ of $X$ which contains an essential self-intersecting closed curve that is not $\gamma^n$ for any simple closed curve $\gamma$ and $n\in\mathbb{Z}$, the projection $\pi(A)$ on $X$ of $A$ under $\Gamma_{X}$ fails to cover $\Sigma$.

       \end{lemma}
       \begin{proof}
       Let $\Sigma$ be such a subsurface of $X$ and let $\alpha$ be such an essential self-intersecting closed curve on $\Sigma$. Assume that $\pi(A)$ covers $\Sigma$, then $\pi(A)$ covers $\alpha$. This implies that $A$ contains a connected component of a lift of $\alpha$ in $\mathbb{H}^2$, called $\tilde{\alpha}$. Note that $\alpha$ is self-intersecting and cannot be written as $\gamma^n$ for any simple closed curve $\gamma$ and $n\in\mathbb{Z}$, then $\tilde{\alpha}$ intersects the boundary at infinity $\widehat{\mathbb{R}}$ with at least four points. Moreover, since $A_i$ and $A_j$ are disjoint for all $i\not= j$, then $A$ has a component $A_n$ which contains $\tilde{\alpha}$ and intersects $\widehat{\mathbb{R}}$ with at least four points. By Definition \ref{definition of a removable set}, each component of a removable set intersects $\widehat{\mathbb{R}}$ with at most two points. This contradiction proves that $A$ fails to cover $\Sigma$.
       \end{proof}

       \begin{lemma}\label{Wolpert's inequality in our case}
       Let $X_{0}$ be a complete hyperbolic surface of infinite type with boundary. For any $X=[X,f_1]$, $Y=[Y,f_2]$ in $\mathcal{T}(X_0)$, let $f=f_2\circ {f_1}^{-1}$ and let $K=K[f]$ be the maximal dilation of $f$, then
       \begin{equation*}
            \frac{1}{K}\leq\frac{\ell_{f(\alpha)}(Y)}{\ell_{\alpha}(X)}\leq K,
        \end{equation*}
        for all $\alpha\in\mathcal{S}(X_0)\cup\mathcal{A}(X_0)$.
       \end{lemma}
        \begin{proof}
         We recall a result of Wolpert (see \cite{Wolpert}), which says that given any $K'$-quasiconformal map $h$ between two hyperbolic surfaces $X'$ and $Y'$ without boundary, we have
        \begin{equation*}
            \frac{1}{K'}\leq\frac{\ell_{h(\alpha)}(Y')}{\ell_{\alpha}(X')}\leq K',
        \end{equation*}
        for all isotopy classes of essential closed curves $\alpha$ on $X'$. Note that this result also holds for isotopy classes of essential closed curves and essential arcs on hyperbolic surfaces with boundary, by applying an argument of doubling (see e.g. Theorem 2.1 in \cite{LTBoundary}). Therefore, for any $X=[X,f_1]$, $Y=[Y,f_2]$ in $\mathcal{T}(X_0)$, since $f=f_2\circ {f_1}^{-1}:X\rightarrow Y$ is a $K$-quasiconformal map, we have
         \begin{equation*}
            \frac{1}{K}\leq\frac{\ell_{f(\alpha)}(Y)}{\ell_{\alpha}(X)}\leq K,
        \end{equation*}
        for all $\alpha\in\mathcal{S}(X_0)\cup\mathcal{A}(X_0)$.
        \end{proof}

      Now we give some properties of the geometric condition $(\star)$ as follows.

        \begin{proposition}\label{the measure of the limit set of T(X_0)} Let $X_{0}$ be a complete hyperbolic surface of infinite type with boundary which satisfies the geometric condition $(\star)$. Then for any $X\in \mathcal{T}(X_{0})$, the following statements hold:

        \begin{enumerate}[(1)]
            \item The number of all boundary components of $X$ is countably infinite.
             \item $X$ satisfies the weak geometric condition $(\diamond)$.
            \item The limit set of the Fuchsian group $\Gamma_{X}$ has  1-dimensional measure zero.
            \item The sum of the lengths of all boundary components of $X$ is infinite.
        \end{enumerate}





       \end{proposition}

       \begin{proof}
      \emph{The proof of (1)}: Recall that the hyperbolic surfaces of infinite type in this paper admit a countable pair of pants decomposition.  This implies that if $X$ has infinitely many boundary components, then the number of its boundary components is countably infinite. Note that $X=[X,f]$ for a quasiconformal map $f:X_{0}\rightarrow X$ which leaves each puncture and each boundary component setwise fixed.

      It suffices to prove that the number of all boundary components of $X_{0}$ is infinite. Note that $X_{0}$ is complete and thus each boundary component of $X_{0}$ is a simple closed geodesic, which implies that $X_{0}$ has no boundary component of infinite length.  

      We argue by contradiction. Assume that $X_0$ has finitely many geodesic boundary components and denote them by $\beta_1, \beta_2,...,\beta_n$. Then
      \begin{equation*}
      \sum_{i=1}^{n}\ell_{\beta_i}(X_0)^{\frac{1}{2}}<\infty.
      \end{equation*}
      By Theorem \ref{Matsuzaki 2}, the Hausdorff dimension of the limit set $\Lambda(\Gamma_{X_0})$ is 1. However, note that $X_0$ satisfies the geometric condition ($\star$) and by Theorem \ref{Matsuzaki 1}, the Hausdorff dimension of the limit set $\Lambda(\Gamma_{X_0})$ is less than 1. This produces contradiction.

        \emph{The proof of (2)}:
        By Theorem \ref{bi-Lip and qc map}, there exists a constant $C$ such that for every $[X,f]\in\mathcal{T}(X_{0})$, we have $d_{bL}([X_{0},id],[X,f])\leq C d_{T}([X_{0},id],[X,f])$. By the definition of $d_{bL}$, there exists a bi-Lipschitz homeomorphism $g: X_{0}\rightarrow X$ homotopic to $f$ with the bi-Lipschitz constant 
       $L(g)\leq e^{2Cd_{T}(X_{0}, X)}+\epsilon_{0}$
       where $\epsilon_{0}$ is a sufficiently small positive number. Let 
       $M=e^{2Cd_{T}(X_{0}, X)}+\epsilon_{0}$.
       We obtain that
       \begin{equation}\label{inequality for hyperbolic distance}
      \frac{1}{M} \rho_{X_{0}}(x,y)\leq \rho_{X}(g(x),g(y))\leq M \rho_{X_{0}}(x,y),
       \end{equation}
       for any two points $x$, $y$ on $X_{0}$, where $\rho_{X_{0}}$ (resp. $\rho_{X}$) denotes the hyperbolic distance on $X_{0}$ (resp. $X$) induced by the hyperbolic structure of $X_{0}$ (resp. $X$).

       Since $X_{0}$ satisfies the geometric condition $(\star)$, then there exists a positive constant $L$  and a removable set 
       $A=\cup_{n\in \mathbb{N}}A_{n}\subset\mathbb{H}^{2}$ for $\Gamma_{X_{0}}$
       such that
       \begin{equation*}
       \rho_{X_{0}}(x, \partial X_{0})\leq L,
       \end{equation*}
        for any point $x$ on $X_{0}$ except the image of the removable set $A\subset \mathbb{H}^{2}$ under the universal covering map $\pi_{0}$ of $X_{0}$, where $\rho_{X_{0}}(x, \partial X_{0})=\inf\limits_{y\in \partial X_{0}}\rho_{X_{0}}(x,y)$.

        It follows directly
        from \eqref{inequality for hyperbolic distance} that
        \begin{equation}\label{ineq: bounded distace from boundary of X}
       \rho_{X}(p, \partial X)\leq ML,
       \end{equation}
         for any point $p$ on $X$ except the set 
         $g(\pi_{0}(A))\subset X$.

         Let $\tilde{g}$ be a lift of the map $g$ to the universal covering space of $X$. Set $A'=\tilde{g}(A)$. First we claim that $A'=\cup_{n\in \mathbb{N}}\,\tilde{g}(A_{n})$ is a weakly removable set for $\Gamma_{X}$. Indeed, $A'$ is a disjoint union of open sets in $\mathbb{H}^2$, since $\tilde{g}$ is a homeomorphism and $A$ is a disjoint union of open sets in $\mathbb{H}^2$. Note that $A$ is $\Gamma_{X_0}$-invariant, and $\tilde{g}$ is equivariant with respect to $\Gamma_{X_0}$ and $\Gamma_{X}$, then $A'$ is invariant under the action of $\Gamma_X$. By Definition \ref{definition of a removable set}, for each $n\in\mathbb{N}$, $A_n$ is possibly a hyperbolic disk, a horodisk tangent to $\widehat{\mathbb{R}}$, or an $r$-neighbourhood of a complete geodesic in $\mathbb{H}^2$ for some $r>0$.

         Note that $\tilde{g}:\mathbb{H}^2\rightarrow \mathbb{H}^2$ is a bi-Lipschitz homeomorphism with respect to the hyperbolic metrics, then it induces a homeomorphism from $\widehat{\mathbb{R}}$ to $\widehat{\mathbb{R}}$. If $A_n$ is a hyperbolic disk, then $\tilde{g}(A_n)$ is a topologically disk in $\mathbb{H}^{2}$ whose euclidean closure does not intersect $\widehat{\mathbb{R}}$. If $A_n$ is a horodisk tangent to $\mathbb{\widehat{\mathbb{R}}}$ at $\xi\in\widehat{\mathbb{R}}$, then the euclidean clousre of $\tilde{g}(A_n)$ intersects $\mathbb{\widehat{\mathbb{R}}}$ exactly at $\tilde{g}(\xi)\in\widehat{\mathbb{R}}$. If $A_n$ is a neighbourhood of a complete geodesic with two distinct endpoints $\xi_1,\xi_2\in \widehat{\mathbb{R}}$, then the euclidean clousre of $\tilde{g}(A_n)$ intersects $\mathbb{\widehat{\mathbb{R}}}$ exactly at two distinct points $\tilde{g}(\xi_1), \tilde{g}(\xi_2)\in\widehat{\mathbb{R}}$. Therefore, the euclidean closure of each $\tilde{g}(A_n)$ intersects $\widehat{\mathbb{R}}$ with the set of 1-dimensional measure zero. By Definition \ref{definition of a weakly removable set}, $A'$ is a weakly removable set for $\Gamma_{X}$.

        By \eqref{ineq: bounded distace from boundary of X} and the fact that $g(\pi_{0}(A))=\pi(A')$, there exists a constant $L'=ML>0$ and a removable set $A'=\tilde{g}(A)$ for $\Gamma_{X}$ such that
        \begin{equation*}
       \rho_{X}(p, \partial X)\leq L',
       \end{equation*}
        for all $p\in X\setminus \pi(A')$. This implies that $X$ satisfies the weak geometric condition $(\diamond)$.

        \emph{The proof of (3)}:
        By \emph{Statement (2)} and Remark \ref{the weak geometric condition theorem}, the 1-dimensional measure of the limit set $\Lambda(\Gamma_{X})$ is zero.

         \emph{The proof of (4)}:
         First we prove that it suffices to show this statement for the special case $X=X_0$. Indeed, by Lemma \ref{Wolpert's inequality in our case}, for any $X=[X,f]\in\mathcal{T}(X_0)$, let $K$ be the maximal dilation of $f$, we have
        \begin{equation}\label{length under quasiconformal maps}
            \frac{1}{K}\leq\frac{\ell_{f(\alpha)}(X)}{\ell_{\alpha}(X_0)}\leq K,
        \end{equation}
        for all $\alpha\in\mathcal{S}(X_0)\cup\mathcal{A}(X_0)$.

          By \emph{Statement (1)}, $X_0$ has infinitely many boundary components and denote the set of boundary components of $X_0$ by $\mathcal{B}(X_0) = \{\beta_1,\beta_{2},...,\beta_k,...\}$. By \eqref{length under quasiconformal maps}, we have
         $\sum\limits_{i=1}\limits^\infty \ell_{\beta_{i}}(X)<\infty$ if and only if  $\sum\limits_{i=1}\limits^\infty \ell_{\beta_{i}}(X_0)<\infty$. Hence, we only need to consider $X=X_0$.

         Denote $b_{i}=\ell_{\beta_{i}}(X_0)$. We argue by contradiction. Suppose
          \begin{equation*}
         \sum\limits_{i=1}\limits^\infty b_{i}<\infty.
         \end{equation*}
         Then $b_{i}\rightarrow0$, as $i\rightarrow \infty$. By the collar lemma (see \cite{Buser}), there exists a collar neighbourhood $\mathcal{N}(\beta_{i})=\{p\in X_0:\rho_{X_0}(p,\beta_{i})\leq r(b_{i})\}$ of $\beta_{i}$ such that $\mathcal{N}(\beta_{i})$ does not intersect any other simple closed geodesics disjoint from $\beta_{i}$, where $\rho_{X_0}$ denotes the hyperbolic distance on $X_0$ and $r(b_{i})=\arcsinh\{1/\sinh(\frac{1}{2}b_{i})\}$.

         Note that $r(b_i)\rightarrow \infty$, as $i\rightarrow \infty$. For any $L>0$ and any removable set $A$ for $\Gamma_{X_0}$, there exists an integer $n_0>0$ (depending on $L$) such that $r(b_i)>L$ for all $i\geq n_0$. Denote by $B(\partial X_0; L)$ the set consisting of the points on $X_0$ lying within the distance $L$ of $\partial X_0$. Let $\Omega = X_0\setminus B(\partial X_0; L)$.

           Note that $\Omega$ contains at least two distinct isotopy classes of simple closed curves, then it contains an essential self-intersecting curve $\alpha$ which is not $\gamma^n$ for any simple closed curve $\gamma$ and $n\in\mathbb{Z}$. By Lemma \ref{the cover of a removable set}, $\pi_0(A)$ fails to cover $\Omega$, where $\pi_0$ is the universal covering map of $X_0$. This contradicts the assumption that $X_0$ satisfies the geometric condition $(\star)$.
       \end{proof}


        \begin{remark}
      The assumption that $X_{0}$ is complete is necessary for Statement (1) of Proposition \ref{the measure of the limit set of T(X_0)}. Otherwise, there exists a hyperbolic surface of infinite type called \emph{tight flute surface} by Basmajian
      (see \cite{Basmajian2, Basmajian3}) satisfying the geometric condition ($\star$) but has only one geodesic boundary component which is a simple open infinite geodesic (see Example \ref{incomplete ex}).
       \end{remark}

       Combining Proposition \ref{Basmajian identity theorem},
      Proposition \ref{McShane identity theorem} and Statement (3) of Proposition
      \ref{the measure of the limit set of T(X_0)}, we have the following corollary.

      \begin{corollary}\label{Basmajina and McShane}
      Let $X_{0}$ be a complete hyperbolic surface of infinite type with boundary which satisfies the geometric condition $(\star)$. Then both the Basmajian identity and the generalized McShane identity hold for $\mathcal{T}(X_{0})$.
      \end{corollary}


      \begin{question}
         Is the geometric condition ($\star$) quasiconformally invariant? That is, if $X_{0}$ satisfies the geometric condition ($\star$), then for any $X\in\mathcal{T}(X_{0})$, does $X$ also satisfy the geometric condition ($\star$)? We can also ask the same question for the weak geometric condition $(\diamond)$.
      \end{question}

        The construction of examples of hyperbolic surfaces of infinite type which satisfy the geometric condition $(\star)$ will be given in Section 5.

      \section{An asymmetric metric on $\mathcal{T}(X_{0})$}
       \begin{definition}\label{def of asymmetric metric}
       An \emph{asymmetric metric} on a set $M$ is a function $\delta:M\times M\rightarrow [0,+\infty)$ satisfying the following conditions.
        \begin{enumerate}[(a)]
          \item The separation axiom: for any $x, y \in M$, $\delta(x,y)=0$ if and only if $x=y$.
          \item The triangle inequality: $\delta(x,y)\leq \delta(x,z)+\delta(z,y)$, for all $x, y, z\in M$.
          \item The asymmetric condition: there exists $x, y \in M$, such that $\delta(x,y)\neq \delta(y,x)$.
        \end{enumerate}

        The pair $(M,\delta)$ defined as above is said to be an \emph{asymmetric metric space} (see \cite{Papado-Th2007, Thurston1998}). In particular, a function $f:M \times M\rightarrow [0,+\infty]$ is said to be \emph{positive definite} if it satisfies the separation axiom (a).
      \end{definition}

           For a Nielsen convex hyperbolic surface $X$ (equivalently, $X$ can be constructed by gluing some generalized hyperbolic pairs of pants along their boundary components), the \emph{Fenchel-Nielsen coordinates} of $X$ associated with a pair of pants decomposition $\mathcal{P}=\{C _{i}\}^{\infty}_{i=1}$ (see \cite{ALPSS}) is defined to be $\{\ell_{C_{i}}(X), t_{C_{i}}(X)\}^{\infty}_{i=1}$ consisting of the hyperbolic lengths with respect to $X$ of all the simple closed curves in $\mathcal{P}$ and the twisting parameters used to glue the pairs of pants, where the positive direction of twisting means turning left. It is understood that if $\alpha_{i}$ is peripheral, then there is no associated twisting parameter, and instead of a pair $(\ell_{C_{i}}(X), t_{C_{i}}(X))$, we take a single parameter $\ell_{C_{i}}(X)$.

        Now we recall some elementary knowledge about measured laminations (see \cite{ALPS1, Thurston}) for the completeness of exposition.

       A \emph{geodesic lamination} $\lambda$ on a hyperbolic surface $X$ is a closed subset of $X$ that is the disjoint union of simple complete geodesics (note that the geodesic with one end or both ends transversely hitting the boundary $\partial X$ is also considered to be complete) called the \emph{leaves} of $\lambda$. By the definition, a leaf $L$ of $\lambda$ on $X\in \mathcal{T}(X_{0})$ may be a geodesic boundary component of $X$, a geodesic ending at a cusp or a boundary component of $X$ ($L$ may transversely hit a boundary component or spiral around it ), or even a geodesic with one or both of its ends never stay in any compact subset of $X$ if $X$ is a surface of infinite type.
       Note that if $L$ is a geodesic that hits $\partial X$ at a point $p\in \partial X$, we require that $L$ is perpendicular to $\partial X$ at $p$.

       Let $\lambda$ be a geodesic lamination on $X$. A \emph{transverse measure} for $\lambda$ is an assignment of a finite positive Borel measure $\mu$ on each embedded arc $k$ on $X$ (transverse to $\lambda$ and with endpoints contained in the complement of $\lambda$), such that $\mu$ satisfies the following conditions:

       (1) The support of $\mu$ is $\lambda\cap k$.

       (2) $\mu$ is invariant under homotopies relative to the leaves of $\lambda$, that is, $\mu(k)=\mu(k')$ for any two transverse arcs $k$ and $k'$ that are homotopic through embedded arcs which move their endpoints within fixed complementary components of $\mu$.

        A \emph{measured geodesic lamination} is a pair $(\lambda, \mu)$, where $\lambda$ is a geodesic lamination and $\mu$ is a transverse measure. For simplicity, we call a ``measured lamination" instead of a ``measured geodesic lamination" and sometimes denote $(\lambda, \mu)$ by $\mu$.
        Denote by $\mathcal{ML}(X)$ the space of all measured laminations on $X$.

       Let $X_{0}$ be a complete hyperbolic surface of infinite type with boundary. Let $\{\mu_{n}\}^{\infty}_{n=0}$ be a sequence of measured laminations in $\mathcal{ML}(X_{0})$. We say that $\mu_{n}$ \emph{converges} to $\mu_{0}$ in $\mathcal{ML}(X_{0})$ if $i(\mu_{n},\alpha)\rightarrow i(\mu_{0},\alpha)$ for all $\alpha\in\mathcal{A}'(X_{0})\cup\mathcal{S}(X_{0})$, where $i(\mu_{n},\alpha)$ denotes the geometric intersection of $\mu_{n}$ and $\alpha$.

       Let $\mathcal{S}^{int}(X_{0})\subset\mathcal{S}(X_{0})$ be the set of homotopy classes of essential interior simple closed curves on $X_{0}$.  Denote by $\mathcal{ML}^{\mathcal{A}'}(X_{0})$ the closure of $\mathcal{A}'(X_{0})$ in $\mathcal{ML}(X_{0})$. That is, for any $\mu\in\mathcal{ML}^{\mathcal{A}'}(X_{0})$, there exists a sequence $\{\gamma_{n}\}^{\infty}_{n=1}$   in $\mathcal{A}'(X_{0})$ with a corresponding sequence of positive weights $\{t_{n}\}^{\infty}_{n=1}$, such that $\{t_{n}\gamma_{n}\}^{\infty}_{n=1}$ converges to $\mu$ in $\mathcal{ML}(X_{0})$.
      The \emph{hyperbolic length of $\mu$} is defined to be $\mathcal{L}_{\mu}(X_{0})=\lim\limits_{n\rightarrow\infty}t_{n}\ell_{\gamma_{n}}(X_{0})$
   (see \cite{P,Thurston-notes} for more details).  In particular, $\mathcal{L}_{\mu}(X_{0})=\ell_{\mu}(X_{0})$ for all $\mu\in\mathcal{A}'(X_{0})\cup\mathcal{S}(X_{0})$. It is known that $\mathcal{L}_{\mu}(X_{0})$ is independent of the choice of the sequence which converges to it. Therefore,
   \begin{equation*}
   \log\sup_{\gamma\in \mathcal{A}'(X_{0})}\frac{\ell_{\gamma}(Y)}{\ell_{\gamma}(X)}
   =\log\sup_{\mu\in\mathcal{ML}^{\mathcal{A}'}(X_{0})}
   \frac{\mathcal{L}_{\mu}(Y)}{\mathcal{L}_{\mu}(X)}.
   \end{equation*}

    In this section, we consider the following two functions on $\mathcal{T}(X_{0})\times\mathcal{T}(X_{0})$:
    \begin{equation*}
   \begin{split}
   d(X,Y)=\log\sup_{\alpha \in \mathcal{A}(X_{0})\cup\mathcal{S}(X_{0})}\frac{\ell_{\alpha}(Y)}{\ell_{\alpha}(X)},\\
     \overline{d}(X,Y)=\log\sup_{\alpha \in \mathcal{A}(X_{0})\cup\mathcal{S}(X_{0})}\frac{\ell_{\alpha}(X)}{\ell_{\alpha}(Y)},
     \end{split}
    \end{equation*}
  for all $X, Y\in \mathcal{T}(X_{0})$.

  \begin{lemma}\emph{(Thurston \cite{Thurston1998},
  Proposition 3.5)}\label{Thurston proposition 3.5}
    For any two complete hyperbolic structures $X, Y$ on a surface $S$ of finite type without boundary, we have
    \begin{equation*}
    \sup_{\alpha\in\mathcal{S}(S)}\frac{\ell_{\alpha}(Y)}{\ell_{\alpha}(X)}
    = \sup_{\alpha\in\pi_{1}(S)-\{0\}}\frac{\ell_{\alpha}(Y)}{\ell_{\alpha}(X)},
    \end{equation*}
    where $\pi_{1}(S)-\{0\}$ is the set of homotopy classes of essential closed curves.
    \end{lemma}

    Indeed, the proof presented by Thurston \cite{Thurston1998} is independent of the topological types of hyperbolic surfaces, thus this equality holds for hyperbolic surfaces of all topological types.

\begin{lemma}\emph{(Proposition 2.7 in \cite{LPST})}\label{Boundary, double}
For any two complete hyperbolic structures $X, Y$ on a surface $S$ of finite type with boundary, we have
  \begin{equation*}
   \sup_{\gamma\in \mathcal{A}'(S)\cup\mathcal{S}(S)}\frac{\ell_{\gamma}(Y)}{\ell_{\gamma}(X)}
   =\sup_{\gamma\in\mathcal{S}(S^{d})}\frac{\ell_{\gamma}(Y^{d})}
   {\ell_{\gamma}(X^{d})},
   \end{equation*}
   where $S^{d}$ denotes the double of $S$ which carries a canonical involution such that the set of fixed points is $\partial S$. $X^{d}$, $Y^{d}$ are respectively the doubled structure of $X$, $Y$ on $S^{d}$.
\end{lemma}

    \begin{proposition}\label{Lemma:A and O}
  For any two complete hyperbolic structures $X, Y$ on a surface $S$ of finite type with boundary, we have
  \begin{equation}\label{simple and nonsimple arcs}
   \sup_{\gamma\in \mathcal{A}(S)\cup\mathcal{S}(S)}\frac{\ell_{\gamma}(Y)}{\ell_{\gamma}(X)}
   =\sup_{\gamma\in\mathcal{A}'(S)\cup\mathcal{S}(S)}\frac{\ell_{\gamma}(Y)}{\ell_{\gamma}(X)}.
   \end{equation}
  \end{proposition}

\begin{proof}
Observe that $\mathcal{A}'(S)\subset \mathcal{A}(S)$. It suffices to verify that for any non-simple essential arc $\gamma_{0}\in\mathcal{A}(S)-\mathcal{A}'(S)$, the following inequality holds.
\begin{equation}\label{boundary, simple, nonsimple arcs}
\frac{\ell_{\gamma_{0}}(Y)}{\ell_{\gamma_{0}}(X)}\leq
\sup_{\gamma\in\mathcal{A}'(S)\cup\mathcal{S}(S)}
\frac{\ell_{\gamma}(Y)}{\ell_{\gamma}(X)}.
\end{equation}
By Lemma \ref{Boundary, double}, we have
\begin{equation}\label{double closed, arc}
\sup_{\gamma\in\mathcal{A}'(S)\cup\mathcal{S}(S)}
\frac{\ell_{\gamma}(Y)}{\ell_{\gamma}(X)}
=\sup_{\gamma\in\mathcal{S}(S^{d})}
\frac{\ell_{\gamma}(Y^{d})}{\ell_{\gamma}(X^{d})}.
\end{equation}
By Lemma \ref{Thurston proposition 3.5}, it follows that
\begin{equation}\label{double simple, nonsimple}
\sup_{\gamma\in\mathcal{S}(S^{d})}
\frac{\ell_{\gamma}(Y^{d})}{\ell_{\gamma}(X^{d})}
=\sup_{\gamma\in\pi_{1}(S^{d})-\{0\}}
\frac{\ell_{\gamma}(Y^{d})}{\ell_{\gamma}(X^{d})}.
\end{equation}
Denote by $\gamma_{0}^{d}$ the double of $\gamma_{0}$ with respect to $\partial S$. Combining \eqref{double closed, arc} and \eqref{double simple, nonsimple}, we derive that
\begin{equation*}
\begin{split}
\frac{\ell_{\gamma_{0}}(Y)}{\ell_{\gamma_{0}}(X)}
=\frac{\ell_{\gamma_{0}^{d}}(Y^{d})}{\ell_{\gamma_{0}^{d}}(X^{d})}
\leq \sup_{\gamma\in\pi_{1}(S^{d})-\{0\}}
\frac{\ell_{\gamma}(Y^{d})}{\ell_{\gamma}(X^{d})}
=\sup_{\gamma\in\mathcal{A}'(S)\cup\mathcal{S}(S)}
\frac{\ell_{\gamma}(Y)}{\ell_{\gamma}(X)}.
\end{split}
\end{equation*}
This implies \eqref{boundary, simple, nonsimple arcs}.
\end{proof}

\begin{remark}
Proposition \ref{Lemma:A and O} shows that the essential simple arcs taken in the definition of the arc metric $d_{A}$ for surfaces of finite type with boundary can be replaced by essential arcs. However, the method for the proof of Lemma \ref{Boundary, double} in \cite{LPST} is not valid if $S$ is a surface of infinite type with boundary. The reason is that the set $\mathcal{ML}^{\mathcal{S}}(S^{d})$ which is the closure of $\mathcal{S}(S^{d})$ in $\mathcal{ML}(S^{d})$ is not compact and it is possible that the value
\begin{equation*}
\sup_{\gamma\in\mathcal{S}(S^{d})}
\frac{\ell_{\gamma}(Y^{d})}{\ell_{\gamma}(X^{d})}
\end{equation*}
cannot be realized by any measured lamination in $\mathcal{ML}^{\mathcal{S}}(S^{d})$. As a result, the method for the proof of Proposition \ref{Lemma:A and O} also fails for the case of surfaces of infinite type with boundary.
\end{remark}

\begin{question}
For any two complete hyperbolic structures $X, Y$ on a surface $S$ of infinite type with boundary, does the following equality still hold?
\begin{equation*}
 \sup_{\gamma\in \mathcal{A}(S)\cup\mathcal{S}(S)}
 \frac{\ell_{\gamma}(Y)}{\ell_{\gamma}(X)}
   =\sup_{\gamma\in\mathcal{A}'(S)\cup\mathcal{S}(S)}
   \frac{\ell_{\gamma}(Y)}{\ell_{\gamma}(X)}.
\end{equation*}
\end{question}

     \begin{lemma}\label{lemma:dense}
     Let $X_{0}$ be a complete hyperbolic surface with boundary which has at least one interior simple closed curve. Then for any $\alpha\in\mathcal{S}^{int}(X_{0})$, the length of $\alpha$ can be approximated by a sequence of lengths of weighted simple geodesic arcs $\gamma_{n}\in \mathcal{A}'(X_{0})$.
     \end{lemma}

  \begin{proof}
   It is equivalent to show that $\mathcal{S}^{int}(X_{0})\subset\mathcal{ML}^{\mathcal{A}'}(X_{0})$.
   Note that $X_{0}$ has at least one geodesic boundary component, then for any $\alpha\in\mathcal{S}^{int}(X_{0})$, we can find a geodesic arc $\gamma\in\mathcal{A}'(X_{0})$ that essentially intersects $\alpha$ in one or two points.
   See Figure \ref{fig:separable} (resp. Figure \ref{fig:non-separable}) for an example of $\gamma$ corresponding to a separable (resp. non-separable) interior simple closed curve $\alpha$.

\begin{figure}[ht]
\begin{tikzpicture}[scale=1.10]
\draw (1.523,0.137) .. controls (1.373,0.4415) and (2.65,0.878) .. (2.8,0.5735);
\draw (1.523,0.137) .. controls (1.6775,-0.172) and (2.95,0.2645) .. (2.8,0.5735);
\draw (1.523,0.137) .. controls (2.014,-1.045) and (2.0925,-1.7385) .. (1.388,-2.904);
\draw(2.8,0.5735) .. controls (3.1455,-0.1265) and (4.2545,-0.17) .. (4.5955,0.489);
\draw[dashed] (4.5955,0.489) .. controls (4.7,0.8435) and (5.977,0.55) .. (5.8725,0.1955);
\draw (4.5955,0.489) .. controls (4.491,0.1845) and (5.7725,-0.109) .. (5.8725,0.1955);
\draw (5.8725,0.1955) .. controls (5.527,-0.982) and (5.5915,-2.094) .. (6.1705,-2.904);
\draw (2.6785,-3.2955) .. controls (3.0285,-1.9995) and (4.4735,-1.9995) .. (4.8215,-3.2955);
\draw[red][dashed] (3.7455,0.03) .. controls (3.389,0.03) and (3.389,-2.317) .. (3.7455,-2.3235);
\draw [red](3.7455,0.03) .. controls (4.102,0.03) and (4.102,-2.317) .. (3.7455,-2.3235);
\node at (3.7905,-1.7865) {$\alpha$};
\node at (1.915,1.0495) {$\beta_{i}$};
\node at (4.6555,-0.805) {$\gamma$};
\draw (1.688,-2.352) .. controls (1.623,-2.678) and (2.8895,-2.9605) .. (2.9545,-2.678);
\draw[dashed] (1.688,-2.352) .. controls (1.753,-2.0695) and (3.0195,-2.352) .. (2.9545,-2.678);
\draw (4.5455,-2.7215) .. controls (4.676,-3.0255) and (5.99,-2.725) .. (5.8595,-2.3775);
\draw[dashed] (4.5455,-2.7215) .. controls (4.4585,-2.439) and (5.7945,-2.117) .. (5.8595,-2.3995);
\draw[dotted, ultra thick] (2.091,-2.8715) -- (2.004,-3.154);
\draw [dotted, ultra thick](5.403,-2.9605) -- (5.555,-3.1995);
\draw (4.591,0.5) .. controls (4.682,0.727) and (4.7275,0.954) .. (4.6365,1.1355);
\draw (5.8635,0.182) .. controls (5.9545,0.4545) and (6,0.636) .. (6.2725,0.7725);
\draw[red] (2.364,0.1365) .. controls (2.5,-0.0905) and (5.6825,-2.2735) .. (5.7725,-2.046);
\draw [red][dashed](2.3185,0.6365) .. controls (2.5,0.4095) and (5.865,-1.91) .. (5.7725,-2.0455);
\draw [red](2.2275,0.091) -- (2.318,-0.045) -- (2.454,0.0455);
\draw[red] (2.182,0.591) -- (2.273,0.4545) -- (2.409,0.5455);
\draw [dotted, ultra thick] (5.5455,1) -- (5.409,0.682);
\end{tikzpicture}
\caption{\small{An example of $\gamma$ corresponding to $\alpha$ (separable), where the complement of $\alpha$ in $X_{0}$ is disconnected.}}
            \label{fig:separable}
		\end{figure}
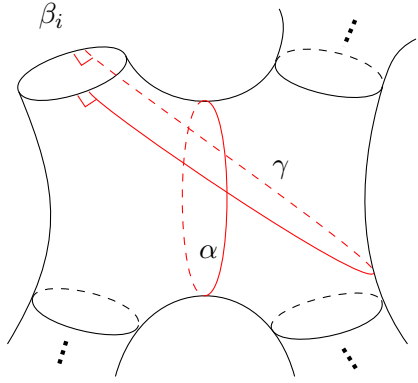

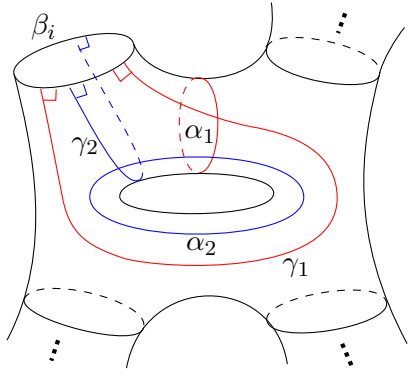
\begin{figure}[ht]
\begin{tikzpicture}[scale=0.9]
\draw (1.0556,0.1056) .. controls (0.9556,0.4556) and (2.7026,0.9582) .. (2.8026,0.5526);
\draw (1.0556,0.1056) .. controls (1.2,-0.4) and (2.9526,0.1526) .. (2.8026,0.5526);
\draw (1.0556,0.1056) .. controls (1.3276,-0.9776) and (1.6526,-2.4846) .. (0.9526,-3.8164);
\draw (2.8026,0.5526) .. controls (3.0526,-0.0974) and (4.4526,-0.15) .. (4.7526,0.5);
\draw[dashed] (4.7526,0.5) .. controls (4.85,0.85) and (6.4974,0.5086) .. (6.3974,0.103);
\draw (4.7526,0.5) .. controls (4.6026,0.0474) and (6.3526,-0.2444) .. (6.3974,0.103);
\draw (6.3974,0.103) .. controls (6.1082,-1.0026) and (5.95,-2.2664) .. (6.8,-3.8664);
\draw (2.6,-1.6992) .. controls (2.6,-1.2548) and (4.85,-1.2548) .. (4.85,-1.6492);
\draw[red][dashed] (3.7526,0.0526) .. controls (3.3526,0.0526) and (3.4026,-1.3548) .. (3.7526,-1.3492);
\draw [red](3.7526,0.0526) .. controls (4.1526,0.0526) and (4.1026,-1.3548) .. (3.7526,-1.3492);
\node at (3.774,-0.7812) {$\alpha_{1}$};
\node at (1.4922,0.7974) {$\beta_{i}$};
\node at (5.1848,-2.7502) {$\gamma_{1}$};
\draw (2.6,-1.6992) .. controls (2.6,-2.0436) and (4.85,-2.0436) .. (4.85,-1.6492);
\draw (2.6996,-4.2272) .. controls (2.9526,-2.8164) and (4.85,-2.819) .. (5.05,-4.169);
\draw[blue] (2.163,-1.7022) .. controls (2.163,-0.92) and (5.2896,-0.9174) .. (5.2896,-1.7078);
\draw[blue] (2.163,-1.7022) .. controls (2.163,-2.4362) and (5.2896,-2.4336) .. (5.2896,-1.7078);
\draw[blue] (1.8702,-0.097) .. controls (2.1284,-0.5604) and (2.6384,-1.6022) .. (2.8918,-1.4518);
\draw[blue][dashed] (2.8918,-1.4518) .. controls (3.1422,-1.3044) and (2.1598,0.0948) .. (1.9572,0.606);
\draw [blue](2.0706,-0.044) -- (2.1206,-0.197) -- (1.9702,-0.2444);
\draw [blue](2.1046,0.6586) -- (2.1546,0.5086) -- (2.0046,0.4586);
\node at (3.8012,-2.433) {$\alpha_{2}$};
\node at (2.099,-0.9596) {$\gamma_{2}$};
\draw (1.2026,-3.2164) .. controls (1.1496,-3.5164) and (2.8996,-3.9746) .. (2.9496,-3.6246);
\draw[dashed] (1.2026,-3.2164) .. controls (1.3552,-2.8164) and (3.0496,-3.2246) .. (2.9496,-3.6246);
\draw (4.75,-3.519) .. controls (4.8,-3.919) and (6.55,-3.6164) .. (6.5,-3.2664);
\draw [dashed](4.75,-3.519) .. controls (4.7,-3.169) and (6.45,-2.8664) .. (6.5,-3.2664);
\draw[dotted,ultra thick] (1.7,-3.8164) -- (1.6,-4.1164);
\draw[dotted,ultra thick] (5.85,-3.869) -- (6,-4.219);
\draw[red] (2.503,0.1608) -- (2.6112,-0.003) -- (2.784,0.1402);
\draw[red] (1.672,-0.1206) -- (1.666,-0.2814) -- (1.4786,-0.2784);
\draw (4.737,0.5) .. controls (4.8422,0.6844) and (4.8422,0.974) .. (4.8156,1.1584);
\draw (6.3948,0.1082) .. controls (6.4474,0.3192) and (6.6318,0.6348) .. (6.7636,0.7926);
\draw[dotted,ultra thick] (5.8948,1) -- (5.7896,0.7104);
\draw [red](2.6668,0.278) .. controls (2.8892,-0.1668) and (3.9436,-0.5552) .. (4.4992,-0.6656);
\draw [red](1.4444,-0.1112) .. controls (1.556,-0.6104) and (1.666,-1.2208) .. (1.7764,-1.7764);
\draw [red](4.5,-0.6668) .. controls (5.1668,-0.778) and (5.778,-1.056) .. (5.778,-1.7224);
\draw [red](2.6676,-2.6116) .. controls (3.2788,-2.778) and (5.7792,-2.8888) .. (5.7776,-1.7224);
\draw[red] (1.7776,-1.778) .. controls (1.8888,-2.278) and (2.1664,-2.444) .. (2.666,-2.6108);
\end{tikzpicture}
\caption{\small{An example of $\gamma_{i}$ corresponding to $\alpha_{i}$ (non-separable).}}
            \label{fig:non-separable}
		\end{figure}

    Let $\gamma_{n}$ be the weighted geodesic arc obtained by taking a power n of a positive Dehn-twist along $\alpha$ with the weight $1/(i(\gamma, \alpha)\,n)$ on $\gamma$. It is obvious that  $\{\frac{1}{i(\gamma, \alpha)\,n}\gamma_{n}\}^{\infty}_{n=1}$ converges to $\alpha$ in $\mathcal{ML}(X_{0})$, hence $\ell_{\alpha}(X_{0})=
    \lim\limits_{n\rightarrow\infty}\frac{1}{i(\gamma, \alpha)\, n}\ell_{\gamma_{n}}(X_{0})$
    and we obtain that  $\mathcal{S}^{int}(X_{0})\subset\mathcal{ML}^{\mathcal{A}'}(X_{0})$.
   \end{proof}




   \begin{theorem}\label{Thurston metric}
   Let $X_{0}$ be a complete hyperbolic surface of infinite type with boundary which satisfies the geometric condition $(\star)$. Then the two functions $d$ and $\overline{d}$ on $\mathcal{T}(X_{0})\times \mathcal{T}(X_{0})$ are asymmetric metrics.
   \end{theorem}

    \begin{proof}

    By Lemma \ref{Wolpert's inequality in our case}, $d$ and $\bar{d}$ are valued in $[0,+\infty)$. Observe that $\overline{d}(X,Y)=d(Y,X)$, for all $X,Y\in\mathcal{T}(X_{0})$, it suffices to consider $d$. Note that the triangle inequality naturally holds for $d$. Now we prove the separation axiom for $d$ by showing that if $X\not = Y\in \mathcal{T}(X_{0})$, then $d(X,Y)>0$.

    Assume that $d(X,Y)\leq 0$, then $\ell_{\alpha}(Y)\leq \ell_{\alpha}(X)$
   for all $\alpha\in \mathcal{S}(X_{0})\cup \mathcal{A}(X_{0})$. In particular, we have that
   \begin{equation}\label{eq:B(X_{0})}
   \ell_{\beta_{j}}(Y)\leq\ell_{\beta_{j}}(X),
   \end{equation}
   for all $\beta_{j}\in\mathcal{B}(X_{0})$.

   Moreover, for each $j\in\mathbb{N}$, we have that
   \begin{equation}\label{eq:d(j,i)}
   d^{j}_{i}(Y)\leq d^{j}_{i}(X),
   \end{equation}
   for all $i\in\mathbb{N}$.
  Here $\{d_{i}^{j}(X)\}^{\infty}_{i=1}$ denotes the 0-th orthogonal spectrum of $X$ related to $\beta_{j}(X)$ and $\mathcal{B}(X)$, and
   $\{d_{i}^{j}(Y)\}^{\infty}_{i=1}$ denotes the 0-th orthogonal spectrum of $Y$ related to $\beta_{j}(Y)$ and $\mathcal{B}(Y)$, where $\beta_{j}(X)$ (resp. $\beta_{j}(Y)$) is the corresponding geodesic boundary component on $X$ (resp. $Y$).

   By Corollary \ref{Basmajina and McShane} and the Basmajian identity \eqref{Basmajian identity}, it follows that for each $j\in\mathbb{N}$ and each $X'\in\mathcal{T}(X_0)$,
   \begin{equation}\label{eq:Basmajian application}
   \ell_{\beta_{j}}(X')=\sum_{i=1}^{\infty} 2\log \coth(\frac{d_{i}^{j}(X')}{2}).
   \end{equation}

   By \eqref{eq:d(j,i)}, \eqref{eq:Basmajian application} and the monotonically decreasing of the function $\log\coth({\frac{x}{2}})$,
    \begin{equation}\label{eq:geq}
       \ell_{\beta_{j}}(Y)=\sum_{i=1}^{\infty} 2\log \coth(\frac{d_{i}^{j}(Y)}{2}) \geq
       \ell_{\beta_{j}}(X)=\sum_{i=1}^{\infty} 2\log \coth(\frac{d_{i}^{j}(X)}{2}).
   \end{equation}

      Since the geodesic arc which minimizes the lengths of all the arcs in a given homotopy class is unique and hits the boundary perpendicularly, then there is a bijection between $\mathcal{A}(X_{0})$ and the set of geodesic arcs (possibly with self-intersections) in $X_{0}$ which are orthogonal to $\partial X_{0}$ at their endpoints. Therefore,
      \begin{equation}\label{arc and homotopy}
      \begin{split}
   \cup_{j}\{d_{i}^{j}(X)\}=\{\ell_{\alpha}(X): \alpha\in\mathcal{A}(X_{0})\}, \\
    \cup_{j}\{d_{i}^{j}(Y)\}=\{\ell_{\alpha}(Y): \alpha\in\mathcal{A}(X_{0})\}.
    \end{split}
       \end{equation}
    Combining \eqref{eq:B(X_{0})} \eqref{eq:geq} and \eqref{arc and homotopy}, we have
    \begin{equation}\label{eq:equal}
    \ell_{\alpha}(Y)=\ell_{\alpha}(X),
    \end{equation}
     for all $\alpha\in \mathcal{B}(X_{0})\cup \mathcal{A}(X_{0})$.

     Note that $X_{0}$ admits a countable pair of pants decomposition $\mathcal{P}=\{C_{i}\}^{\infty}_{i=1}$, then $X_{0}$ can be parameterized by the Fenchel-Nielsen coordinates with respect to $\mathcal{P}$ (see \cite{ALPSS}). Moreover, for each interior simple closed geodesic $C_{i}$ in $\mathcal{P}$, the twisting parameter can be uniquely determined by the length of the shortest simple closed geodesic $\gamma_{i}$ which intersects $C_{i}$ and the length of the geodesic $T_{C_{i}}(\gamma_{i})$ obtained by taking a positive Dehn-twist along $C_{i}$ on $\gamma_{i}$.

     By Lemma \ref{lemma:dense}, the length of an interior simple closed geodesic can be approximated by a sequence of lengths of weighted geodesic arcs $\gamma_{n}\in \mathcal{A}(X_{0})$. From this and \eqref{eq:equal}, we have $\ell_{\alpha}(Y)=\ell_{\alpha}(X)$ for all $\alpha\in\mathcal{S}(X_{0})\cup\mathcal{A}(X_{0})$. Then $X$ and $Y$ have the same Fenchel-Nielsen coordinates and hence $X=Y$, which implies the assumption is false.

     The asymmetric condition of $d$ can be deduced from the example constructed by Thurston (see \cite{Thurston1998}). Let $X$ be a complete hyperbolic surface of infinite type which satisfies the geometric condition $(\star)$ and the following conditions:

     (1) $X$ contains an embedded hyperbolic X-piece $S$ (that is, a hyperbolic surface whose interior is homeomorphic to a sphere with four disjoint closed disks removed) with four geodesic boundary components $\beta_{1}, \beta_{2}, \beta_{3}, \beta_{4}$ of the same length $l$ satisfying $\sinh{\frac{l}{2}}=1$.

     (2) Let $\gamma_{1}$ (resp. $\gamma_{2}$) be the shortest geodesic arc connecting $\beta_{1}$ and $\beta_{2}$ (resp. $\beta_{2}$ and $\beta_{3}$). Denote by $\alpha_{1}$ (resp. $\alpha_{2}$) the third boundary component of the hyperbolic pair of pants determined by $\beta_{1}$, $\beta_{2}$ and $\gamma_{1}$ (resp. $\beta_{2}$, $\beta_{3}$ and $\gamma_{2}$). We choose $X$ such that $\ell_{\alpha_{1}}(X)$ is sufficiently small and the twisting parameter of $\alpha_{1}$ is zero, as indicated in Figure \ref{fig:X}.

     \begin{figure}[ht]
 \begin{tikzpicture}[scale=0.9]
\draw (-2.25,-0.7) .. controls (-2.4,-0.35) and (-1.35,-0.05) .. (-1.2,-0.35);
\draw (-2.25,-0.7) .. controls (-2.1,-1.05) and (-1.05,-0.7) .. (-1.2,-0.35);
\draw[red] (-2.25,-0.7) .. controls (-1.8,-1.5) and (-1.8,-1.8) .. (-2.2,-2.3);
\draw (-2.2,-2.3) .. controls (-2.4,-2.55) and (-1.45,-2.95) .. (-1.25,-2.65);
\draw [dashed] (-2.2,-2.3) .. controls (-1.95,-1.95) and (-1.05,-2.4) .. (-1.25,-2.65);
\draw [red](-1.2,-0.35) .. controls (-0.4,-1.6) and (5.9,-1.6) .. (6.5,-0.3);
\draw (6.5,-0.3);
\draw (6.5,-0.3) .. controls (6.75,0) and (7.6,-0.45) .. (7.45,-0.7);
\draw (6.5,-0.3) .. controls (6.3,-0.65) and (7.25,-1.05) .. (7.45,-0.7);
\draw(7.45,-0.7) .. controls (7.05,-1.45) and (7.1,-1.65) .. (7.5,-2.05);
\draw (-1.25,-2.65) .. controls (-0.4,-1.6) and (5.9,-1.55) .. (6.45,-2.5);
\draw (6.45,-2.5) .. controls (6.6,-2.8) and (7.65,-2.4) .. (7.5,-2.05);
\draw[dashed] (6.45,-2.5) .. controls (6.35,-2.2) and (7.35,-1.7) .. (7.5,-2.05);
\draw (6.45,-2.5) .. controls (6.95,-2.95) and (7.1,-3.25) .. (7,-3.6);
\draw (7.5,-2.05) .. controls (8,-2.7) and (8.7,-2.7) .. (9.35,-2.3);
\draw[dotted,  ultra thick] (7.65,-3.05) -- (7.9,-3.3);
\draw[blue] (2.5556,-1.3) .. controls (2.4668,-1.3) and (2.4668,-1.85) .. (2.5556,-1.85);
\draw[blue][dashed] (2.5556,-1.3) .. controls (2.6444,-1.3) and (2.6444,-1.85) .. (2.5556,-1.85);
\node at (-2,-3.1108) {$\beta_{1}$};
\node at (-1.8888,0.1108) {$\beta_{2}$};
\node at (7.278,0.222) {$\beta_{3}$};
\node at (6.0556,-2.8388) {$\beta_{4}$};
\node at (2.6112,-2.218) {$\alpha_{1}$};
\node at (8.4936,-1.4988) {$X$};
\draw[blue] (-1.8888,-1.5528) .. controls (-1.8888,-1.7192) and (7.1668,-1.7188) .. (7.1668,-1.5524);
\draw[blue][dashed] (-1.8888,-1.5528) .. controls (-1.8888,-1.3864) and (7.1668,-1.3864) .. (7.1668,-1.5528);
\node at (-0.2224,-1.8888) {$\alpha_{2}$};
\node at (0.9444,-0.778) {$\gamma_{2}$};
\node at (-2.3332,-1.2776) {$\gamma_{1}$};
\draw[red] (-1.222,-0.5556) -- (-1.1108,-0.6664) -- (-1.0556,-0.4996);
\draw [red](6.3336,-0.5) -- (6.3888,-0.6664) -- (6.5,-0.5556);
\draw[red] (-2,-0.8332) -- (-1.9448,-1) -- (-2.0556,-0.9996);
\draw[red](-2.1112,-2.278) -- (-2.056,-2.3332) -- (-2.1668,-2.4444);
\end{tikzpicture}
\caption{\small{A hyperbolic surface $X$ of infinite type which satisfies the geometric condition $(\star)$ and the conditions (1) and (2).}}
            \label{fig:X}
		\end{figure}
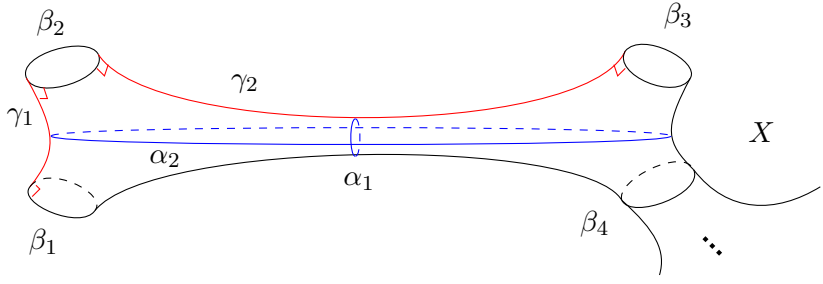

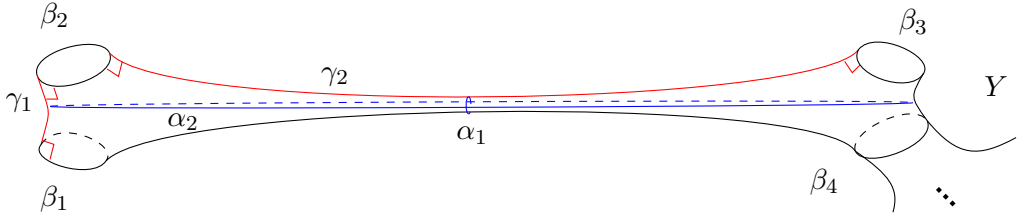
\begin{figure}[ht]
    \begin{tikzpicture}[scale=0.9]
\draw (-3.85,-1) .. controls (-4,-0.65) and (-2.95,-0.35) .. (-2.8,-0.65);
\draw (-3.85,-1) .. controls (-3.7,-1.35) and (-2.65,-1) .. (-2.8,-0.65);
\draw[red] (-3.85,-1) .. controls (-3.65,-1.6) and (-3.65,-1.45) .. (-3.8,-2.0056);
\draw (-3.8,-2.0056) .. controls (-4,-2.2556) and (-3.05,-2.5444) .. (-2.85,-2.2444);
\draw [dashed] (-3.8,-2.0056) .. controls (-3.55,-1.6) and (-2.65,-1.9944) .. (-2.85,-2.2444);
\draw[red](-2.8,-0.65) .. controls (-2,-1.6) and (7.55,-1.45) .. (8.15,-0.6);
\draw (8.15,-0.6);
\draw (8.15,-0.6) .. controls (8.4,-0.3) and (9.25,-0.7) .. (9.1,-0.95);
\draw (8.15,-0.6) .. controls (7.95,-0.95) and (8.9,-1.3) .. (9.1,-0.95);
\draw (9.1,-0.95) .. controls (8.9,-1.3) and (8.9,-1.3) .. (9.15,-1.65);
\draw (-2.85,-2.2444) .. controls (-2,-1.35) and (7.55,-1.25) .. (8.1,-2.1);
\draw (8.1,-2.1) .. controls (8.25,-2.4) and (9.3,-2) .. (9.15,-1.65);
\draw[dashed] (8.1,-2.1) .. controls (8,-1.8) and (8.9444,-1.35) .. (9.15,-1.65);
\draw (8.1,-2.1) .. controls (8.6,-2.45) and (8.75,-2.65) .. (8.65,-3);
\draw (9.15,-1.65) .. controls (9.6,-2.2) and (9.9,-2.2) .. (10.45,-1.9);
\draw[dotted,  ultra thick] (9.3,-2.65) -- (9.55,-2.9);
\draw[blue] (2.45,-1.3) .. controls (2.4,-1.3) and (2.4,-1.55) .. (2.45,-1.55);
\draw[blue][dashed] (2.45,-1.3) .. controls (2.4944,-1.3) and (2.4944,-1.55) .. (2.45,-1.55);
\node at (-3.6,-2.8) {$\beta_{1}$};
\node at (-3.6,-0.1) {$\beta_{2}$};
\node at (8.95,-0.2) {$\beta_{3}$};
\node at (7.65,-2.55) {$\beta_{4}$};
\node at (2.5,-1.8676) {$\alpha_{1}$};
\node at (10.1616,-1.1488) {$Y$};
\draw[blue][dashed] (-3.6668,-1.4412) .. controls (-3.6668,-1.386) and (8.9448,-1.3304) .. (8.9448,-1.3856);
\draw [blue](-3.6668,-1.4412) .. controls (-3.6668,-1.4964) and (8.9448,-1.4408) .. (8.9448,-1.3856);
\node at (-1.7224,-1.6668) {$\alpha_{2}$};
\node at (0.5,-1) {$\gamma_{2}$};
\node at (-4.1112,-1.3888) {$\gamma_{1}$};
\draw[red] (-2.8332,-0.8888) -- (-2.6668,-1) -- (-2.6116,-0.778);
\draw [red](-3.778,-1.9444) -- (-3.6112,-2) -- (-3.6668,-2.222);
\draw [red](-3.6112,-1.1668) -- (-3.5556,-1.3332) -- (-3.7216,-1.3332);
\draw [red](7.9448,-0.7776) -- (8.0556,-0.9444) -- (8.1664,-0.8332);
\end{tikzpicture}
\caption{\small{A deformed surface $Y$ of $X$.}}
            \label{fig:Y}
		\end{figure}

     Note that $X$ satisfies the geometric condition $(\star)$. And such a surface $X$ always exists. Now we deform $X$  by contracting the length of $\alpha_{1}$ appropriately on $X$, while keeping the lengths of $\beta_{i}$ ($i=1,2,3,4$) and the hyperbolic structure of the complement of $S$ in $X$ unchanged. Denote by $Y$ the obtained surface, as presented in Figure \ref{fig:Y}. It is not hard to see that $d(X,Y)\not= d(Y,X)$ if we can choose $Y$ with the contracted length $\ell_{\alpha_{1}}(Y)$ satisfying
     \begin{equation}\label{properties for Y}
     \frac{\ell_{\gamma_{2}}(Y)}{\ell_{\gamma_{2}}(X)}
     \leq\frac{\ell_{\alpha_{2}}(Y)}{\ell_{\alpha_{2}}(X)},
   \hspace{1cm}
     \frac{\ell_{\gamma_{1}}(X)}{\ell_{\gamma_{1}}(Y)}
     \leq\frac{\ell_{\alpha_{1}}(X)}{\ell_{\alpha_{1}}(Y)}.
     \end{equation}
     Indeed, if $Y$ satisfies the property \eqref{properties for Y}, we have
     \begin{equation*}
     \begin{split}
     d(X,Y)&=\log\sup_{\alpha\in\mathcal{S}(X_{0})\cup\mathcal{A}(X_{0})}
     \frac{\ell_{\alpha}(Y)}{\ell_{\alpha}(X)}
     =\frac{\ell_{\alpha_{2}}(Y)}{\ell_{\alpha_{2}}(X)},\\
       d(Y,X)&=\log\sup_{\alpha\in\mathcal{S}(X_{0})\cup\mathcal{A}(X_{0})}
     \frac{\ell_{\alpha}(X)}{\ell_{\alpha}(Y)}
     =\frac{\ell_{\alpha_{1}}(X)}{\ell_{\alpha_{1}}(Y)}
     \approx e^{\frac{1}{2}(\ell_{\alpha_{2}}(Y)-\ell_{\alpha_{2}}(X))}.
     \end{split}
    \end{equation*}
   Choosing $\ell_{\alpha_{2}}(Y)$ appropriately large (equivalently, contracting the length of $\alpha_{1}$ appropriately on $X$), we have $d(X,Y)\not= d(Y,X)$.

   Now we show that such a surface $Y$ with the property \eqref{properties for Y} always exists. For simplicity, still denote the length of $\alpha_{i}$ (resp. $\beta_{i}$, $\gamma_{i}$) by $\alpha_{i}$ (resp. $\beta_{i}$, $\gamma_{i}$) for $i=1,2$. By the formulae for right-angled pentagons and right-angled hexagons respectively (see \cite{Buser}) and the assumption that $\sinh{\frac{\beta_{i}}{2}}=\sinh{\frac{l}{2}}=1$ for $i=1,2,3,4$, we have
   \begin{equation}\label{eq:d(X,Y)}
   \begin{split}
 \cosh{\frac{\alpha_{2}}{4}}
 &=\sinh{\frac{\beta_{2}}{2}}\sinh{\frac{\gamma_{2}}{2}}
  =\sinh{\frac{\gamma_{2}}{2}},\\
 \cosh{\gamma_{1}}
 &=\frac{\cosh{\frac{\alpha_{1}}{2}}+
 \cosh{\frac{\beta_{1}}{2}}\cosh{\frac{\beta_{2}}{2}}}
 {\sinh{\frac{\beta_{1}}{2}}\sinh{\frac{\beta_{2}}{2}}}
 =\cosh{\frac{\alpha_{1}}{2}}+2.
 \end{split}
   \end{equation}

 By \eqref{eq:d(X,Y)} and the growth trends of the two functions $y=\cosh{x}$ and $y=\sinh{x}$ (resp. $y=\cosh{x}$ and $y=\cosh{x}+2$), as presented in Figure \ref{fig:f1} (resp. Figure \ref{fig:f2}), we can always find such a deformed surface $Y$ with the property \eqref{properties for Y}.

   \begin{figure}
     \begin{minipage}[t]{0.55\linewidth} 
    \centering
      \begin{tikzpicture}[scale=0.88]
    \draw[eaxis] (-1.5,0) -- (2.5,0) node[below] {$x$};
    \draw[eaxis] (0,-0.9) -- (0,4.7) node[above] {$y$};
    \draw[elegant,blue,domain=-0.01:2.4] plot(\x,{cosh{\x}});
    \draw[elegant,orange,domain=-0.01:2.4] plot(\x,{sinh{\x}});
     \end{tikzpicture}
     \caption{\footnotesize{Two functions: $y=\cosh{x}$ (above) and $y=\sinh{x}$ (below), where $x\geq 0$.}}
      \label{fig:f1}
        \end{minipage}%
  \begin{minipage}[t]{0.56\linewidth}
    \centering
   \begin{tikzpicture}[scale=0.75]
    \draw[eaxis] (-1.5,0) -- (2.5,0) node[below] {$x$};
    \draw[eaxis] (0,-1) -- (0,5.5) node[above] {$y$};
    \draw[elegant,orange,domain=-0.01:2.4] plot(\x,{cosh{\x}});
    \draw[elegant,blue,domain=-0.01:2.4] plot(\x,{cosh{\x}+2});
      \end{tikzpicture}
      \caption{\footnotesize{Two functions: $y=\cosh{x}+2$ (above) and $y=\cosh{x}$ (below), where $x\geq 0$.}}
            \label{fig:f2}
  \end{minipage}
\end{figure}
    \end{proof}

\begin{remark}
    By Proposition \ref{Lemma:A and O}, the asymmetric metric $d$ can be viewed as an analogue, for surfaces of infinite type with boundary, of the arc metric defined for surfaces of finite type with boundary. That's why we also call $d$ the arc metric.

    The problem that if the function $d$ in Theorem \ref{Thurston metric} 
    is positive definite can be viewed as a particular version of the marked length spectrum rigidity problem (see e.g. \cite{Croke, DLK}). In general, let $(M,g)$ be a Riemannian manifold and let $\Sigma$ be a set of homotopy classes of the curves on $M$ one wants to consider. The \emph{$\Sigma$-marked length spectrum} of $(M,g)$ is the length vector $(\ell_{\gamma}(g))_{\gamma\in\Sigma}$ indexed over $\Sigma$, where $\ell_{\gamma}(g)$ is the infinimum of the lengths under the metric $g$ of all the curves in the homotopy class $[\gamma]\in\Sigma$. The \emph{marked length spectrum rigidity problem} asks whether an inequality between the marked length spectra of two Riemannian manifolds implies an isometry homotopic to the identity between them.


    In our case, the rigidity problem is the marked $\mathcal{S}(X_0)\cup\mathcal{A}(X_0)$-spectrum rigidity problem in the special case of complete hyperbolic surfaces of infinite type with geodesic boundary.

 It is necessary to take arcs into consideration in the definition of $d$, since for any complete hyperbolic surface $X_{0}$ of infinite type with geodesic boundary components whose lengths are uniformly bounded above, we can find two distinct elements $X,Y$ in  $\mathcal{T}(X_{0})$ such that $\ell_{\alpha}(Y)\leq\ell_{\alpha}(X)$ for all $\alpha\in\mathcal{S}(X_{0})$. To see this, let $X=X_{0}$. Denote by $\bar{X}$ the Riemann surface such that its convex core is exactly the hyperbolic surface $X$. Let $\bar{Y}$ be the Nielsen extension of $\bar{X}$. Note that there exists a quasiconformal homeomorphism from $\bar{X}$ to $\bar{Y}$. Then we obtain another hyperbolic surface $Y\in\mathcal{T}(X_{0})$ which is the convex core of $\bar{Y}$. By generalized Schwarz lemma, we have $\ell_{\alpha}(Y)\leq\ell_{\alpha}(X)$ for all $\alpha\in\mathcal{S}(X_{0})$. This implies that
\begin{equation*}
\log\sup_{\alpha\in\mathcal{S}(X_{0})}
\frac{\ell_{\alpha}(Y)}{\ell_{\alpha}(X)}\leq 0.
\end{equation*}
\end{remark}

\begin{remark}
 Let $X_{0}$ be a complete hyperbolic surface of infinite type with geodesic boundary. Recall that the set of boundary components of $X_{0}$ is $\mathcal{B}(X_{0})=\{\beta_{1}, \beta_{2}, ... , \beta_{k}, ... \}$. Let $L=(L_{\alpha})_{\alpha \in \mathcal{B}(X_{0})} \in \mathbb{R}^{|\mathcal{B}(X_{0})|}_{>0}$, where $|\mathcal{B}(X_{0})|$ denotes the number of the elements in $\mathcal{B}(X_{0})$ which is finite or countably infinite. Denote by $\mathcal{T}(X_{0},L)$ the subspace of $\mathcal{T}(X_{0})$ which consists of the equivalence classes of marked hyperbolic surfaces with geodesic boundary components of fixed lengths, that is, the geodesic length $\ell_{\beta_{i}}(X)$ of $\beta_{i}$ under each element $X$ of $\mathcal{T}(X_{0},L)$ is $L_{\beta_{i}}$ for each $i\in \mathbb{N}$.  For convenience, we denote $L_{\beta_{i}}$ by $L_{i}$.

 If $X_{0}$ satisfies the geometric condition $(\star)$, as discussed in Theorem \ref{Thurston metric}, by applying the generalized McShane identity \eqref{McShane identity} and the Basmajian identity \eqref{Basmajian identity},
 the following two functions are asymmetric metrics on $\mathcal{T}(X_{0},L)$.
   \begin{equation*}
   d_{1}(X,Y)=\log\sup_{\alpha \in \mathcal{S}(X_{0})}\frac{\ell_{\alpha}(Y)}{\ell_{\alpha}(X)},
   \end{equation*}
   \begin{equation*}
   d_{2}(X,Y)= \log\sup_{\alpha\in \mathcal{A}(X_{0})}\frac{\ell_{\alpha}(Y)}{\ell_{\alpha}(X)},
   \end{equation*}
   for all $X, Y\in \mathcal{T}(X_{0},L)$.
\end{remark}

The following theorem shows that one can obtain the same asymmetric metric by taking the supremum over $\mathcal{A}(X_{0})\cup\mathcal{B}(X_{0})$ instead of $\mathcal{A}(X_{0})\cup\mathcal{S}(X_{0})$ in the formula which defines the arc metric $d$ on $\mathcal{T}(X_{0})$ in Theorem \ref{Thurston metric}.

    \begin{theorem}\label{asymmetric arc metrics}
    Let $X_{0}$ be a complete hyperbolic surface of infinite type with boundary, then the following equality still holds for all $X, Y\in \mathcal{T}(X_{0})$.
     \begin{equation}\label{eq:arc and boundary}
  \sup_{\alpha \in \mathcal{A}(X_{0})\cup\mathcal{B}(X_{0})}\frac{\ell_{\alpha}(Y)}{\ell_{\alpha}(X)}
  =\sup_{\gamma \in \mathcal{A}(X_{0})\cup\mathcal{S}( X_{0})}\frac{\ell_{\gamma}(Y)}{\ell_{\gamma}(X)}.
   \end{equation}
  In particular, if $X_{0}$ satisfies the geometric condition $(\star)$, then the following equality defines the same asymmetric metric on $\mathcal{T}(X_{0})$.
   \begin{equation}\label{AB and AS}
   \log\sup_{\alpha \in \mathcal{A}(X_{0})\cup\mathcal{B}(X_{0})}\frac{\ell_{\alpha}(Y)}{\ell_{\alpha}(X)}
    =\log\sup_{\gamma \in \mathcal{A}(X_{0})\cup\mathcal{S}( X_{0})}\frac{\ell_{\gamma}(Y)}{\ell_{\gamma}(X)}.
   \end{equation}
    \end{theorem}

     \begin{proof}
   Obviously,
      \begin{equation*}
   \sup_{\alpha \in \mathcal{A}(X_{0})\cup\mathcal{B}(X_{0})}
   \frac{\ell_{\alpha}(Y)}{\ell_{\alpha}(X)}
   \leq\sup_{\gamma \in \mathcal{A}(X_{0})\cup\mathcal{S}( X_{0})}\frac{\ell_{\gamma}(Y)}{\ell_{\gamma}(X)}.
   \end{equation*}

   It suffices to verify that
   \begin{equation}\label{ineq:leq}
   \sup_{\alpha\in\mathcal{A}(X_{0})\cup\mathcal{S}(X_{0})}
   \frac{\ell_{\gamma}(Y)}{\ell_{\gamma}(X)}
   \leq \sup_{\gamma\in \mathcal{A}(X_{0})\cup\mathcal{B}(X_{0})}\frac{\ell_{\gamma}(Y)}{\ell_{\gamma}(X)}.
   \end{equation}

     By Lemma \ref{lemma:dense},    we have
  \begin{equation*}
  \mathcal{S}^{int}(X_{0})\subset\mathcal{ML}^{\mathcal{A}}(X_{0}).
  \end{equation*}

 Observe that
 $\mathcal{S}(X_{0})=\mathcal{S}^{int}(X_{0})\cup\mathcal{B}(X_{0})$ and
     \begin{equation*}
   \sup_{\gamma\in \mathcal{A}(X_{0})}\frac{\ell_{\gamma}(Y)}{\ell_{\gamma}(X)}
   =\sup_{\mu\in\mathcal{ML}^{\mathcal{A}}(X_{0})}
   \frac{\mathcal{L}_{\mu}(Y)}{\mathcal{L}_{\mu}(X)}.
   \end{equation*}

   Therefore, the inequality \eqref{ineq:leq} holds. If $X_{0}$ satisfies the geometric condition $(\star)$, it follows from Theorem \ref{Thurston metric} that the equality \eqref{AB and AS} defines the same asymmetric metric. This completes the proof of this theorem.
    \end{proof}

 \begin{remark}
     Using the same method for the proof of the equality
     \eqref{eq:arc and boundary} in Theorem \ref{asymmetric arc metrics}, we give an affirmative answer to the following question (see Problem 5.5 in \cite{LSZ}): does the equality \eqref{question1} hold on Teichm\"uller spaces of surfaces of infinite type with boundary?
    \begin{equation}\label{question1}
    \log\sup_{\alpha \in \mathcal{A}'(X_{0})\cup \mathcal{B}(X_{0}) }\frac{\ell_{\alpha}(Y)}{\ell_{\alpha}(X)}
   = \log\sup_{\gamma \in \mathcal{A}'(X_{0})\cup\mathcal{S}(X_{0})}\frac{\ell_{\gamma}(Y)}{\ell_{\gamma}(X)}.
   \end{equation}
\end{remark}

    \begin{question}
    For a complete hyperbolic surface $X_{0}$ of infinite type with geodesic boundary, is the following function $\delta$ an asymmetric metric on $\mathcal{T}(X_{0})$? 
     \begin{equation*}
    \delta(X,Y)=\log \sup\limits_{\gamma\in\mathcal{A}(X_{0})}\frac{\ell_{\gamma}(Y)}{\ell_{\gamma}(X)},
    \end{equation*}
    for any $X,Y\in\mathcal{T}(X_{0})$.
    \end{question}

    It is clear that $\delta$ is positive definite for hyperbolic surfaces of finite type with boundary, by an application of the Bridgeman identity\cite{Bridgeman2}.

   \begin{question}
   Let $X_{0}$ be a complete hyperbolic surface $X_{0}$ of infinite type with geodesic boundary, are there two elements $X,Y$ in $\mathcal{T}(X_{0})$ satisfying the following condition?
   \begin{equation*}
   \log\sup_{\alpha\in\mathcal{S}(X_{0})}\frac{\ell_{\alpha}(Y)}{\ell_{\alpha}(X)}<0.
   \end{equation*}
   \end{question}

   For a complete hyperbolic surface of finite type with geodesic boundary, it is true. We can construct some examples by Nielsen extension (see \cite{Bers, Cao}) or strip deformation (see \cite{PG2010, Kassel}).
   Does it still work for the case of surfaces of infinite type with boundary?

  \section{Several examples of hyperbolic surfaces of infinite type which satisfy the geometric condition $(\star)$}

 In this section, we construct several examples of hyperbolic surfaces of infinite type which satisfy the geometric condition $(\star)$. We find that these hyperbolic surfaces may be incomplete. And we prove that there is no direct relation between the geometric condition $(\star)$ and the Shiga's condition.

\subsection{The construction of examples}  In order to construct the desired examples, we first give the following two lemmas.

       \begin{lemma}\label{Lemma for polygon}
       Let $\mathbb{P}_{n}$ be a geodesically convex hyperbolic n-polygon in $\mathbb{H}^{2}$ with consecutive edges $\alpha_{1}, \alpha_{2}, ..., \alpha_{n}$, where the endpoints of the edge $\alpha_{i}$ are denoted by $Q_{i}$ and $Q_{i-1}$, here $Q_{0}=Q_{n}$. Then
       $\rho(x,\alpha_{1})\leq \sup \limits_{2\leq i \leq n-1}\{\rho(Q_{i}, \alpha_{1})\}$
         for any point $x\in \mathbb{P}_{n}$,
         where $\rho$ denotes the hyperbolic distance on $\mathbb{H}^{2}$ and $\rho(x, \alpha_{1})=\inf\limits_{y\in \alpha_{1}}\rho(x,y)$.
       \end{lemma}

          \begin{proof}
       By the continuity of the hyperbolic distance on $\mathbb{P}_{n}$, we only need to consider the hyperbolic distance from each point of the piecewise geodesic boundary $\partial \mathbb{P}_{n}$ to $\alpha_{1}$. Note that $\mathbb{P}_{n}$ is geodesically convex, the function $f:\partial \mathbb{P}_{n}\backslash\alpha_{1}\rightarrow\mathbb{R}_{\geq0}$ which assigns $\rho(x, \alpha_{1})$ to $x$ restricted to each smooth edge except $\alpha_{1}$ attains its maximum only if $x$ is one of the two endpoints. Therefore,
       \begin{equation*}
       \sup \limits _{x\in \partial \mathbb{P}_{n}\backslash\alpha_{1}}
       \rho(x, \alpha_{1})=\sup \limits_{2\leq i \leq n-1}\{\rho(Q_{i},\alpha_{1})\},
       \end{equation*}
       which implies the desired result.
       \end{proof}

       \begin{lemma}\label{lemma H_{n}}
Let $H_{n} (n\geq1)$ be a right-angled hexagon in $\mathbb{H}^{2}$ with pairwise non-adjacent edges $\alpha_{n},\beta_{n},\gamma_{n}$ whose lengths are respectively $l_{0}, l_{n}, l_{n}$, where $l_{0}>0$, $\{l_{n}\}^{\infty}_{n=1}$ is a strictly increasing sequence of positive numbers and $l_{n}\rightarrow\infty$ as $n\rightarrow\infty$. Then there exists a constant $M>0$, such that
\begin{equation*}
\sup_{n}\{\sup_{x\in H_{n}}\rho(x,\beta_{n})\}=\sup_{n}\{\sup_{x\in H_{n}}\rho(x,\gamma_{n})\}\leq M,
\end{equation*}
 where $\rho$ denotes the hyperbolic distance on $\mathbb{H}^{2}$.
\end{lemma}
\begin{proof}
Denote the vertices of $H_{n}$ which are not on the edge $\gamma_{n}$ by $A_{n}$, $C_{n}$, $D_{n}$, $B_{n}$ respectively in the counter-clockwise order as presented in Figure \ref{fig:not t-eq1}. Denote $\rho(A_{n},\gamma_{n})=a_{n}$, $\rho(B_{n},\gamma_{n})=b_{n}$, $\rho(C_{n},\gamma_{n})=c_{n}$, $\rho(D_{n},\gamma_{n})=d_{n}$.

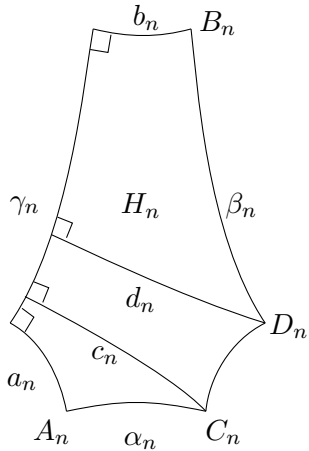
\begin{figure}[ht]
\begin{tikzpicture}[scale=0.78]
\draw (-2.8,2.5) .. controls (-2.2,2.35) and (-1.7,2.35) .. (-1.15,2.5);
\draw (-2.8,2.5) .. controls (-2.95,1.65) and (-3.2,-1.05) .. (-4.2,-2.5);
\draw (-1.15,2.5) .. controls (-1.05,1.65) and (-0.9,-1) .. (0.1,-2.5);
\draw (-4.2,-2.5) .. controls (-3.65,-2.8) and (-3.35,-3.45) .. (-3.25,-4);
\draw (0.1,-2.5) node (v1) {} .. controls (-0.5,-2.8) and (-0.85,-3.45) .. (-0.9,-4);
\draw (-3.25,-4) .. controls (-2.3,-3.8) and (-1.85,-3.8) .. (-0.9,-4);
\node at (-3.95,-0.5) {$\gamma_{n}$};
\node at (-0.3,-0.5) {$\beta_{n}$};
\node at (-2,-4.5) {$\alpha_{n}$};
\node at (-2,-0.5) {$H_{n}$};
\node at (-3.5,-4.35) {$A_{n}$};
\node at (-0.6,-4.35) {$C_{n}$};
\node at (0.5,-2.55) {$D_{n}$};
\node at (-0.7,2.6) {$B_{n}$};
\draw (-4,-2.25) -- (-3.8,-2.4) -- (-3.95,-2.65);
\draw (-2.85,2.15) -- (-2.55,2.1) -- (-2.5,2.4);
\draw (0.1,-2.5) .. controls (-1.35,-2.05) and (-3.05,-1.2) .. (-3.5,-1);
\draw (-3.4,-0.7) -- (-3.15,-0.8) -- (-3.25,-1.05);
\draw (-0.9,-4) .. controls (-1.5,-3.35) and (-3.35,-2.35) .. (-3.95,-2.05);
\draw (-3.8,-1.8) -- (-3.55,-1.9) -- (-3.65,-2.15);
\node at (-4,-3.5) {$a_{n}$};
\node at (-2.6,-3.05) {$c_{n}$};
\node at (-2,-2.1) {$d_{n}$};
\node at (-1.9,2.7) {$b_{n}$};
\end{tikzpicture}
\caption{\small{The right-angled hexagon $H_{n}$ in Lemma \ref{lemma H_{n}}.}}
            \label{fig:not t-eq1}
\end{figure}

 Note that $\rho(x,\beta_{n})=\rho(x,\gamma_{n})$. By Lemma \ref{Lemma for polygon}, it suffices to show that
$\sup\limits_{n}\{a_{n}, b_{n}, c_{n}, d_{n}\}\leq M$ for a constant $M>0$.

By the formula for a right-angled hexagon and the formula for a trirectangle (that is, a quadrilateral with three right angles) \cite{Buser}, we have
\begin{equation}\label{eq:A}
\cosh{a_{n}}=\frac{\cosh{l_{n}}+\cosh{l_{n}}\cosh{l_{0}}}{\sinh{l_{n}}\sinh{l_{0}}}
\leq\frac{\coth{l_{1}}}{\sinh{l_{0}}}+\coth{l_{1}}\coth{l_{0}}.
\end{equation}
\begin{equation}\label{eq:B}
\cosh{b_{n}}=\frac{\cosh{l_{0}}+\cosh{l_{n}}\cosh{l_{n}}}{\sinh{l_{n}}\sinh{l_{n}}}
\leq\frac{\cosh{l_{0}}}{(\sinh{l_{1}})^{2}}+(\coth{l_{1}})^{2}.
\end{equation}
\begin{equation}\label{eq:C}
\sinh{c_{n}}=\sinh{a_{n}}\cosh{l_{0}}\leq\cosh{a_{n}}\cosh{l_{0}}.
\end{equation}
\begin{equation}\label{eq:D}
\sinh{d_{n}}=\sinh{b_{n}}\cosh{l_{n}}=\sqrt{(\cosh{b_{n}})^{2}-1}\cosh{l_{n}}.
\end{equation}

Substitute \eqref{eq:B} into \eqref{eq:D}, we have
\begin{equation}\label{eq:D1}
\begin{split}
\sinh{d_{n}}=\sqrt{\frac{\cosh^{2}{l_{0}}\coth^{2}{l_{n}}}{\sinh^{2}{l_{n}}}
+2\coth^{4}{l_{n}}\cosh{l_{0}}+\cosh^{2}{l_{n}}(\coth^{4}{l_{n}}-1)}\\
\leq\sqrt{\frac{\cosh^{2}{l_{0}}\coth^{2}{l_{1}}}{\sinh^{2}{l_{1}}}
+2\coth^{4}{l_{1}}\cosh{l_{0}}+\cosh^{2}{l_{n}}(\coth^{4}{l_{n}}-1)}
\end{split}
\end{equation}

Note that $\coth{x}\rightarrow 1$, $\sech{x}\rightarrow 0$ as $x\rightarrow \infty$ and $(\coth{x})'=-\csch^{2}{x}$, $(\sech{x})'=-\sech{x}\tanh{x}$, we have
\begin{equation}\label{eq:D2}
\begin{split}
&\lim\limits_{x\rightarrow\infty}\cosh^{2}{x}(\coth^{4}{x}-1)\\
&=\lim\limits_{x\rightarrow\infty}
{\frac{\coth^{4}{x}-1}{\sech^{2}{x}}}\\
&=\lim\limits_{x\rightarrow\infty}
\frac{-4\coth^{3}{x}\csch^{2}{x}}{-2\sech^{2}{x}\tanh{x}}\\
&=\lim\limits_{x\rightarrow\infty}
2\coth^{6}{x}\\
&=2
\end{split}
\end{equation}

Combining \eqref{eq:A}, \eqref{eq:B}, \eqref{eq:C},  \eqref{eq:D1} and  \eqref{eq:D2}, we have the desired result.
\end{proof}

\begin{example}\label{Ex:my}
Now we construct a complete hyperbolic surface $X_{0}$ of infinite type which satisfies the geometric condition $(\star)$.

Let $\{l_{n}\}^{\infty}_{n=1}$ be a strictly increasing divergent sequence of positive numbers.
Let $P_{n}$ be a hyperbolic pair of pants with boundary lengths $(2l_{0}, 2l_{n}, 2l_{n})$. Then we glue $P_{n}$ with its copy $P_{n}'$ along the geodesic boundary component of common length $2l_{n}$. Denote by $X_{n}$ the obtained X-piece for $n\geq1$. Let $X_{0}$ be the surface obtained by gluing the sequence $\{X_{n}\}^{\infty}_{n=1}$ in succession along the geodesic boundary component $\alpha_{n}$ of common length $2l_{0}$, as indicated in Figure \ref{figure1}. Note that the amount of the twisting along the gluing curves can be taken arbitrarily in the above process.

\begin{figure}[ht]
\begin{tikzpicture}[scale=0.65]
\draw (-0.54,0.24) .. controls (-0.49,0.5) and (2.2,0.5175) .. (2.15,0.2175);
\draw (-0.55,0.25) .. controls (-0.6,0) and (2.1,-0.0325) .. (2.15,0.2175);
\draw (-0.55,0.25) .. controls (-0.65,-0.2) and (-0.8,-0.4) .. (-1.15,-0.6);
\draw (-1.15,-0.6) .. controls (-1.4,-0.7) and (-1.15,-1.9235) .. (-0.9,-1.9235);
\draw (-1.15,-0.6) .. controls (-0.9,-0.5) and (-0.6,-1.8235) .. (-0.9,-1.9235);
\draw (-0.9,-1.9235) .. controls (-0.5,-1.7735) and (0.2,-2.3) .. (0,-2.7);
\draw (0,-2.7) .. controls (0,-3) and (2.6,-2.9325) .. (2.6,-2.6325);
\draw [dashed] (0,-2.7) .. controls (0,-2.5) and (2.6,-2.4325) .. (2.6,-2.6325);
\draw(2.6,-2.6325) .. controls (2.75,-2.3325) and (2.85,-2.1325) .. (3.2,-1.9325);
\draw (2.15,0.2175) .. controls (2.1,-0.1825) and (2.5,-0.659) .. (3.05,-0.609);
\draw (3.05,-0.609) .. controls (3.3,-0.609) and (3.45,-1.8325) .. (3.2,-1.9325);
\draw[dashed]  (3.05,-0.609) .. controls (2.8,-0.609) and (3,-1.9825) .. (3.2,-1.9325);
\draw (3.05,-0.609) .. controls (3.55,-0.459) and (3.75,-0.3325) .. (3.85,0.2175);
\draw (3.85,0.2175) .. controls (3.9,0.4675) and (7.55,0.482) .. (7.5,0.182);
\draw (3.85,0.2175) .. controls (3.8,-0.0825) and (7.4,-0.068) .. (7.5,0.182);
\draw (7.5,0.182) .. controls (7.4,-0.118) and (7.6,-0.718) .. (8.2,-0.618);
\draw (3.2,-1.9325) .. controls (3.65,-1.7825) and (4.1,-2.1325) .. (4.1,-2.7325);
\draw (4.1,-2.7325) .. controls (4,-2.9825) and (7.85,-2.965) .. (7.85,-2.665);
\draw [dashed] (4.1,-2.7325) .. controls (4.2,-2.4325) and (7.9,-2.415) .. (7.85,-2.665);
\draw (7.85,-2.665) .. controls (8.05,-2.265) and (8.2,-2.165) .. (8.4,-1.965);
\draw[dashed] (8.2,-0.618) .. controls (7.95,-0.618) and (8.2,-2.065) .. (8.4,-1.965);
\draw (8.2,-0.618) .. controls (8.5,-0.518) and (8.7,-1.815) .. (8.4,-1.965);
\draw (8.2,-0.618) .. controls (8.8,-0.418) and (9.05,-0.118) .. (9.1,0.2615);
\draw (8.4,-1.965) .. controls (8.95,-1.765) and (9.6,-2.2885) .. (9.55,-2.6885);
\draw (9.1,0.2615) .. controls (9.05,0.0115) and (14.2235,-0.0945) .. (14.2,0.279);
\draw (9.1,0.2615) .. controls (9.2,0.5115) and (14.2,0.479) .. (14.2,0.279);
\draw[dashed] (9.55,-2.6885) .. controls (9.55,-2.4885) and (14.6735,-2.3945) .. (14.6735,-2.6445);
\draw (9.55,-2.6885) .. controls (9.55,-2.8885) and (14.6235,-2.8945) .. (14.6735,-2.6445);
\draw (14.2,0.279) .. controls (14.4205,-0.4475) and (14.6,-0.6005) .. (15.1,-0.524);
\draw (14.6735,-2.6445) .. controls (14.7235,-2.3445) and (14.9235,-2.0945) .. (15.4735,-1.7975);
\draw[dashed] (15.1,-0.524) .. controls (14.85,-0.721) and (15.1735,-1.8475) .. (15.4235,-1.7975);
\draw (15.1,-0.524) .. controls (15.4,-0.474) and (15.6735,-1.6975) .. (15.4235,-1.7975);
\draw [dashed](-0.1,-2.2) .. controls (-0.3235,-1.9735) and (2.1265,-0.059) .. (2.3735,-0.3325);
\draw (-0.1,-2.2) .. controls (0.05,-2.45) and (2.4735,-0.509) .. (2.3735,-0.3325);
\draw [dashed](4,-2.2825) .. controls (3.85,-1.9825) and (7.45,-0.118) .. (7.55,-0.368);
\draw (4,-2.2825) .. controls (4.15,-2.4825) and (7.65,-0.568) .. (7.55,-0.368);
\draw[dashed] (9.4235,-2.2885) .. controls (9.25,-2.0885) and (14.2235,-0.0945) .. (14.3,-0.124);
\draw (9.4235,-2.2885) .. controls (9.45,-2.4885) and (14.4,-0.374) .. (14.3,-0.124);
\node at (-0.0307,-0.7731) {$P_{1}$};
\node at (2.3345,-1.5249) {$P_{1}'$};
\node at (4.85,-0.8056) {$P_{2}$};
\node at (7.308,-1.5309) {$P_{2}'$};
\node at (10.4231,-0.618) {$P_{3}$};
\node at (13.8528,-1.413) {$P_{3}'$};
\node at (0.8655,-3.353) {$X_{1}$};
\node at (5.6886,-3.312) {$X_{2}$};
\node at (11.9962,-3.268) {$X_{3}$};
\node at (3.75,-1.2325) {$\alpha_{1}$};
\node at (8.9735,-1.1915) {$\alpha_{2}$};
\node at (15.85,-1.0975) {$\alpha_{3}$};
\draw (15.1,-0.524) .. controls (15.7,-0.324) and (16.15,0.0025) .. (16.2,0.4025);
\draw (15.4735,-1.7975) .. controls (16.05,-1.671) and (16.5,-1.971) .. (16.65,-2.221);
\node at (0.9614,0.8462) {$\beta_{1}$};
\node at (5.6538,0.7693) {$\beta_{2}$};
\node at (11.9231,0.7307) {$\beta_{3}$};
\node at (1.4614,-2.1924) {$\beta_{1}'$};
\node at (6.0769,-2.1155) {$\beta_{2}'$};
\node at (12.5,-2.1155) {$\beta_{3}'$};
\end{tikzpicture}
 \caption{\small{The hyperbolic surface $X_{0}$ of infinite type in Example \ref{Ex:my}.}}
   \label{figure1}
  \end{figure}
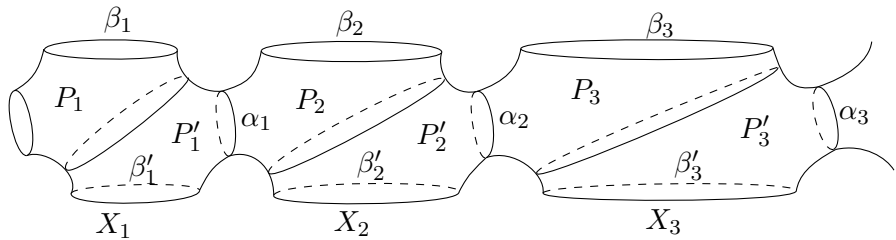
Since any closed ball of radius 1 on the surface $X_{0}$ is contained in a finite number of pairs of pants of the given decomposition as show in Figure \ref{figure1}, then it is compact. By the Hopf-Rinow Theorem, $X_{0}$ is complete.

We claim that $X_{0}$ satisfies the geometric condition $(\star)$.
Indeed, $P_{n}$ can be constructed by pasting two copies of the right-angled hexagon $H_{n}$ with pairwise non-adjacent edges of lengths $l_{0}, l_{n}, l_{n}$ along the remaining edges. Denote by $\beta_{n}$ (resp. $\beta_{n}'$) the boundary component of $P_{n}$ (resp. $P_{n}'$) which is contained in $\partial X_{0}$ and has length $2l_{n}$.
For each $x\in X_{0}$, there exists an integer $N\geq1$, such that $x$ lies in $P_{N}$ or $P_{N}'$. Without lost of generality, we assume that $x\in P_{N}$. By Lemma \ref{lemma H_{n}}, there exists a constant $M>0$ independent of $N$ such that
\begin{equation*}
\rho(x,\partial X_{0})\leq\rho(x, \beta_{N})\leq M,
\end{equation*}
which implies the claim.
\end{example}

          In order to construct more examples which satisfy the geometric condition $(\star)$, we introduce the following notations and propositions given by Basmajian (see \cite{Basmajian2, Basmajian3}).

         A \emph{flute surface} is a hyperbolic surface of infinite type obtained by gluing a sequence of generalized hyperbolic pairs of pants $\{P_{i}\}^{\infty}_{i=0}$ in succession along the common length boundary components, that is, any two adjacent pairs of pants $P_{i}$, $P_{i+1}$ have exactly one common geodesic boundary component which is denoted by $\alpha_{i+1}$ for $i\geq 0$. Note that $P_{0}$ has at least one geodesic boundary component $\alpha_{1}$ and $P_{i}$ has at least two geodesic boundary components $\alpha_{i}$, $\alpha_{i+1}$ for $i\geq 1$.
         We say that a flute surface is \emph{tight} if all the pants holes that have not been glued along are in fact cusps. In this case, denote by $\alpha_{0}$ the image of a horocycle under the universal covering of this surface, which is a simple closed curve of length one and homotopic to a cusp of $P_{0}$ (see Figure \ref{fig:tight flute}).
          We say a subsurface $S$ is a \emph{spike} if it is isometric to the region $\{z=x+iy: 0\leq x\leq1, y>a\}$ of $\mathbb{H}^{2}$, for some $a>0$.

          Let $\ell_{i}$ be the length of $\alpha_{i}$ for $i\geq0$. Denote by $d_{i}$ the hyperbolic distance from $\alpha_{i}$ to $\alpha_{i+1}$ and denote by $s_{i}$ the amount of the twisting along $\alpha_{i+1}$ for $i\geq 0$. Here the amount of a positive Dehn-twist along $\alpha_{i+1}$ is defined to be the hyperbolic length of $\alpha_{i+1}$.

                       \begin{figure}[ht]
\begin{tikzpicture}[scale=0.90]
\draw (-4.75,3.4) .. controls (-5,0.65) and (-5.45,0.45) .. (-7.45,-0.05);
\draw (-7.45,-0.05) .. controls (-5.4,0.2) and (-1.85,0.25) .. (0.8,-0.25);
\draw (-4.7,3.4) .. controls (-4.6,0.9) and (-4.35,1) .. (-3.6,1.05);
\draw (-3.6,1.05) .. controls (-3.3,1.2) and (-2.75,1.7) .. (-2.55,4.387);
\draw (-2.5,4.387) .. controls (-2.35,1.9) and (-2.15,2) .. (-1.7,2.05);
\draw (-1.7,2.05) .. controls (-1.35,2.15) and (-1,3.15) .. (-0.85,5.7896);
\draw (-0.8,5.7896) .. controls (-0.75,3.55) and (-0.55,3.3) .. (-0.15,3.45);
\draw (-0.15,3.45) .. controls (0.35,3.6) and (0.6,4.55) .. (0.6,6.837);
\draw[dashed] (-5.9332,0.4056) .. controls (-6.0332,0.4556) and (-6.0332,0.1) .. (-5.9332,0.1);
\draw (-5.9332,0.4056) .. controls (-5.8332,0.4556) and (-5.8332,0.1) .. (-5.9332,0.1);
\draw[dashed] (-3.8,1.05) .. controls (-4,1.05) and (-4.05,0.15) .. (-3.85,0.15);
\draw (-3.8,1.05) .. controls (-3.6,1.05) and (-3.65,0.15) .. (-3.85,0.15);
\draw[dashed] (-1.8,2.05) .. controls (-2,2.05) and (-2.15,0.1) .. (-1.95,0.1);
\draw (-1.8,2.05) .. controls (-1.55,2.05) and (-1.75,0.05) .. (-1.95,0.1);
\draw[dashed] (-0.35,3.4) .. controls (-0.65,3.45) and (-1,0) .. (-0.75,0);
\draw (-0.35,3.4) .. controls (-0.15,3.35) and (-0.5,-0.05) .. (-0.75,0);
\node at (-3,-0.25) {$d_{1}$};
\node at (-1.3,-0.3) {$d_{2}$};
\draw (0.8,-0.25) .. controls (1.25,-0.3) and (1.55,5.95) .. (1.6,7.8422);
\draw[dashed] (0.8,-0.25) .. controls (0.5,-0.15) and (0.95,5.95) .. (1,7.5818);
\draw[dotted, ultra thick] (0.2,1.95) -- (0.65,1.95);
\draw[dotted, ultra thick] (-0.05,-0.45) -- (0.4,-0.5);
\node at (-4.8,1.25) {$P_{0}$};
\node at (-2.65,1.35) {$P_{1}$};
\node at (-1.1,1.4) {$P_{2}$};
\node at (-3.3388,0.55) {$\alpha_{1}$};
\node at (-1.4444,0.6) {$\alpha_{2}$};
\node at (-0.1888,0.6) {$\alpha_{3}$};
\node at (1.4,0.6) {$\alpha$};
\node at (-5.556,0.3332) {$\alpha_{0}$};
\node at (-5,-0.278) {$d_{0}$};
\end{tikzpicture}
			\caption{\small{A tight flute surface $Y_{0}$ (where $\alpha_{0}$ is the image of a horocycle under the universal covering and has length one).}}
            \label{fig:tight flute}
		\end{figure}
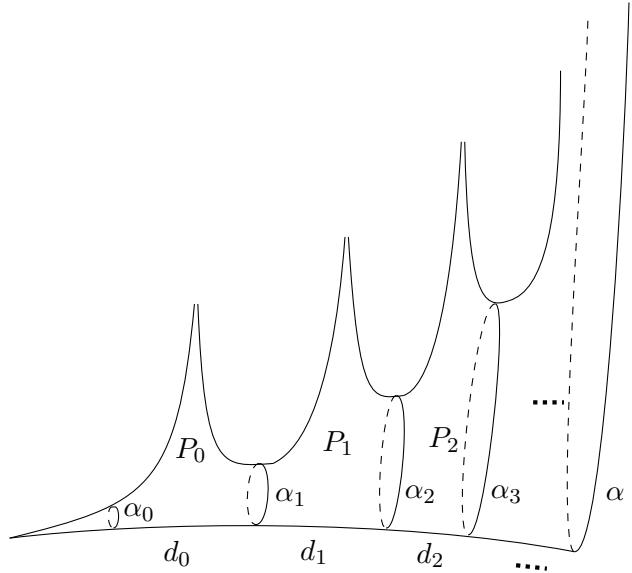

         Let $\{L_{i}\}_{i\geq0}$ be a sequence of geodesics in $\mathbb{H}^{2}$. We say that $\{L_{i}\}$ is a \emph{nested sequence of geodesics} if $L_{i-1}$ and $L_{i+1}$ lie in different components of $\mathbb{H}^{2}-L_{i}$ for each $i\geq1$, and if the $L_{i}$ are disjoint in $\overline{\mathbb{H}}^{2}$.
         $\{L_{i}\}$ \emph{converges to the geodesic $L$} if the endpoints of $L_{i}$ converge to the endpoints of $L$ on $\partial\mathbb{H}^{2}$. If the endpoints of $L_{i}$ converge to a single point of $\partial\mathbb{H}^{2}$, then we say that the sequence $\{L_{i}\}$ \emph{converges to a point} on the boundary of the hyperbolic plane. It is known that the limit of a nested sequence of geodesics in $\mathbb{H}^{2}$ is unique and is one of the above two possibilities.

         Introduce a positive direction for each $L_{i}$ by designating $L_{i+1}$ to lie to the right of $L_{i}$. Let $\sigma_{i}$ be the unique common perpendicular between the geodesics $L_{i}$ and $L_{i+1}$. The distance from $\sigma_{i}$ to $\sigma_{i+1}$ is measured by traversing $L_{i+1}$. Let $s_{i}=\rho(\sigma_{i},\sigma_{i+1})$ if $L_{i+1}$ is traversed in the positive direction, and let $s_{i}=-\rho(\sigma_{i},\sigma_{i+1})$ if $L_{i+1}$ is traversed in the negative direction. Set $d_{i}=\rho(L_{i},L_{i+1})$. Here $\rho$ is the hyperbolic distance on $\mathbb{H}^{2}$.

         \begin{proposition}\label{Basmajian proposition1}
         \emph{(Basmajian \cite{Basmajian2})}
         Let the sequence $\{L_{i}\}_{i\geq0}$ be a nested sequence of geodesics in $\mathbb{H}^{2}$. Then $L_{i}$ converges to a geodesic if and only if
         \begin{equation*}
         \lim_{i\rightarrow\infty}\rho(L_{1}, L_{i})<\infty,
         \end{equation*}
         where $\rho(L_{1},L_{i})$ is the hyperbolic distance between $L_{1}$ and $L_{i}$.
         \end{proposition}

        \begin{proposition}\label{Basmajian proposition2}
        \emph{(Basmajian \cite{Basmajian2}, The Pair of Pants Theorem)}
        Suppose $\gamma$ and $\beta$ are nonelliptic elements. Let $d$ be the hyperbolic distance between the axes of $\gamma$ and $\beta$ (if $\gamma$ is parabolic, the axis of $\gamma$ is the horocycle based at the fixed point of $\gamma$ whose projection to $\mathbb{H}^{2}/\langle\gamma\rangle$ has length one). Then $(\gamma, \beta)$ form standard generators for a tight pair of pants (that is, the third boundary component of this pair of pants is a cusp) if and only if $c(\gamma)+c(\beta)= d$. Here $c(\gamma)=\log2$ if $\gamma$ is parabolic, and $c(\gamma)=\log\coth{\frac{T(\gamma)}{4}}$ if $\gamma$ is hyperbolic, where $T(\gamma)$ is the translation length of $\gamma$.

        \end{proposition}

       We first give an example of incomplete hyperbolic surfaces of infinite type which satisfy the geometric condition $(\star)$ by the following proposition.

       \begin{proposition}\label{Ex1:tight flute surface }
         Let $Y_{0}$ be a tight flute surface with $\sum d_{i}<\infty$ and $\sum |s_{i}|<\infty$, where the sum is taken over all $i\geq0$. If there exists a constant $M>0$ such that
        \begin{equation}\label{condition for incomplete}
        \sinh\big(\sum\limits^{\infty}\limits_{i=n-1}d_{i}\big)\cosh{\frac{ l_{n}}{2}}\leq M,
        \end{equation}
        for all $n\geq1$, then $Y_{0}$ is incomplete and satisfies the geometric condition $(\star)$.
       \end{proposition}

       \begin{proof}

         By the assumptions that $\sum d_{i}<\infty$ and
         $\sum |s_{i}|<\infty$, it follows from
         Proposition \ref{Basmajian proposition1} that the nested sequence $\{\alpha_{i}\}$ converges to a geodesic. We denote it by $\alpha$. We claim that the length of $\alpha$ must be infinity. Otherwise, assume that the length of $\alpha$ is a finite positive number $l$. Then $l_{i}\leq l+1$ for all $i\geq N$, where $N$ is a sufficiently big integer. By the formula for a pentagon with four right angles and an angle of zero (see \cite{Buser}), we have
          \begin{equation*}
          \cosh{d_{i}}=\frac{1+\cosh{\frac{l_{i}}{2}}\cosh{\frac{l_{i+1}}{2}}}
          {\sinh{\frac{l_{i}}{2}}\sinh{\frac{l_{i+1}}{2}}}
          \geq \frac{1}{\sinh^{2}{\frac{l+1}{2}}}+\coth^{2}{\frac{l+1}{2}},
          \end{equation*}
          for all $i\geq N$, which contradicts the assumption that $\sum{d_{i}}<\infty$.
         Therefore, $\ell_{i}\rightarrow \infty$, as $i\rightarrow \infty$ and $Y_{0}$ is an incomplete hyperbolic surface with a simple open infinite geodesic boundary.

         Now we prove that $Y_{0}$ satisfies the geometric condition $(\star)$ if it satisfies the condition
        \eqref{condition for incomplete}.

        First we consider the special case that $s_{i}=0$ for all $i\geq 0$. In this case, $Y_{0}$ can be constructed by pasting two copies of the geodesically convex ideal region $R$ with infinitely many geodesic edges along all the edges of common lengths  except $\alpha'$, which is half of the geodesic $\alpha$.

\begin{figure}[ht]
\begin{tikzpicture}[scale=1.25]
\draw (-4.5,3.27) .. controls (-4.7295,0.9485) and (-5.24,0.42) .. (-7.6,-0.05);
\draw (-7.6,-0.05) .. controls (-5.23,0.25) and (-1.2655,0.16) .. (0.51,0.0785);
\draw (-4.5,3.27) .. controls (-4.3445,1.5705) and (-4.0125,1.7425) .. (-3.79,1.721);
\draw (-3.79,1.721) .. controls (-3.531,1.7495) and (-3.224,2.37) .. (-3.0285,4.45);
\draw (-3.0285,4.45) .. controls (-2.876,1.9965) and (-2.6995,1.3355) .. (-1.93,1.304);
\draw (-1.93,1.304) .. controls (-1.6625,1.2735) and (-0.9795,2.4205) .. (-0.827,5.1);
\draw (-0.827,5.1) .. controls (-0.6715,2.69) and (-0.591,2.6015) .. (-0.31,2.44);
\draw (-0.31,2.44) .. controls (0.063,2.351) and (0.313,3.664) .. (0.4115,6.025);
\draw[red] (-5.58,0.62) .. controls (-5.48,0.44) and (-5.44,0.29) .. (-5.48,0.1) node (v2) {};
\draw[red] (-3.79,1.721) .. controls (-3.9585,0.7712) and (-3.9,0.32) .. (-3.9,0.1485) node (v1) {};
\draw [red](-1.93,1.304) .. controls (-2.047,0.8375) and (-2.01,0.45) .. (-2.01,0.1485);
\draw[red] (-0.31,2.44) .. controls (-0.597,2.0245) and (-0.56,1.05) .. (-0.61,0.1385);
\node at (-2.96,-0.3) {$d_{1}$};
\node at (-1.25,-0.34) {$d_{2}$};
\draw (0.51,0.0785) .. controls (0.54,2.75) and (0.63,3.5449) .. (0.82,6.5155);
\draw[dotted, ultra thick] (-0.14,1.6) -- (0.31,1.6);
\draw[dotted, ultra thick] (-0.29,-0.35) -- (0.16,-0.35);
\node at (-3.69,0.74) {$\alpha_{1}'$};
\node at (-1.8,0.68) {$\alpha_{2}'$};
\node at (-0.35,1.28) {$\alpha_{3}'$};
\node at (-3.8488,2.0095) {$E_{1}$};
\node at (-1.8955,1.564) {$E_{2}$};
\node at (-0.277,2.69) {$E_{3}$};
\node at (-3.69,0.3) {$F_{1}$};
\node at (-1.84,0.26) {$F_{2}$};
\node at (-0.43,0.26) {$F_{3}$};
\node at (-5.63,0.83) {$E_{0}$};
\node at (-5.62,-0.11) {$F_{0}$};
\draw (-3.79,1.721) .. controls (-3.8782,0.77) and (-4.9085,0.823) .. (-5.104,1.083);
\node at (-5.3203,1.2715) {$G_{1}$};
\draw[blue] (-5.104,1.083) .. controls (-4.997,0.82) and (-4.957,0.39) .. (-4.957,0.1385);
\node at (-4.707,0.3) {$H_{1}$};
\draw (-2.21,1.344) .. controls (-2.5264,0.5966) and (-3.9855,0.6345) .. (-3.79,1.721);
\node at (-2.298,1.6245) {$G_{2}$};
\draw[blue] (-2.2095,1.353) .. controls (-2.384,0.7275) and (-2.3525,0.37) .. (-2.3525,0.1485);
\node at (-2.6725,0.3) {$H_{2}$};
\draw (-0.31,2.44) .. controls (-0.5564,1.9054) and (-1.084,1.8748) .. (-1.4025,2.036);
\draw[blue] (-1.391,2.0245) .. controls (-1.2155,1.5135) and (-1.2555,1.05) .. (-1.2655,0.1385);
\node at (-1.561,2.176) {$G_{3}$};
\node at (-1.0655,0.26) {$H_{3}$};
\draw (-3.9,0.1485) -- (-3.75,0.01) -- (-3.06,0.01) -- (-3.01,-0.09) -- (-2.94,0.01) -- (-2.11,0.01) -- (-2.01,0.1485);
\draw (-2.01,0.1485) -- (-1.9,0) -- (-1.32,0) -- (-1.26,-0.14) -- (-1.2,0) -- (-0.71,0) -- (-0.61,0.1385);
\draw[red] (-3.9,0.25) -- (-4,0.25) -- (-4,0.1485);
\draw[blue] (-2.3525,0.25) -- (-2.4525,0.25) -- (-2.4525,0.1485);
\draw [red](-2.11,0.1485) -- (-2.11,0.25) -- (-2.01,0.25);
\draw[blue] (-1.3655,0.1385) -- (-1.3655,0.24) -- (-1.2655,0.24);
\draw[red] (-0.71,0.1385) -- (-0.71,0.24) -- (-0.61,0.24);
\draw (-5.0443,1.2115) -- (-4.9043,1.1215) -- (-5.004,0.9833);
\draw (-3.93,1.7325) -- (-3.9497,1.5943) -- (-3.8097,1.5843);
\draw (-2.2925,1.413) -- (-2.3698,1.3115) -- (-2.27,1.2515);
\draw (-2.0545,1.3135) -- (-2.0897,1.187) -- (-1.9609,1.1658);
\draw (-1.3525,2.163) -- (-1.2428,2.113) -- (-1.2928,1.9745);
\draw (-0.3985,2.5015) -- (-0.4697,2.4033) -- (-0.3697,2.3303);
\draw[red] (-5.6703,0.59) -- (-5.6203,0.49) -- (-5.53,0.541);
\draw[red] (-5.46,0.2) -- (-5.58,0.2) -- (-5.58,0.1);
\node at (-7.854,-0.077) {$A_{0}$};
\node at (-4.5,3.5505) {$A_{1}$};
\node at (-3.017,4.7075) {$A_{2}$};
\node at (-0.8,5.323) {$A_{3}$};
\node at (0.657,-0.038) {$A$};
\node at (0.843,6.746) {$A'$};
\node at (0.7725,1.8635) {$\alpha'$};
\draw (0.506,0.1925) -- (0.3925,0.192) -- (0.3925,0.077);
\draw[dashed] (-5.1075,1.0825) .. controls (-4.6824,0.9246) and (-4.155,1.1272) .. (-3.79,1.721);
\draw [dashed](-3.79,1.721) .. controls (-3.6905,1.0482) and (-2.6209,0.8894) .. (-2.1985,1.3545);
\draw[dashed] (-1.381,2.031) .. controls (-0.7565,2.074) and (-0.6036,2.1925) .. (-0.2985,2.462);
\node at (-4.54,1.2755) {$\gamma_{1}$};
\node at (-3.0555,1.2785) {$\gamma_{2}$};
\node at (-0.88,2.303) {$\gamma_{3}$};
\draw(-5.59,0.641) .. controls (-5.06,0.44) and (0.3,0.4) .. (0.5,0.4);
\node at (0.7,0.42) {$I_{0}$};
\draw (0.5,0.52) -- (0.4,0.52) -- (0.4,0.42);
\draw (-5.48,0.12) -- (-5.4,-0.04) -- (-4.78,-0.04) -- (-4.7,-0.14) -- (-4.64,-0.04) -- (-4.04,0) -- (-3.88,0.14);
\node at (-4.68,-0.34) {$d_{0}$};
\node at (-5.28,0.36) {$\alpha_{0}'$};
\draw[blue] (-4.95,0.24) -- (-5.06,0.24) -- (-5.06,0.14);
\end{tikzpicture}
\caption{\small{The geodesically convex ideal region $R$ (where  each spike $S_{n}$ is an open subset of $R$, the dashed lines are geodesics $\gamma_{n}$ between $G_{n}$ and $E_{n}'$ and the angle at $E_{n}'$ between the finite edge $\overline{G_{n}E_{n}'}$ of the spike $S_{n}$ and the common perpendicular $\alpha_{n}'$ is zero. In this figure, $E_{1}'=E_{2}'=E_{1}, E_{3}'=E_{3}$).}}
            \label{fig:convex ideal region R}
		\end{figure}

         Now we consider the geodesically convex ideal region $R$ with ideal vertices $\{A_{i}\}^{\infty}_{i=0}$ corresponding to the cusps of $Y_{0}$, as shown in Figure
          \ref{fig:convex ideal region R}. Denote by $A, A'$ the two endpoints of $\alpha'$ (where $A'$ is an ideal vertex). Let $\alpha_{0}'$ be half of the simple closed curve $\alpha_{0}$ and let $\alpha_{i}'$ be the common perpendicular between the infinite geodesic edge $\overline{A_{i}A_{i+1}}$ and the infinite geodesic edge $\overline{A_{0}A}$ for $i\geq1$.

          It suffices to find a constant $M'>0$ and a disjoint union $S$ of spikes in $R$ such that any point in $R\setminus S$ is within the distance $M'$ of $\alpha'$.

         To see this, we denote by $E_{i}, F_{i}$ the two endpoints of $\alpha_{i}'$ (where $F_{i}$ lies in the edge $\overline{A_{0}A}$). Since the length of $\alpha_{i}$ is $l_{i}$, then the hyperbolic length of $\alpha_{i}'$ is $\frac{1}{2}\ell_{i}$. For the ideal vertex $A_{0}$, we take a spike $S_{0}$ which has a finite edge $\alpha_{0}'$ of length $\frac{1}{2}l_{0}=\frac{1}{2}$. For each ideal vertex $A_{i}$ ($i\geq 1$), we take a spike $S_{i}$ such that it goes through the point $E_{i}'$, where $E_{i}'=E_{i}$ if $l_{i}\geq l_{i-1}$, and $E_{i}'=E_{i-1}$ if $l_{i}<l_{i-1}$. Denote by $G_{i}$ the third vertex of $S_{i}$ except the two vertices $A_{i}$ and $E_{i}'$ (note that each spike $S_{i}$ is an open subset of $R$ and the finite edge $\overline{E_{i}'G_{i}}$ is not a geodesic). Draw a geodesic segment which starts at $G_{i}$ and intersects the edge $\overline{A_{0}A}$ perpendicularly at the point $H_{i}$ for $i\geq1$.

         We claim that each spike $S_{i}$ is disjoint from any other spikes. Indeed, we represent the geodesically convex ideal region $R$ in the upper half-plane model of $\mathbb{H}^{2}$ (see Figure \ref{fig:convex ideal region R in H2}). It suffices to consider the position of $S_{i+1}$ in the special case that $l_{i}=l_{i+1}$. In this case,  the vertex $G_{i+1}$ of $S_{i+1}$ coincides with the point $E_{i}$ (see Figure \ref{fig:spikes in H2}). By the construction of $S_{i}$, we have that $S_{i}$ and $S_{j}$ are disjoint for all $i\not= j$.

\begin{figure}[ht]
\begin{tikzpicture}[scale=0.81]
\draw[->] (-3.5,0.5) -- (8.7,0.5);
\draw[->] (-3.5,0.5) -- (-3.5,11.9979);
\node at (6.85,9.5798) {$\mathbb{H}^{2}$};
\draw[very thick](7.3,0.5) arc (0:90:10.8);
\draw [very thick](5.5,0.5) arc (180:0:0.3);
\node at (-4,11.2798) {$A$};
\node at (-3.6,0.05) {$A_{0}$};
\node at (-0.4872,0) {$A_{1}$};
\node at (1.2171,0) {$A_{2}$};
\node at (4.0234,0) {$A_{3}$};
\node at (5.6427,-0.05) {$A_{4}$};
\node at (7.3914,0) {$A'$};
\node at (-3.9383,4.2809) {$F_{1}$};
\node at (-4,6.6607) {$F_{2}$};
\node at (-4,8.6809) {$F_{3}$};
\draw [fill=gray,opacity=0.5] (-0.5,1.13) ellipse (0.638 and 0.638);
\draw [fill=gray,opacity=0.5] (1.1,1.5) ellipse (1 and 1);
\draw [fill=gray,opacity=0.5] (4.1,1.12) ellipse (0.62 and 0.62);
\draw[fill=gray, opacity=0.5] (-3.5,0.5) arc (-90:90:1.1);

\draw[very thick][fill=white] (-0.5,0.5) arc (180:0:0.8);
\draw[very thick][fill=white] (1.1,0.5) arc (180:0:1.5);
\draw[very thick][fill=white] (4.1,0.5) arc (180:0:0.7);
\draw[very thick][fill=white] (-3.5,0.5) node (v4) {} arc (-180:-360:1.5);

\draw [blue](0.2,0.5) arc (0:90:3.7);
\draw [blue](2.4,0.5) arc (0:90:5.9);
\draw[blue] (4.74,0.5) arc (0:90:8.2377);

\draw [red][dashed](-0.743,0.5) arc (0:90:2.757);
\draw[red] [dashed](2.2,0.5) arc (0:90:5.7);
\draw[red] [dashed](3.775,0.5) arc (0:90:7.275);
\draw (0.3,0.5) -- (0.018,1.8);
\draw (2.6,0.5) -- (2.1383,2.3468);
\draw (4.8,0.5) -- (4.68,1.56);
\node at (-3.9383,3.4447) {$H_{1}$};
\node at (-3.9383,6.0245) {$H_{2}$};
\node at (-4,7.8043) {$H_{3}$};
\node at (-1.4872,4.1457) {$\alpha_{1}'$};
\node at (-0.1849,5.9689) {$\alpha_{2}'$};
\node at (1.4703,7.8201) {$\alpha_{3}'$};
\draw[blue] (-3.5,8.9043) -- (-3.35,8.9043) -- (-3.35,8.7043);
\draw [blue](-3.5,6.5639) -- (-3.35,6.5639) -- (-3.35,6.4139);
\draw[blue] (-3.5,4.3575) -- (-3.35,4.3575) -- (-3.35,4.2075);
\draw[blue] (0.1809,1.1749) -- (0.2692,1.1949) -- (0.2297,1.3213);
\draw [blue](2.2505,1.8436) -- (2.38,1.87) -- (2.34,1.9936);
\draw[blue] (4.73,1.0734) -- (4.8426,1.088) -- (4.838,1.23);
\draw [very thick](-3.5,11.3064) -- (-3.5,0.5);
\node at (0.4043,1.8181) {$E_{1}$};
\node at (-0.8021,2.0149) {$G_{1}$};
\node at (2.5022,2.367) {$E_{2}$};
\node at (1.5458,2.1085) {$G_{2}$};
\node at (5.016,1.5351) {$E_{3}$};
\node at (3.9926,1.967) {$G_{3}$};
\node at (-2.3085,1.5947) {$E_{0}$};
\node at (-3.9383,2.8564) {$F_{0}$};
\draw (-2.4702,2.08) -- (-2.3819,2.122) -- (-2.3319,1.9564);
\draw (-3.5,2.8479) -- (-3.35,2.8479) -- (-3.35,2.6979);
\node at (0.3341,0) {$O_{1}$};
\node at (2.6448,-0.0617) {$O_{2}$};
\node at (4.8768,0) {$O_{3}$};
\node at (-1.9936,0) {$O_{0}$};
\draw (-2,0.4958) -- (-2,0.5617);
\draw[dotted, ultra thick] (6.4104,-0.0625) -- (6.8613,-0.0625);
\draw [dotted, ultra thick](6.4945,0.7509) -- (6.9096,0.7509);
\draw [dotted, ultra thick] (-3.0575,10.1895) -- (-3.0575,9.747);
\draw [dotted, ultra thick] (-4,10.132) -- (-4,9.747);
\node at (3.1244,9.9704) {$\alpha'$};
\draw (-3.5,0.5) .. controls (-0.5,1.2532) and (-0.1403,1.2043) .. (1.3236,1.5256);
\draw (-3.5,0.5) .. controls (1.8702,1.9872) and (1.7363,1.8149) .. (3.2916,2.2468);
\draw (-3.5,0.5) .. controls (3.5489,1.1234) and (4.3766,1.1851) .. (5.766,1.2714);
\node at (-2.6915,2.6915) {$\alpha_{0}'$};
\end{tikzpicture}
\caption{\small{The geodesically convex ideal region $R$ in $\mathbb{H}^{2}$, where the boundary of $R$ is drawn in bold lines, and $O_{i}$ is the Euclidean center of the semi-circle corresponding to the infinite geodesic edge $\overline{A_{i}A_{i+1}}$ of $R$ ( in this figure, $E_{1}'=E_{2}'=E_{1}, E_{3}'=E_{3}$).}}
            \label{fig:convex ideal region R in H2}
		\end{figure}
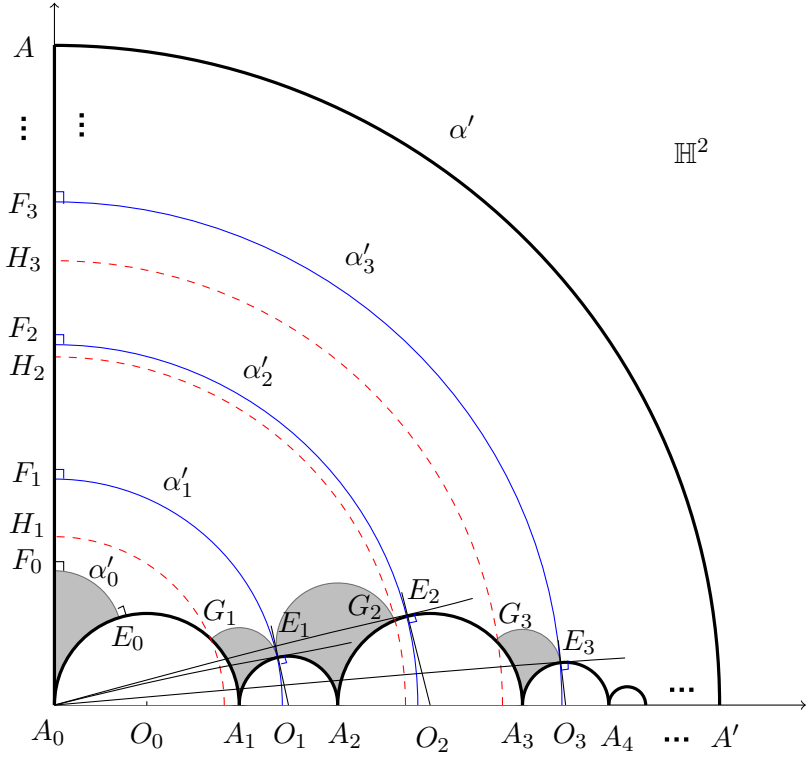

 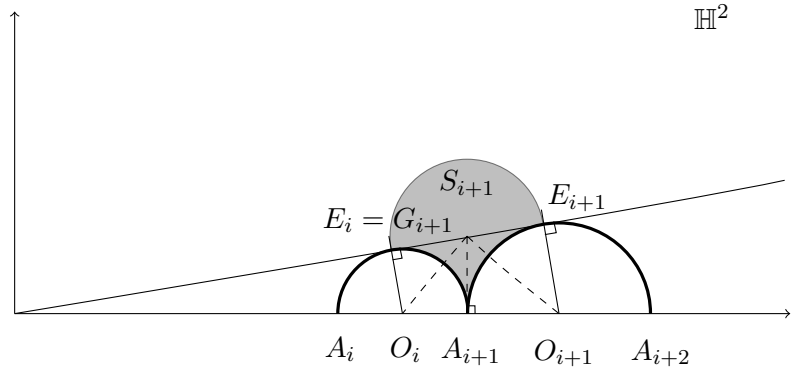
\begin{figure}[ht]
\begin{tikzpicture}[scale=0.85]
\draw[fill=gray,opacity=0.5]  (7,1.209) node (v4) {} ellipse (1.192 and 1.2);
\draw[very thick][fill=white] (5,0) arc (180:0:1.008);
\draw[very thick][fill=white] (7.008,0) node (v5) {} arc (180:0:1.415);
\draw [->] (0,0)--(12,0);
\draw [->] (0,0) node (v1) {}--(0,4.7118);
\draw (6,0) node (v2) {} -- (5.79,1.212);
\draw (8.423,0) node (v3) {} -- (8.15,1.622);
\node at (8.704,1.811) {$E_{i+1}$};
\node  at (5.7826,1.445) {$E_{i}=G_{i+1}$};
\node at (7.0588,-0.5912) {$A_{i+1}$};
\node at (5.037,-0.55) {$A_{i}$};
\draw (5.96,1.0318) -- (5.9946,0.8648) -- (5.8474,0.84);
\draw (8.3414,1.439) -- (8.3802,1.2626) -- (8.2038,1.233);
\node at (10.79,4.6356) {$\mathbb{H}^{2}$};
\node at (10,-0.6176) {$A_{i+2}$};
\node at (8.5,-0.6176) {$O_{i+1}$};
\node at (6.049,-0.5588) {$O_{i}$};
\draw[dashed] (7,1.209) .. controls (7,0.854) and (7,0.549) .. (7,0);
\draw[dashed] (7,1.209) .. controls (6.5,0.622) and (6.5,0.622) .. (6,0);
\draw [dashed](7,1.209) .. controls (7.5,0.744) and (7.5,0.744) .. (8.423,0);
\draw (7,0.122) -- (7.122,0.122) -- (7.124,0);
\draw (0,0) .. controls (3,0.5) and (10.378,1.744) .. (11.914,2.08);
\node at (7,2.0296) {$S_{i+1}$};
\end{tikzpicture}
\caption{\small{The position of the spike $S_{i+1}$ when $l_{i}=l_{i+1}$ ( in this figure, $E_{i+1}'=E_{i+1}$).} }
            \label{fig:spikes in H2}
		\end{figure}

       Let $c_{i}$ be the hyperbolic distance between $H_{i}$ and $F_{i}$ for $i\geq1$. Let $d$ be the hyperbolic distance between $\alpha_{0}$ and $\alpha'$. Then $\sum c_{i}
         <\sum d_{i}=d<\infty$.

       Let $S=\cup_{i\geq0} S_{i}$.  We need to show that there exists a constant $M'>0$ such that any point in $R\setminus S$ is within the distance $M'$ of $\alpha'$.

           Note that Lemma \ref{Lemma for polygon} can be generalized to the case of a  geodesically convex simply connected region in $\mathbb{H}^{2}$ with infinitely many geodesic edges, and each geodesic arc $\gamma_{n}$ connecting $E_{n}$ and $G_{n}$ is contained in $S_{n}$, it suffices to consider $\rho(E_{n},\alpha')$ for $n\geq0$ and $\rho(G_{n},\alpha')$ for $n\geq1$, where $\rho$ is the hyperbolic distance on $R$.

         Now we compute $\rho(E_{n}, \alpha')$. Draw a geodesic from $E_{n}$ to $\alpha'$ such that it intersects $\alpha'$ at the point $I_{n}$ perpendicularly. Then $\rho(E_{n},\alpha')=\rho(E_{n},I_{n})$. The geodesic segment $\overline{E_{n}I_{n}}$ is an edge of the trirectangle with consecutive vertices $A, I _{n}, E_{n}, F_{n}$ (see the trirectangle with consecutive vertices $A, I _{0}, E_{0}, F_{0}$ in Figure
         \ref{fig:convex ideal region R} as an example). By the formula for trirectangles (see \cite{Buser}), for each $n\geq0$, we have
         \begin{equation}\label{eq:d(E)}
         \sinh{\rho(E_{n},\alpha')}
        =\sinh{(\sum\limits^{\infty}_{i=n}d_{i})}\cosh{\frac{\ell_{n}}{2}}.
         \end{equation}

          To estimate $\rho(G_{n},\alpha')$, we need to estimate the length $b_{n}$ of the geodesic segment $\overline{G_{n}H_{n}}$.
         Note that $b_{n}\leq \frac{l_{n}}{2}$ if $E_{n}'=E_{n}$ and $b_{n}\leq \frac{l_{n-1}}{2}$ if $E_{n}'=E_{n-1}$.
          Then $b_{n}\leq \max\{\frac{l_{n}}{2},\frac{l_{n-1}}{2}\}$ for all $n\geq 1$.

          Similarly, we compute $\rho(G_{n},\alpha')$ in a trirectangle.
          For each $n\geq2$,
           \begin{equation}\label{eq:d(G)}
          \begin{split}
           \sinh{\rho(G_{n},\alpha')}
         &=\sinh{(\sum\limits^{\infty}_{i=n}d_{i}+c_{n})}
         \cosh{b_{n}}\\
           &=\big(\sinh(\sum\limits^{\infty}_{i=n}d_{i})\cosh c_{n}+
         \cosh(\sum\limits^{\infty}_{i=n}d_{i})\sinh c_{n}\big)\cosh b_{n}
          \end{split}
         \end{equation}
          \begin{equation*}
         \begin{split}
            &\leq 2\cosh{d}\sinh (\sum\limits^{\infty}_{i=n-1}d_{i}) \max\{\cosh{\frac{l_{n}}{2}}, \cosh{\frac{l_{n-1}}{2}}\}\\
         &\leq2\cosh{d} \max\{\sinh (\sum\limits^{\infty}_{i=n-1}d_{i})\cosh{\frac{l_{n}}{2}}, \sinh (\sum\limits^{\infty}_{i=n-2}d_{i})\cosh{\frac{l_{n-1}}{2}}\} .
        \end{split}
         \end{equation*}

        By $\eqref{eq:d(E)}$, $\eqref{eq:d(G)}$ and the given condition $\eqref{condition for incomplete}$, for all $n\geq1$, we get
       \begin{equation*}
       \sinh{ \rho(E_{n}, \alpha')}\leq M,
       \end{equation*}
       and for all $n\geq2$, we have
         \begin{equation*}
       \sinh{ \rho(G_{n}, \alpha')}\leq
        2M\cosh{d}.
         \end{equation*}
        Besides,
         \begin{equation*}
         \begin{split}
        \sinh{\rho(E_{0},\alpha')}&=\sinh{d}\cosh{\frac{1}{2}}\leq
        \cosh{d}\cosh{\frac{1}{2}},\\
        \sinh{\rho(G_{1},\alpha')}&=\sinh({c_{1}+\sum\limits_{i=1}\limits^
         {\infty}d_{i}})\cosh{\frac{l_{1}}{2}}
         \leq\cosh{d}\cosh{\frac{l_{1}}{2}}.
         \end{split}
         \end{equation*}

         Note that $x<\sinh{x}$ for all $x>0$. Let $M'=2\cosh d (M+\cosh\frac{1}{2}+\cosh\frac{l_1}{2})$. Then any point in $R\setminus S$ is within the distance $M'$ of $\alpha'$.

          Now we consider the general case that  $\sum |s_{i}|<\infty$. Denote by $\bar{\rho}(E_{n}, \alpha')$ (resp. $\bar{\rho}(G_{n}, \alpha')$) the hyperbolic distance between $E_{n}$ (resp. $G_{n}$) and $\alpha'$. Then
         \begin{equation*}
        \bar{\rho}(E_{n}, \alpha')
   \leq\rho(E_{n}, \alpha')+\sum\limits^{\infty}\limits_{i=n}|s_{i}|,
      \end{equation*}
      \begin{equation*}
   \bar{\rho}(G_{n}, \alpha')
   \leq\rho(G_{n}, \alpha')+\sum\limits^{\infty}\limits_{i=n}|s_{i}|.
      \end{equation*}
     Note that $\sum|s_{i}|<\infty$,  the statement is also true for the general case.
    This completes the proof of this proposition.
       \end{proof}

       \begin{example}\label{incomplete ex}
       Now we construct an incomplete hyperbolic surface $Y_{0}$ of infinite type which satisfies the geometric condition $(\star)$.

       Consider a tight flute surface with the sequence $\{P_{i}\}_{i\geq0}$ of glued generalized hyperbolic pairs of pants. Since each pair of pants $P_{i}$ is tight for $i\geq0$, it follows from Proposition
       \ref{Basmajian proposition2} that $c(\alpha_{i})+c(\alpha_{i+1})=d_{i}$ for $i\geq0$.
    For each $n\geq2$, we obtain that
      \begin{equation*}
    \begin{split}
   \sinh\big(\sum\limits^{\infty}\limits_{i=n-1}d_{i}\big)
   \cosh{\frac{ l_{n}}{2}}
     &=\sinh{\{\sum\limits^{\infty}\limits_{i=n-1}\big(c(\alpha_{i})+c(\alpha_{i+1})\big)\}} \cosh{\frac{ l_{n}}{2}}\\
    &=\sinh{\{\log{\coth{\frac{l_{n-1}}{4}}}+
    2\sum\limits^{\infty}\limits_{i=n}
    \log{\coth{\frac{l_{i}}{4}}}\}} \cosh{\frac{ l_{n}}{2}}\\
    &\leq \sinh{( 2\sum\limits^{\infty}\limits_{i=n-1}
    \log{\coth{\frac{l_{i}}{4}}})}\cosh{\frac{ l_{n}}{2}}.
     \end{split}
     \end{equation*}

     Note that the sequence $\{d_{0}, d_{1}, d_{2}, ...\}$ is completely determined by the sequences $\{l_{0},l_{1},l_{2}, ...\}$. Let $Y_{0}$ be a tight flute surface with $\sum|s_{i}|<\infty$ and the sequence $\{l_{0},l_{1},l_{2}, ...\}$ satisfying
     $\log{\coth{\frac{l_{i}}{4}}}=\frac{1}{2^{i}}$ for each $i\geq0$. Then
     \begin{equation*}
     d=\sum\limits_{i=0}\limits^{\infty}d_{i}=\log2+2\sum\limits^{\infty}\limits_{i=1}
    \log{\coth{\frac{l_{i}}{4}}}
    =\log2+2\sum\limits^{\infty}\limits_{i=1}\frac{1}{2^{i}}=\log2+2<\infty.
    \end{equation*}

     For $n=1$,   $\sinh\big(\sum\limits^{\infty}\limits_{i=n-1}d_{i}\big)
   \cosh{\frac{ l_{n}}{2}}=\sinh d
   \cosh{\frac{ l_{1}}{2}}<\infty$.
    For $n\geq2$, we get
    \begin{equation*}
    \begin{split}
    \sinh{(2\sum\limits^{\infty}\limits_{i=n-1}
    \log{\coth{\frac{l_{i}}{4}}})}\cosh{\frac{ l_{n}}{2}}
    &=\sinh{(2\sum\limits^{\infty}\limits_{i=n-1}
    \frac{1}{2^{i}})} \cosh{\frac{ l_{n}}{2}}\\
     &=\sinh{ \frac{8}{2^{n}}} \cosh{\frac{ l_{n}}{2}}\\
    &=\sinh{(8\log{\coth{\frac{l_{n}}{4}}})} \cosh{\frac{ l_{n}}{2}}.
    \end{split}
    \end{equation*}

   Observe that $l_{i}=4\arcoth{e^{\frac{1}{2^{i}}}}\rightarrow \infty$ as $i\rightarrow \infty$, $\coth{x}\rightarrow 1$, $\sech{x}\rightarrow 0$ as $x\rightarrow \infty$ and $(\coth{x})'=-\csch^{2}{x}$, $(\sech{x})'=-\sech{x}\tanh{x}$, we have
    \begin{equation*}
     \begin{split}
    \lim_{x\rightarrow \infty}\sinh({8\log{\coth{\frac{x}{4}}}})\cosh{\frac{x}{2}}&=
    \lim_{x\rightarrow \infty}\frac{\sinh(8{\log{\coth{\frac{x}{4}}}})}
    {\sech{\frac{x}{2}}}\\
    &=\lim_{x\rightarrow \infty} \frac{4\cosh{(8\log\coth{\frac{x}{4}})}\csch^{2}\frac{x}{4}}
    {\coth{\frac{x}{4}}\sech{\frac{x}{2}}\tanh{\frac{x}{2}}}\\
    &=\lim_{x\rightarrow \infty} 8(\coth{\frac{x}{2}})^{2}\\
    &=8.
    \end{split}
    \end{equation*}

   Hence, the surface $Y_{0}$  constructed above satisfies the condition \eqref{condition for incomplete}. By Proposition
   \ref{Ex1:tight flute surface }, it is incomplete and satisfies the geometric condition $(\star)$.
       \end{example}

   To construct some other examples of complete hyperbolic surfaces of infinite type which satisfy the geometric condition $(\star)$, we prove the following proposition.

       \begin{proposition}\label{EX1}
       Let $X_{0}$ be a flute surface of which all the pants holes that have not been glued are boundary components and the series $\sum d_{i}$ is divergent. If there exists a positive constant $L$ such that
       \begin{equation*}
       \sup\limits_{n\in \mathbb{N}}\{a_{n}, b_{n}, c_{n}, d_{n}\} \leq L,
       \end{equation*}
       then $X_{0}$ is complete and satisfies the geometric condition $(\star)$.
       Here $a_{n}$, $b_{n}$, $c_{n}$, $d_{n}$ satisfy that
       \begin{equation*}
       \begin{split}
       \cosh{a_{n}}&=\frac{\cosh{\alpha_{n+1}'}+\cosh{\alpha_{n}'}\cosh{\beta_{n}'}}
       {\sinh{\alpha_{n}'}\sinh{\beta_{n}'}},\\
       \cosh{b_{n}}&=\frac{\cosh{\alpha_{n}'}+\cosh{\alpha_{n+1}'}\cosh{\beta_{n}'}}
       {\sinh{\alpha_{n+1}'}\sinh{\beta_{n}'}},\\
       \sinh{c_{n}}&=\sinh{a_{n}}\cosh{\alpha_{n}'},\\
        \sinh{d_{n}}&=\sinh{b_{n}}\cosh{\alpha_{n+1}'},
          \end{split}
       \end{equation*}
     where $\alpha_{n}'=\frac{1}{2}\ell_{\alpha_{n}}(X_{0})$, $\beta_{n}'=\frac{1}{2}\ell_{\beta_{n}}(X_{0})$, $\alpha_{0}\cup\beta_{0}=\partial P_{0}\cap \partial X_{0}$, $\beta_{n}=\partial P_{n}\cap\partial X_{0}$ for $n\geq1$, as shown in Figure \ref{fig:EX1}.
       \end{proposition}

      \begin{proof}
      Since $\sum d_{i}$ diverges, it follows from Proposition \ref{Basmajian proposition1} that $\{\widetilde{\alpha_{n}}\}_{n=0}^{\infty}$ converges to a point of $\partial \mathbb{H}^{2}$, where $\widetilde{\alpha_{n}}$ is a lift of $\alpha_{n}$ in $\mathbb{H}^{2}$. Hence,
     each geodesic boundary component of $X_{0}$ is a simple closed geodesic and $X_{0}$ is complete.

       \begin{figure}[ht]
      \begin{tikzpicture}[scale=0.49]
\draw (-4.956,0.954) .. controls (-4.956,1.404) and (-3.046,1.404) .. (-3.046,0.954);
\draw[dashed] (-4.956,0.954) .. controls (-4.956,0.504) and (-3.046,0.504) .. (-3.046,0.954);
\draw (-4.956,0.954) .. controls (-4.956,0.052) and (-5.558,-0.55) .. (-6.354,-0.498);
\draw (-6.354,-0.498) .. controls (-6.652,-0.498) and (-6.652,-2.102) .. (-6.354,-2.05);
\draw (-6.354,-0.498) .. controls (-6.006,-0.498) and (-6.006,-2.05) .. (-6.354,-2.05);
\draw (-3.046,0.954) .. controls (-3.046,0.154) and (-2.546,-0.39) .. (-1.748,-0.492);
\draw (-6.354,-2.05) .. controls (-5.614,-2.052) and (-2.646,-2.002) .. (-1.748,-2.002);
\draw (-1.748,-0.492) .. controls (-1.398,-0.39) and (-1.398,-2.002) .. (-1.748,-2.002);
\draw[dashed] (-1.748,-0.492) .. controls (-2.098,-0.44) and (-2.048,-2.002) .. (-1.748,-2.002);
\draw (-1.748,-0.492) .. controls (-0.802,-0.34) and (-0.404,0.16) .. (-0.354,0.96);
\draw (-1.748,-2.002) .. controls (-0.702,-1.952) and (2.102,-1.952) .. (3.046,-1.952);
\draw (-0.354,0.96) .. controls (-0.354,1.462) and (1.596,1.462) .. (1.596,0.96);
\draw[dashed] (-0.354,0.96) .. controls (-0.354,0.458) and (1.596,0.458) .. (1.596,0.96);
\draw (1.596,0.96) .. controls (1.596,0.158) and (2.15,-0.386) .. (3.096,-0.436);
\draw (3.096,-0.436) .. controls (3.398,-0.436) and (3.398,-1.952) .. (3.046,-1.952);
\draw[dashed] (3.096,-0.436) .. controls (2.744,-0.436) and (2.794,-1.952) .. (3.046,-1.952);
\draw (3.096,-0.436) .. controls (3.844,-0.386) and (4.44,0.22) .. (4.49,1.018);
\draw (3.046,-1.952) .. controls (3.944,-1.952) and (7.116,-1.902) .. (7.76,-1.902);
\draw (4.49,1.018) .. controls (4.54,1.418) and (6.426,1.418) .. (6.376,1.068);
\draw[dashed] (4.49,1.018) .. controls (4.49,0.516) and (6.376,0.516) .. (6.376,1.068);
\draw (6.376,1.068) .. controls (6.426,0.17) and (7.116,-0.38) .. (7.76,-0.43);
\draw (7.76,-0.43) .. controls (8.114,-0.43) and (8.114,-1.902) .. (7.76,-1.902);
\draw [dashed](7.76,-0.43) .. controls (7.456,-0.43) and (7.456,-1.902) .. (7.76,-1.902);
\draw (7.76,-0.43) .. controls (8.656,-0.38) and (9.254,0.074) .. (9.304,1.03);
\draw (7.76,-1.902) .. controls (8.758,-1.902) and (11.862,-1.852) .. (13.706,-1.902);
\draw (9.304,1.03) .. controls (9.354,1.482) and (11.268,1.482) .. (11.218,1.03);
\draw[dashed] (9.304,1.03) .. controls (9.304,0.578) and (11.218,0.578) .. (11.218,1.03);
\draw (11.218,1.03) .. controls (11.268,0.024) and (11.86,-0.58) .. (13.706,-0.478);
\draw (12.556,-0.428) .. controls (12.86,-0.428) and (12.86,-1.852) .. (12.556,-1.852);
\draw [dashed](12.556,-0.428) .. controls (12.252,-0.428) and (12.252,-1.852) .. (12.556,-1.852);
\node at (3.03,-2.754) {$\alpha_{2}$};
\node at (-1.696,-2.756) {$\alpha_{1}$};
\node at (7.892,-2.654) {$\alpha_{3}$};
\node at (12.656,-2.552) {$\alpha_{4}$};
\node at (-6.194,-2.778) {$\alpha_{0}$};
\node at (-3.898,2.01) {$\beta_{0}$};
\node at (0.704,2) {$\beta_{1}$};
\node at (5.5,2.01) {$\beta_{2}$};
\node at (10.5,2.01) {$\beta_{3}$};
\node at (-4,-1) {$P_{0}$};
\node at (0.694,-0.898) {$P_{1}$};
\node at (5.5,-0.898) {$P_{2}$};
\node at (10.306,-0.898) {$P_{3}$};
\end{tikzpicture}
\caption{\small{A flute surface with geodesic boundary components $\alpha_{0}, \beta_{n}$ for $n\geq0$ and the series $\sum d_{i}$ divergent (this figure is a special case that the lengths of all $\alpha_{n}$ are equal).}}
            \label{fig:EX1}
		\end{figure}
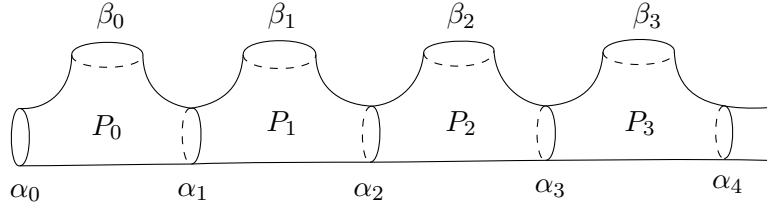

     Note that $P_{n}$ can be constructed by pasting two copies of the right-angled geodesic hexagon $H_{n}$ with pairwise non-adjacent edges $\frac{1}{2}\beta_{n}$, $\frac{1}{2}\alpha_{n}$, $\frac{1}{2}\alpha_{n+1}$ along the remaining three edges. Denote the vertices of $H_{n}$ not on the edge $\frac{1}{2}\beta_{n}$ by $A_{n}$, $C_{n}$, $D_{n}$, $B_{n}$ respectively in the anticlockwise order, as indicated in Figure \ref{fig:EX3}.

  \begin{figure}[ht]
\begin{tikzpicture}[scale=1]
\draw (-4.8771,0.7269) .. controls (-4.2317,0.6066) and (-3.6997,0.6066) .. (-3.0832,0.7269);
\draw (-4.8771,0.7269) .. controls (-5.0625,-0.05) and (-5.0993,-0.6565) .. (-6.1083,-0.9753);
\draw (-6.1083,-0.9753) .. controls (-5.9643,-1.3416) and (-5.8623,-1.7254) .. (-5.8626,-2.1834);
\draw (-3.0832,0.7269) .. controls (-2.8792,-0.05) and (-2.7083,-0.6749) .. (-1.7666,-0.9417);
\draw (-5.8626,-2.1834) .. controls (-4.2917,-2.0251) and (-3.6417,-2.0251) .. (-2.0417,-2.1751);
\draw (-1.7666,-0.9417) .. controls (-1.906,-1.2894) and (-2.0477,-1.7191) .. (-2.0417,-2.1751);
\node at (-3.9,0.9245) {$\frac{1}{2}\beta_{n}$};
\node at (-6.3626,-1.5251) {$\frac{1}{2}\alpha_{n}$};
\node at (-1.3,-1.4917) {$\frac{1}{2}\alpha_{n+1}$};
\node at (-3.9583,-0.8749) {$H_{n}$};
\node at (-6.3749,-0.8753) {$A_{n}$};
\node at (-5.8749,-2.4583) {$C_{n}$};
\node at (-2,-2.4583) {$D_{n}$};
\node at (-1.5,-0.8753) {$B_{n}$};
\end{tikzpicture}
\caption{\small{A right-angled hexagon $H_{n}$.}}
            \label{fig:EX3}
		\end{figure}
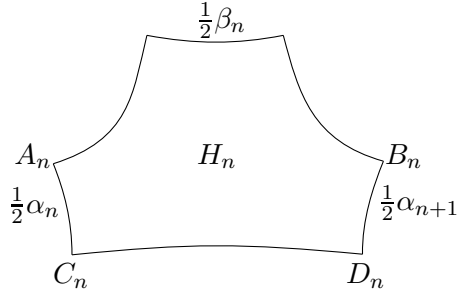

     By Lemma \ref{Lemma for polygon}, for any $y\in H_{n}$,
     \begin{equation*}
     \rho(y,\frac{1}{2}\beta_{n}) =\max{\{ \rho(A_{n},\frac{1}{2}\beta_{n}), \rho(B_{n},\frac{1}{2}\beta_{n}), \rho(C_{n},\frac{1}{2}\beta_{n}), \rho(D_{n},\frac{1}{2}\beta_{n}\}}.
     \end{equation*}

     For simplicity, denote
     $\frac{1}{2}\ell_{\alpha_{n}}(X_{0})=\alpha_{n}'$, $\frac{1}{2}\ell_{\beta_{n}}(X_{0})=\beta_{n}'$,
      $\rho(A_{n},\frac{1}{2}\beta_{n})=a_{n}$,
     $\rho(B_{n},\frac{1}{2}\beta_{n})=b_{n}$, $\rho(C_{n},\frac{1}{2}\beta_{n})=c_{n}$,
     $\rho(D_{n},\frac{1}{2}\beta_{n})=d_{n}$.

    Then we have
    \begin{equation*}
       \begin{split}
     \cosh{a_{n}}&=\frac{\cosh{\alpha_{n+1}'}+\cosh{\alpha_{n}'}\cosh{\beta_{n}'}}
       {\sinh{\alpha_{n}'}\sinh{\beta_{n}'}},\\
       \cosh{b_{n}}&=\frac{\cosh{\alpha_{n}'}+\cosh{\alpha_{n+1}'}\cosh{\beta_{n}'}}
       {\sinh{\alpha_{n+1}'}\sinh{\beta_{n}'}},\\
       \sinh{c_{n}}&=\sinh{a_{n}}\cosh{\alpha_{n}'},\\
        \sinh{d_{n}}&=\sinh{b_{n}}\cosh{\alpha_{n+1}'}.
          \end{split}
       \end{equation*}

       For any point $x\in X_{0}$, there exists an integer $N\geq0$ such that $x\in P_{N}$. In particular, $x\in H_{N}$. By assumption, we obtain that
       \begin{equation*}
        \rho(x,\beta_{N})\leq\max\{a_{N}, b_{N}, c_{N}, d_{N}\}\leq L.
       \end{equation*}
       Therefore, $ \rho(x,\partial X_{0})\leq  \rho(x,\beta_{N})\leq L$, which implies that $X_{0}$ satisfies the geometric condition $(\star)$.
       \end{proof}

       \begin{example}\label{EX2}
        Let $X_{0}'$ be a flute surface. Let $\alpha_{n}'$, $\beta_{n}'$ denote the lengths of the corresponding simple geodesic segments in Proposition \ref{EX1}. Suppose that $\alpha_{n}'$, $\beta_{n}'$ satisfy the following conditions:

    (1) $\alpha_{n}'=l_{0}$ for all $n\geq0$;

    (2) $\{\beta_{n}'\}_{n=0}^{\infty}$ is a strictly increasing sequence of positive numbers such that
    \begin{center}
    $\beta_{n}'\rightarrow\infty$ as $n\rightarrow \infty$.
    \end{center}

    We claim that $X_{0}'$ is complete and satisfies the geometric condition $(\star)$.
     Indeed, using the same notations $a_{n}, b_{n}, c_{n}, d_{n}$ as in Proposition \ref{EX1}, by direct computation, we obtain that
       \begin{equation*}
       \begin{split}
       \cosh{a_{n}}= \cosh{b_{n}}&=\frac{\cosh{l_{0}}+\cosh{\beta_{n}'}\cosh{l_{0}}}
       {\sinh{\beta_{n}'}\sinh{l_{0}}}\\
      & = \csch{\beta_{n}'}\coth{l_{0}}+\coth{\beta_{n}'}\coth{l_{0}}\\
       &\leq 2\coth{\beta_{n}'}\coth{l_{0}},
       \end{split}
       \end{equation*}

      \begin{equation*}
      \sinh{c_{n}}=\sinh{d_{n}}=\sinh{a_{n}}\cosh{l_{0}}
      \leq \cosh{a_{n}}\cosh{l_{0}}.
      \end{equation*}

     Note that the sequence $\{\coth{\beta_{n}'}\}$ strictly decreases, we have that
     \begin{equation*}
     \sup\limits_{n\in \mathbb{N}}\{a_{n}, b_{n}, c_{n}, d_{n}\} \leq L,
      \end{equation*}
      for a constant $L>0$. Moreover, it is easy to see that $\sum d_{i}=\infty$. By Proposition \ref{EX1}, $X_{0}'$ is complete and satisfies the geometric condition $(\star)$, as indicated in Figure \ref{fig:X_{0}'}.

  \begin{figure}[ht]
\begin{tikzpicture}[scale=0.49]
\draw (-4.854,1.056) .. controls (-4.854,1.506) and (-3.25,1.506) .. (-3.25,1.056);
\draw[dashed] (-4.854,1.056) .. controls (-4.854,0.606) and (-3.25,0.606) .. (-3.25,1.056);
\draw (-4.854,1.056) .. controls (-4.854,0.052) and (-5.252,-0.55) .. (-6.252,-0.6);
\draw (-6.252,-0.6) .. controls (-6.55,-0.6) and (-6.652,-1.898) .. (-6.252,-1.948);
\draw (-6.252,-0.6) .. controls (-5.904,-0.6) and (-5.904,-1.948) .. (-6.252,-1.948);
\draw (-3.25,1.056) .. controls (-3.25,0.052) and (-2.75,-0.594) .. (-1.85,-0.594);
\draw (-6.252,-1.948) .. controls (-5.002,-1.95) and (-2.85,-1.9) .. (-1.85,-1.9);
\draw (-1.85,-0.594) .. controls (-1.5,-0.594) and (-1.5,-1.9) .. (-1.85,-1.9);
\draw[dashed] (-1.85,-0.594) .. controls (-2.2,-0.644) and (-2.15,-1.9) .. (-1.85,-1.9);
\draw (-1.85,-0.594) .. controls (-0.904,-0.544) and (-0.302,-0.146) .. (-0.252,0.756);
\draw (-1.85,-1.9) .. controls (-0.906,-1.85) and (2.51,-1.85) .. (3.454,-1.85);
\draw (-0.252,0.756) .. controls (-0.252,1.156) and (2.004,1.156) .. (2.004,0.756);
\draw[dashed] (-0.252,0.756) .. controls (-0.252,0.356) and (2.004,0.356) .. (2.004,0.756);
\draw (2.004,0.756) .. controls (2.004,-0.046) and (2.558,-0.59) .. (3.504,-0.64);
\draw (3.504,-0.64) .. controls (3.806,-0.64) and (3.806,-1.85) .. (3.454,-1.85);
\draw[dashed] (3.504,-0.64) .. controls (3.152,-0.64) and (3.202,-1.85) .. (3.454,-1.85);
\draw (3.504,-0.64) .. controls (4.15,-0.59) and (4.542,-0.392) .. (4.694,0.304);
\draw (3.454,-1.85) .. controls (4.25,-1.85) and (8.442,-1.8) .. (9.086,-1.8);
\draw (4.694,0.304) .. controls (4.744,0.704) and (7.854,0.704) .. (7.804,0.354);
\draw[dashed] (4.694,0.304) .. controls (4.694,-0.096) and (7.804,-0.096) .. (7.804,0.354);
\draw (7.804,0.354) .. controls (8.058,-0.442) and (8.442,-0.584) .. (9.086,-0.634);
\draw (9.086,-0.634) .. controls (9.44,-0.634) and (9.44,-1.8) .. (9.086,-1.8);
\draw [dashed](9.086,-0.634) .. controls (8.68,-0.634) and (8.782,-1.8) .. (9.086,-1.8);
\draw (9.086,-0.634) .. controls (9.778,-0.584) and (10.07,-0.436) .. (10.324,-0.092);
\draw (9.086,-1.8) .. controls (9.982,-1.8) and (15.228,-1.75) .. (17.072,-1.8);
\draw (10.324,-0.092) .. controls (10.374,0.258) and (14.736,0.258) .. (14.686,-0.092);
\draw[dashed] (10.324,-0.092) .. controls (10.324,-0.34) and (14.686,-0.34) .. (14.686,-0.092);
\draw (14.686,-0.092) .. controls (15.042,-0.486) and (15.226,-0.682) .. (17.072,-0.682);
\draw (15.922,-0.632) .. controls (16.226,-0.632) and (16.226,-1.75) .. (15.922,-1.75);
\draw [dashed](15.922,-0.632) .. controls (15.516,-0.632) and (15.516,-1.75) .. (15.922,-1.75);
\node at (-3.902,1.956) {$\beta_{0}$};
\node at (1.104,1.756) {$\beta_{1}$};
\node at (6.398,1.254) {$\beta_{2}$};
\node at (12.634,0.708) {$\beta_{3}$};
\node at (-6.252,-2.55) {$\alpha_{0}$};
\node at (-1.798,-2.45) {$\alpha_{1}$};
\node at (3.506,-2.45) {$\alpha_{2}$};
\node at (9.188,-2.45) {$\alpha_{3}$};
\node at (15.972,-2.4) {$\alpha_{4}$};
\node at (-4,-1) {$P_{0}$};
\node at (1.01,-1) {$P_{1}$};
\node at (6.306,-1) {$P_{2}$};
\node at (12.684,-0.99) {$P_{3}$};
\end{tikzpicture}
\caption{\small{The flute surface $X_{0}'$ in Example \ref{EX2}.}}
            \label{fig:X_{0}'}
		\end{figure}
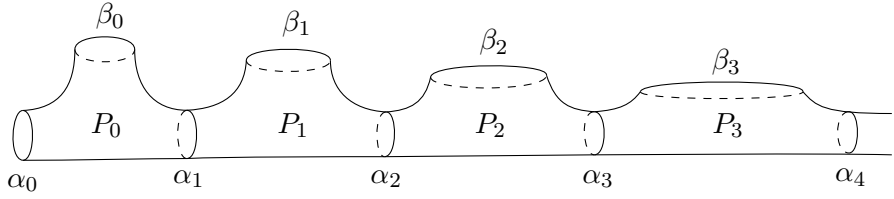
       \end{example}

\begin{example}\label{ex: horodisks}
We construct a hyperbolic surface $X_1$ of infinite type, which satisfies the geometric condition $(\star)$ with a removable set $A$ for $\Gamma_{X_1}$ consisting of horodisks.

Let $Z_{0}$ be a flute surface with $\alpha_{n}=\beta_{n}=1$ for $n\geq0$, where $\alpha_{n}$ and $\beta_{n}$ denote the same simple closed geodesics as in Proposition \ref{EX1}. Then we construct $X_{1}$ by inserting a generalized hyperbolic pair of pants with one cusp and two boundary components $\alpha^{l}_{n}, \alpha^{r}_{n}$ of lengths 1 along both sides of $\alpha_{n}$ for $n\geq1$, as shown in Figure \ref{fig: removable set is cusp}. Let $A$ be the removable set for $\Gamma_{X_1}$, whose projection $\pi(A)$ on $X_1$ under $\Gamma_{X_1}$ is a disjoint union of open cusps with the boundary $\gamma'_n$ of length 1. Note that the geodesics $\alpha^{l}_{n}$, $\alpha^{r}_{n}$, $\beta_{n}$ and $\gamma'_n$ have the same length 1 for all $n\geq1$. It is not hard to see that $X_1\setminus \pi(A)$ is contained in a bounded distance of $\partial X_1$. This implies that $X_1$ satisfies the geometric condition $(\star)$.

\begin{figure}[ht]
\begin{tikzpicture}[scale=0.68]
\draw (-4,0) -- (11.5,0);
\draw (-4,1) .. controls (-3,1) and (-3,1.5) .. (-3,2.3);
\draw (-2,2.3) .. controls (-2,1.5) and (-2,1) .. (-1,1);
\draw (1,1) .. controls (2,1) and (2,1.5) .. (2,2.3);
\draw (3,2.3) .. controls (3,1.5) and (3,1) .. (4,1);
\draw (1,1) .. controls (2,1) and (2,1.5) .. (2,2.3);
\draw (3,2.3) .. controls (3,1.5) and (3,1) .. (4,1);
\draw (6,1) .. controls (7,1) and (7,1.5) .. (7,2.3);
\draw (8,2.3) .. controls (8,1.5) and (8,1) .. (9,1);
\draw (-1,1) .. controls (0,1) and (0,5.3) .. (0,5.3);
\draw (0,5.3) .. controls (0,5.3) and (0,1) .. (1,1);
\draw (4,1) .. controls (5,1) and (5,5.3) .. (5,5.3);
\draw (5,5.3) .. controls (5,5.3) and (5,1) .. (6,1);
\draw (9,1) .. controls (10,1) and (10,5.3) .. (10,5.3);
\draw (10,5.3) .. controls (10,5.3) and (10,1) .. (11,1) node (v1) {};
\draw (-4,1) .. controls (-4.3,1) and (-4.3,0) .. (-4,0);
\draw  (-4,1) node (v2) {} .. controls (-3.7,1) and (-3.7,0) .. (-4,0) node (v3) {};
\draw [dashed] (-1,1) .. controls (-1.3,1) and (-1.3,0) .. (-1,0);
\draw  (-1,1) .. controls (-0.7,1) and (-0.7,0) .. (-1,0);
\draw [dashed]  (1,1) .. controls (0.7,1) and (0.7,0) .. (1,0);
\draw  (1,1) .. controls (1.3,1) and (1.3,0) .. (1,0);
\draw [dashed] (4,1) .. controls (3.7,1) and (3.7,0) .. (4,0);
\draw  (4,1) .. controls (4.3,1) and (4.3,0) .. (4,0);
\draw [dashed] (6,1) .. controls (5.7,1) and (5.7,0) .. (6,0);
\draw  (6,1) .. controls (6.3,1) and (6.3,0) .. (6,0);
\draw [dashed] (9,1) .. controls (8.7,1) and (8.7,0) .. (9,0);
\draw  (9,1) .. controls (9.3,1) and (9.3,0) .. (9,0);
\draw [dashed] (11,1) .. controls (10.7,1) and (10.7,0) .. (11,0);
\draw  (11,1) node (v4) {} .. controls (11.3,1) and (11.3,0) .. (11,0);
\draw (-3,2.3) .. controls (-3,2.6) and (-2,2.6) .. (-2,2.3);
\draw (-3,2.3) .. controls (-3,2) and (-2,2) .. (-2,2.3);
\draw (2,2.3) .. controls (2,2.6) and (3,2.6) .. (3,2.3);
\draw (2,2.3) .. controls (2,2) and (3,2) .. (3,2.3);
\draw (7,2.3) .. controls (7,2.6) and (8,2.6) .. (8,2.3);
\draw(7,2.3) .. controls (7,2) and (8,2) .. (8,2.3);
\draw (11,1) .. controls (11,1) and (11,1) .. (11.5,1);
\node at (-4,-0.5) {$\alpha_0$};
\node at (-1,-0.5) {$\alpha_1^l$};
\node at (1,-0.5) {$\alpha_1^r$};
\node at (4,-0.5) {$\alpha_2^l$};
\node at (6,-0.5) {$\alpha_2^r$};
\node at (9,-0.5) {$\alpha_3^l$};
\node at (11,-0.5) {$\alpha_3^r$};
\node at (-2.5,2.9) {$\beta_0$};
\node at (2.5,2.9) {$\beta_1$};
\node at (7.5,2.9) {$\beta_2$};
\draw [dashed](-0.5,1.5) .. controls (-0.5,1.8) and (0.5,1.8) .. (0.5,1.5);
\draw (-0.5,1.5) .. controls (-0.5,1.2) and (0.5,1.2) .. (0.5,1.5);
\draw [dashed](4.5,1.5) .. controls (4.5,1.8) and (5.5,1.8) .. (5.5,1.5);
\draw (4.5,1.5) .. controls (4.5,1.2) and (5.5,1.2) .. (5.5,1.5);
\draw [dashed](9.5,1.5) .. controls (9.5,1.8) and (10.5,1.8) .. (10.5,1.5);
\draw (9.5,1.5) .. controls (9.5,1.2) and (10.5,1.2) .. (10.5,1.5);
\node at (0,0.7761) {$\gamma'_1$};
\node at (5,0.7761) {$\gamma'_2$};
\node at (10,0.7761) {$\gamma'_3$};
\end{tikzpicture}
\caption{\small{The flute surface $X_1$ in Example \ref{ex: horodisks} for the case that the removable set for $\Gamma_{X_1}$ is a disjoint union of horodisks.}}
            \label{fig: removable set is cusp}
\end{figure}
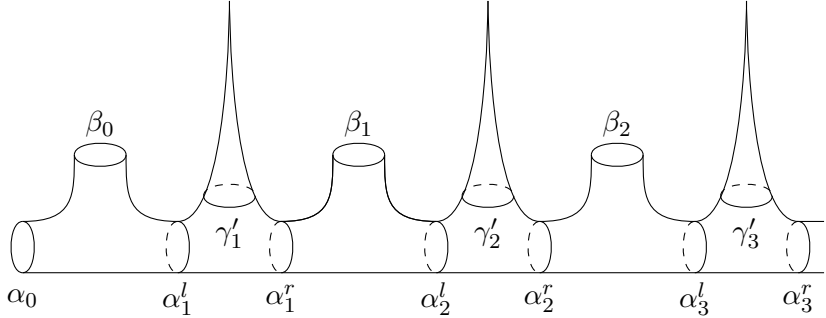
       \end{example}

\begin{example}\label{ex: neighbourhoods of a geodesic}
We construct a hyperbolic surface $X_2$ of infinite type, which satisfies the geometric condition $(\star)$ with a removable set $A$ for $\Gamma_{X_2}$ consisting of neighbourhoods of complete geodesics in $\mathbb{H}^2$ whose radii tend to infinity.

Let $Z_{0}$ be a flute surface with $\alpha_{n}=\beta_{n}=1$ for $n\geq0$, where $\alpha_{n}$ and $\beta_{n}$ denote the same simple closed geodesics as in Proposition \ref{EX1}. Then we construct $X_{2}$ by inserting a hyperbolic pair of pants with two boundary components $\alpha^{l}_{n}, \alpha^{r}_{n}$ of lengths 1 and the other boundary component $\gamma_n$ of length $\frac{1}{2n}$ along both sides of $\alpha_{n}$ for $n\geq1$, as indicated in Figure \ref{fig: removable set is annulus}. Let $A$ be a removable set for $\Gamma_{X_2}$, whose projection $\pi(A)$ on $X_2$ under $\Gamma_{X_2}$ is a disjoint union of relatively open annuli with two boundary components $\gamma_n$ and $\gamma'_n$, where $\gamma'_n$ is an equidistant curve of the geodesic $\gamma_n$ for a distance $r_n=\arcsinh \{1/\sinh{(\frac{1}{2}\ell_{\gamma_n}(X_2))}\}=\arcsinh \{1/\sinh{(\frac{1}{4n})}\}$ (this is ensured by the collar lemma, see \cite{Buser}). Note that $\gamma'_{n}$ is not a geodesic and the relation between $\ell_{\gamma'_{n}}(X_2)$ and $\ell_{\gamma_{n}}(X_2)$ is given by the following formula (see \cite[Example 1.3.2]{Buser}):
\begin{equation*}
\ell_{\gamma'_{n}}(X_2) = \ell_{\gamma_{n}}(X_2)\cosh d_{X_2}(\gamma_n,\gamma'_n),
\end{equation*}
where $d_{X_2}(\gamma_n,\gamma'_n)$ is the distance between $\gamma_n$ and $\gamma'_n$ on $X_2$. By computation,
\begin{equation*}
\ell_{\gamma'_n}(X_2)=\frac{1}{2n}\cosh {r_n}
=\sqrt{\frac{1}{4{n}^2}+\frac{4}{(4n\sinh(\frac{1}{4n}))^2}}\rightarrow 2,
\end{equation*}
as $n\rightarrow \infty$. Hence, there exists $n_0\in\mathbb{N}$ and $\epsilon_0>0$ such that
\begin{equation*}
2-\epsilon_0<\ell_{X_2}(\gamma'_n)\leq 2+\epsilon_0,
\end{equation*}
for all $n\geq n_0$. Combined with the fact that the geodesics $\alpha^{l}_{n}$, $\alpha^{r}_{n}$ and $\beta_{n}$ have the same length 1 for all $n\geq1$, it follows that $X_2\setminus \pi(A)$ is contained in a bounded distance of $\partial X_2$. This implies that $X_2$ satisfies the geometric condition $(\star)$. We obtain the desired surface $X_2$.
\begin{figure}[ht]
\begin{tikzpicture}[scale=0.6]
\draw (-5.2,0) -- (11.8,0);
\draw (-5.2,1) .. controls (-4.2,1) and (-4.2,1.4) .. (-4.2,2.1);
\draw (-3.2,2.1) .. controls (-3.2,1.4) and (-3.2,1) .. (-2.2,1);
\draw (0.5,1) .. controls (1.5,1) and (1.5,1.4) .. (1.5,2.1);
\draw (2.5,2.1) .. controls (2.5,1.4) and (2.5,1) .. (3.5,1);
\draw (0.5,1) .. controls (1.5,1) and (1.5,1.4) .. (1.5,2.1);
\draw (2.5,2.1) .. controls (2.5,1.4) and (2.5,1) .. (3.5,1);
\draw (6,1) .. controls (7,1) and (7,1.4) .. (7,2.1);
\draw (8,2.1) .. controls (8,1.4) and (8,1) .. (9,1);
\draw (-2.2,1) .. controls (-1.1,1) and (-1.1,2.9) .. (-1.1,3.1);
\draw (-0.4,3.1) .. controls (-0.4,2.9) and (-0.5,1) .. (0.5,1);
\draw (3.5,1) .. controls (4.5,1) and (4.5,4.8) .. (4.5,4.8);
\draw (4.8,4.8) .. controls (4.8,4.8) and (4.9,1) .. (6,1);
\draw (9,1) .. controls (10,1) and (10,7.6) .. (10,7.6);
\draw (10.2,7.6) .. controls (10.2,7.6) and (10.4,1) .. (11.3,1) node (v1) {};
\draw (-5.2,1) .. controls (-5.5,1) and (-5.5,0) .. (-5.2,0);
\draw  (-5.2,1) node (v2) {} .. controls (-4.9,1) and (-4.9,0) .. (-5.2,0) node (v3) {};
\draw [dashed] (-2.2,1) .. controls (-2.5,1) and (-2.5,0) .. (-2.2,0);
\draw  (-2.2,1) .. controls (-1.9,1) and (-1.9,0) .. (-2.2,0);
\draw [dashed]  (0.5,1) .. controls (0.2,1) and (0.2,0) .. (0.5,0);
\draw  (0.5,1) .. controls (0.8,1) and (0.8,0) .. (0.5,0);
\draw [dashed] (3.5,1) .. controls (3.2,1) and (3.2,0) .. (3.5,0);
\draw  (3.5,1) .. controls (3.8,1) and (3.8,0) .. (3.5,0);
\draw [dashed] (6,1) .. controls (5.7,1) and (5.7,0) .. (6,0);
\draw  (6,1) .. controls (6.3,1) and (6.3,0) .. (6,0);
\draw [dashed] (9,1) .. controls (8.7,1) and (8.7,0) .. (9,0);
\draw  (9,1) .. controls (9.3,1) and (9.3,0) .. (9,0);
\draw [dashed] (11.3,1) .. controls (11,1) and (11,0) .. (11.3,0);
\draw  (11.3,1) node (v4) {} .. controls (11.6,1) and (11.6,0) .. (11.3,0);
\draw (-4.2,2.1) .. controls (-4.2,2.4) and (-3.2,2.4) .. (-3.2,2.1);
\draw (-4.2,2.1) .. controls (-4.2,1.8) and (-3.2,1.8) .. (-3.2,2.1);
\draw (1.5,2.1) .. controls (1.5,2.4) and (2.5,2.4) .. (2.5,2.1);
\draw (1.5,2.1) .. controls (1.5,1.8) and (2.5,1.8) .. (2.5,2.1);
\draw (7,2.1) .. controls (7,2.4) and (8,2.4) .. (8,2.1);
\draw(7,2.1) .. controls (7,1.8) and (8,1.8) .. (8,2.1);
\draw (11.3,1) .. controls (11.3,1) and (11.3,1) .. (11.8,1);
\node at (-4.5,-0.5) {$\alpha_0$};
\node at (-1.5,-0.5) {$\alpha_1^l$};
\node at (0.5,-0.5) {$\alpha_1^r$};
\node at (3.5,-0.5) {$\alpha_2^l$};
\node at (6,-0.5) {$\alpha_2^r$};
\node at (9,-0.5) {$\alpha_3^l$};
\node at (11.3,-0.5) {$\alpha_3^r$};
\node at (-3.7,2.7) {$\beta_0$};
\node at (2,2.7) {$\beta_1$};
\node at (7.5,2.7) {$\beta_2$};
\draw (-1.1,3.1) .. controls (-1.1,3.3) and (-0.4,3.3) .. (-0.4,3.1);
\draw (-1.1,3.1) .. controls (-1.1,2.9) and (-0.4,2.9) .. (-0.4,3.1);
\draw (4.5,4.9) .. controls (4.5,5) and (4.8,5) .. (4.8,4.9);
\draw (4.5,4.9) .. controls (4.5,4.7) and (4.8,4.7) .. (4.8,4.9);
\draw (10,7.6) .. controls (10,7.7) and (10.2,7.7) .. (10.2,7.6);
\draw (10,7.6) .. controls (10,7.5) and (10.2,7.5) .. (10.2,7.6);
\node at (-0.7,3.6) {$\gamma_1$};
\node at (4.7,5.3) {$\gamma_2$};
\node at (10.1,8.1) {$\gamma_3$};
\draw[dashed] (-1.25,1.9) .. controls (-1.25,2.1) and (-0.3,2.1) .. (-0.3,1.9);
\draw (-1.25,1.9) .. controls (-1.25,1.7) and (-0.3,1.7) .. (-0.3,1.9);
\draw[dashed] (4.15,1.8) .. controls (4.15,2) and (5.28,2) .. (5.28,1.8);
\draw (4.15,1.8) .. controls (4.15,1.6) and (5.287,1.6) .. (5.28,1.8);
\draw[dashed] (9.45,1.6) .. controls (9.45,1.8) and (10.85,1.8) .. (10.85,1.6);
\draw (9.45,1.6) .. controls (9.45,1.4) and (10.85,1.4) .. (10.85,1.6);
\node at (-0.8,1.3) {$\gamma'_1$};
\node at (4.7,1.2) {$\gamma'_2$};
\node at (10.2,1) {$\gamma'_3$};
\end{tikzpicture}
\caption{\small{The flute surface $X_2$ in Example \ref{ex: neighbourhoods of a geodesic} for the case that the removable set for $\Gamma_{X_2}$ is a disjoint union of neighbourhoods of complete geodesics in $\mathbb{H}^2$.}}
            \label{fig: removable set is annulus}
\end{figure}
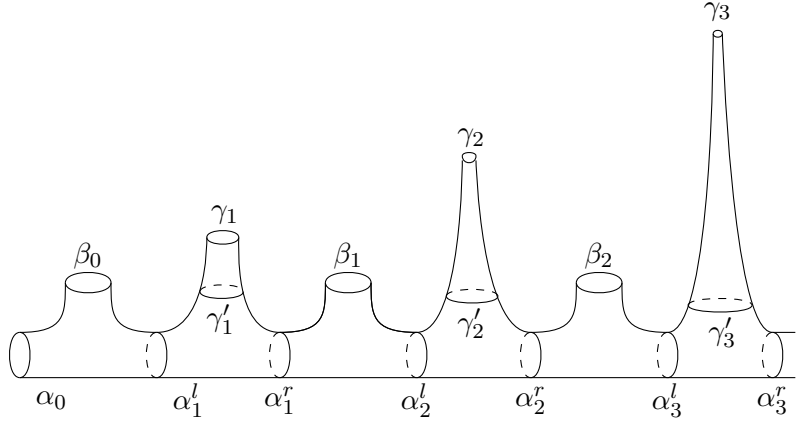
       \end{example}

\begin{example}\label{ex: hyperbolic disks}
    We construct a hyperbolic surface $X_3$ of infinite type, which satisfies the geometric condition $(\star)$ with a removable set $A$ for $\Gamma_{X_3}$ consisting of hyperbolic disks whose radii tend to infinity.
	
    For each integer $n \geq3$, let $T_n$ be a trirectangle with one angle $\theta_n=\pi/n$ and three right angles. Denote the four consecutive edges of $T_n$ by $\alpha_n$, $b_n$, $a_n$ and $\beta_n$. Let $\theta_n$ be the angle bounded by $\alpha_n$ and $\beta_n$ (see Figure \ref{fig:ex_YL}). For convenience, we also denote the lengths of $\alpha_n$, $b_n$, $a_n$, $\beta_n$ by $\alpha_n$, $b_n$, $a_n$, $\beta_n$, respectively.
    Note that $T_n$ can be uniquely (up to isometries) determined by $\theta_n$ and $\alpha_n$. We choose $\alpha_n$ such that $\sin{\theta_n}\cosh{\alpha_n}=2$ for all $n\geq 3$. By the formula $\cosh{a_n}=\cosh{\alpha_n}\sin{\theta_n}$ for a trirectangle $T_n$ (see \cite{Buser}), $\cosh{a_n}=2$ for all $n\geq 3$.
	
	We claim that in each $T_n$ we have
	\begin{equation}
		\label{eq:ex_YL_1}
		\beta_n < \alpha_n.
	\end{equation}
	Indeed, by the formulae for a trirectangle $T_n$ (see \cite{Buser}), we get
     \begin{equation}\label{eq:estimate for the angle}
       \begin{split}
        \cos \theta_n&=\sinh a_n \sinh b_n,\\
        \frac{\cosh a_n}{\cosh b_n} &= \frac{\cosh \alpha_n}{\cosh \beta_n}.
        \end{split}
        \end{equation}
   Hence,
	\begin{equation*}
		\sinh b_n
		= \frac{\cos \theta_n}{\sinh a_n}
        = \frac{\cos \theta_n}{\sqrt{\cosh^2 a_n - 1}}
		= \frac{\cos \theta_n}{\sqrt{3}}
		< 1<\sqrt{3}
		= \sinh a_n.
	\end{equation*}
	Combined with \eqref{eq:estimate for the angle}, we have $\beta_n < \alpha_n$ for all $n\geq 3$.
	
	\begin{figure}[ht]
		\begin{tikzpicture}[scale=0.38]
			\coordinate (v0) at (0,0) {};
			\coordinate (v1) at (8.18,0) {};
			\coordinate (v2) at (0,8.18) {};
			\coordinate (v3) at (-8.18,0) {};
			\coordinate (v4) at (0,-8.18) {};
			
			\coordinate (v1') at (5.58,5.58) {};
			\coordinate (v2') at (-5.58,5.58) {};
			\coordinate (v3') at (-5.58,-5.58) {};
			\coordinate (v4') at (5.58,-5.58) {};
			
			\coordinate (v1'') at (9,3) {};
			\coordinate (v2'') at (-3,9) {};
			\coordinate (v3'') at (-9,-3) {};
			\coordinate (v4'') at (3,-9) {};
			
			\draw  (v0) -- (v1) (v0) -- (v2) (v0) -- (v3) (v0) -- (v4);
			\draw  (v0) -- (v1') (v0) -- (v2') (v0) -- (v3') (v0) -- (v4');
			
			\draw[dashed] (v1') arc (150:210:11.2);
			\draw[dashed] (v2') arc (-120:-60:11.2);
			\draw[dashed] (v3') arc (-30:30:11.2);
			\draw[dashed] (v4') arc (60:120:11.2);
			
			\draw (v1'') arc (150:210:6);
			\draw (v2'') arc (-120:-60:6);
			\draw (v3'') arc (-30:30:6);
			\draw (v4'') arc (60:120:6);
			
			\draw (v1'') arc (-120:-150:16.4);
			\draw (v2'') arc (-30:-60:16.4);
			\draw (v3'') arc (60:30:16.4);
			\draw (v4'') arc (150:120:16.4);
			
			\draw  (v0) ellipse (7.92 and 7.92);
			
			\draw (1,0) arc (0:45:1);
			\node() at (8,5) {$a_n$};
			\node () at (9,1.5) {$b_n$};
			\node () at (6,-.5) {$\alpha_n$};
			\node() at (2,3) {$\beta_n$};
			\node () at (5,2) {$L_n$};
			\node () at (1.5,0.5) {$\theta_n$};
		\end{tikzpicture}
		\caption{\small{The trirectangle $T_n$ in the right-angled $2n$-polygon $P_n$ for $n=4$.}}
		\label{fig:ex_YL}
	\end{figure}
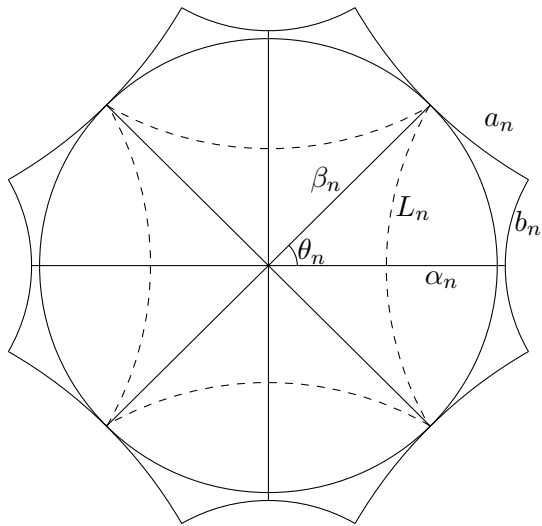
	
	Denote by $L_n$ the length of the geodesic perpendicular to $\alpha_n$ through the intersection point of $\beta_n$ and $a_n$, as shown in Figure \ref{fig:ex_YL}. By \eqref{eq:ex_YL_1} and a formula for a right-angled triangle (see \cite{Buser}), we obtain that
	\begin{equation*}
		\sinh L_n
		= \sin \theta_n \sinh \beta_n
		< \sin \theta_n \sinh \alpha_n
		< \sin \theta_n \cosh \alpha_n
		= \cosh a_n
		= 2,
	\end{equation*}
	for all $n\geq3$. Therefore,
	\begin{equation}\label{eq:ex_YL_2}
		L_n < \arcsinh 2,
	\end{equation}
	for all $n\geq3$.

      Now we construct a right-angled $2 n$-polygon $P_n$ by gluing $2 n$ copies of $T_n$ along the edges of the common lengths $\alpha_n$ and $\beta_n$ alternately (see Figure \ref{fig:ex_YL}). Then $P_n$ has $n$ sides of lengths $2 a_n$ and the other $n$ sides of lengths $2 b_n$. Denote the $2 b_n$-length sides of $P_n$ by $e_1$, $e_2$, ..., $e_{n}$ in the anti-clockwise order. Let $B_n$ be the maximal embedding open hyperbolic disk in $P_n$ whose center is the center of $P_n$.
     	
	By the inequality \eqref{eq:ex_YL_1} and the construction of $P_n$, the radius of $B_n$ is $\beta_n$. Note that $\theta_n = \pi / n$, $\cosh a_n = \cosh \alpha_n \sin \theta_n \equiv 2$, it follows that $\alpha_n\rightarrow \infty$ as $n\rightarrow \infty$.
	Combined with the formula for a trirectangle $T_n$ (see \cite{Buser}) that $\cosh \beta_n \sinh a_n = \sinh \alpha_n$, the radius $\beta_n$ of $B_n$ tends to infinity as $n\rightarrow\infty$.

	Take another copy $P_n'$ of $P_n$ and denote by $e_1'$, $e_2'$, ..., $e_{n}'$ the $2 b_n$-length sides corresponding to the sides $e_1$, $e_2$, ..., $e_{n}$ of $P_n$. Let $B_n'$ be the maximal embedded open hyperbolic disk in $P_n'$. Denote by $S_n$ the surface obtained by gluing $P_n$ and  $P_n'$ along $e_i$ and $e_i'$ for $i = 1, 2, ..., n$.
	Then $S_n$ is a hyperbolic surface with $n$ consecutive boundary components $\gamma^{(n)}_1$, $\gamma^{(n)}_2$, ..., $\gamma^{(n)}_{n}$ of the same length $4 a_n = 4 \arccosh 2>0$ for all $n\geq3$.
	The hyperbolic disks $B_n$ and $B_n'$ are disjoint from each other and tangent to $\gamma^{(n)}_i$ for $i = 1, 2, ..., n$. Moreover, they have the same radius $\beta_n$, which tends to infinity as $n\rightarrow\infty$.

	It is not hard to see that $S_n \setminus ( B_n \cup B_n' )$ is within the distance $L_n$ of the boundary $\partial S_n$ of $S_n$. By \eqref{eq:ex_YL_2}, we have
	\begin{equation}
		\label{eq:ex_YL_3}
		d (p, \partial S_n) < \arcsinh 2,
	\end{equation}
	for all $p\in S_n \setminus ( B_n \cup B_n' )$ and all $n\geq 3$.
	
	We construct $X_3$ by pasting the boundary component $\gamma^{(n)}_1$ of $S_{n}$ and the boundary component $\gamma^{(n+1)}_{n}$ of $S_{n+1}$ one by one for $n\geq 3$ (see Figure \ref{fig:removable set for hyperbolic disks}). Let $A$ be a removable set for $\Gamma_{X_3}$ whose projection $\pi(A)$ on $X_3$ under $\Gamma_{X_3}$ is a disjoint union of hyperbolic disks $B_n$ and $B'_n$ over $n\geq 3$.
By inequality \eqref{eq:ex_YL_3} and the construction of $X_3$, it is not hard to see that $X_3 \setminus \pi(A)$ is within the distance $L=2 \arcsinh 2$ of $\partial X_3$. This implies that $X_3$ satisfies the geometric condition $(\star)$ and we obtain the desired surface $X_3$. 	

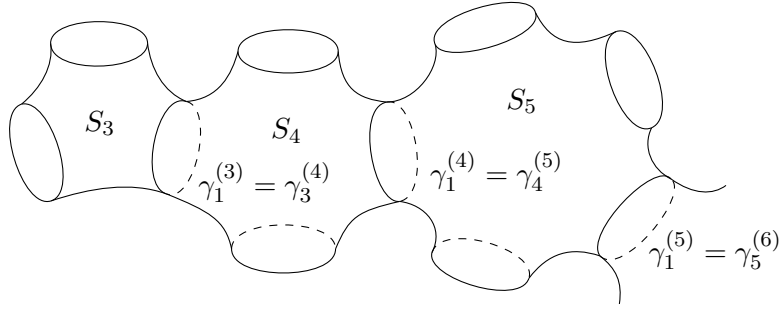
\begin{figure}[ht]
		\begin{tikzpicture}[scale=0.85]

\draw (-3.95,2.15) .. controls (-3.95,2.6) and (-2.45,2.6) .. (-2.45,2.15);
\draw (-3.95,2.15) .. controls (-3.95,1.65) and (-2.45,1.7) .. (-2.45,2.15);

\draw (-1.05,2.0236) .. controls (-1.05,2.4736) and (0.5,2.4736) .. (0.5,2.0236);
\draw (-1.05,2.0236) .. controls (-1.05,1.5736) and (0.5,1.5736) .. (0.5,2.0236);

\draw[dashed] (-1.15,-1.0588) .. controls (-1.15,-0.5588) and (0.4588,-0.5588) .. (0.4588,-1.0588);
\draw (-1.15,-1.0588) .. controls (-1.15,-1.5588) and (0.4588,-1.5588) .. (0.4588,-1.0588);

\draw (2,2.1324) .. controls (1.8,2.5824) and (3.45,3.0324) .. (3.55,2.6324);
\draw (2,2.1324) .. controls (2.15,1.7824) and (3.7,2.2324) .. (3.55,2.6324);

\draw (4.55,-1.1676) .. controls (4.3,-0.8176) and (5.4,0.3324) .. (5.65,0.0324);
\draw[dashed] (4.55,-1.1676) .. controls (4.9,-1.5176) and (5.95,-0.2176) .. (5.65,0.0324);

\draw[dashed] (1.95,-1.1176) .. controls (2.1,-0.6176) and (3.6,-1.0676) .. (3.45,-1.5176);
\draw (1.95,-1.1176) .. controls (1.85,-1.5176) and (3.3,-1.9676) .. (3.45,-1.5176);

\draw (-1.85,1.2412) .. controls (-2.3,1.4) and (-2.6088,-0.05) .. (-2.1588,-0.2);
\draw[dashed] (-1.85,1.2412) .. controls (-1.4,1.1412) and (-1.7088,-0.3676) .. (-2.1588,-0.2);

\draw (1.25,1.2324) .. controls (0.75,1.1324) and (1.05,-0.3676) .. (1.45,-0.3176);
\draw[dashed] (1.25,1.2324) .. controls (1.7,1.3324) and (1.95,-0.2176) .. (1.45,-0.3176);

\draw (-4.4,1.2) .. controls (-4.85,1.05) and (-4.3912,-0.4588) .. (-3.9412,-0.3088);
\draw (-4.4,1.2) .. controls (-3.95,1.3) and (-3.4912,-0.1588) .. (-3.9412,-0.3088);

\draw (4.8,2.3324) .. controls (4.35,2.1324) and (4.9,0.6324) .. (5.35,0.8324);
\draw (4.8,2.3324) .. controls (5.25,2.4824) and (5.8,1.0324) .. (5.35,0.8324);

\draw (-3.95,2.15) .. controls (-3.9,1.8) and (-4.1,1.45) .. (-4.4,1.2);

\draw (-2.45,2.15) .. controls (-2.4,1.75) and (-2.2088,1.3912) .. (-1.85,1.2412);

\draw (-3.9412,-0.3088) .. controls (-3.3412,-0.1088) and (-2.7088,0) .. (-2.1588,-0.2);

\draw (-1.05,2.0236) .. controls (-1.15,1.6588) and (-1.3412,1.2824) .. (-1.85,1.2412);

\draw (-2.1588,-0.2) .. controls (-1.5,-0.4676) and (-1.2412,-0.6088) .. (-1.15,-1.0588);

\draw (0.4588,-1.0588) .. controls (0.6676,-0.5588) and (0.8,-0.4676) .. (1.45,-0.3176);

\draw (0.5,2.0236) .. controls (0.55,1.4324) and (0.8,1.2824) .. (1.25,1.2324);
\draw (2,2.1324) .. controls (2.1,1.7824) and (1.7,1.2324) .. (1.25,1.2324);

\draw (3.55,2.6324) .. controls (3.8,2.1324) and (4.15,2.0324) .. (4.8,2.3324);

\draw (1.45,-0.3176) .. controls (1.85,-0.3676) and (2.15,-0.7176) .. (1.95,-1.1176);

\draw (3.45,-1.5176) .. controls (3.75,-1.1176) and (4.1,-1.0176) .. (4.55,-1.1676);

\draw (5.65,0.0324) .. controls (5.35,0.1824) and (5.2,0.5324) .. (5.35,0.8324);
\draw (5.65,0.0324) .. controls (6,-0.2176) and (6.3,-0.2176) .. (6.5,-0.0676);
\draw (4.55,-1.1676) .. controls (4.8,-1.3676) and (4.9,-1.5676) .. (4.85,-1.9176);
\node at (-3.1912,0.9088) {$S_3$};
\node at (-0.3,0.7912) {$S_4$};
\node at (3.3384,1.25) {$S_5$};
\node at (-0.6324,-0.0176) {$\gamma^{(3)}_{1}=\gamma^{(4)}_{3}$};
\node at (2.9528,0.156) {$\gamma^{(4)}_{1}=\gamma^{(5)}_{4}$};
\node at (6.3352,-1.0676) {$\gamma^{(5)}_{1}=\gamma^{(6)}_{5}$};
\end{tikzpicture}
		\caption{\small{The hyperbolic surface $X_3$ in Example \ref{ex: hyperbolic disks} for the case that the removable set for $\Gamma_{X_3}$ is a disjoint union of hyperbolic disks.}}
		\label{fig:removable set for hyperbolic disks}
	\end{figure}
\end{example}

     \subsection{The relation between the Shiga's condition and the geometric condition $(\star)$.} To investigate the relation between the two conditions,  we first recall some terminology as follows.

     We say that $\mathcal{P}=\{C _{i}\}^{\infty}_{i=1}$ is an \emph{upper-bounded pants decomposition} of $X_{0}$ if there exists a constant $M>0$ such that $\ell_{C_{i}}(X_{0})\leq M$ for each $i\in \mathbb{N}$. Similarly,  we say that $\mathcal{P}=\{C _{i}\}^{\infty}_{i=1}$ is a \emph{lower-bounded pants decomposition} of $X_{0}$ if there exists a constant $m>0$ such that $\ell_{C_{i}}(X_{0})\geq m$ for each $i\in \mathbb{N}$. Furthermore, $\mathcal{P}=\{C _{i}\}^{\infty}_{i=1}$ is said to be a \emph{bounded pants decomposition} of $X_{0}$ if it is both an upper-bounded pants decomposition and a lower-bounded pants decomposition of $X_{0}$. Recall that a hyperbolic surface $X$ of infinite type satisfies the \emph{Shiga's condition} (see \cite{Shiga}) if it admits a bounded pants decomposition.


      We claim that there is no direct relation between the Shiga's condition and the geometric condition $(\star)$.

       Indeed,
       consider the surface $X_{0}$ in Example \ref{Ex:my} and the flute surface $X_{0}'$ in Example \ref{EX2}. The length of the boundary component $\beta_{n}$ of $X_{0}$ (resp. $X_{0}'$) tends to infinity, as $n\rightarrow\infty$. In Example \ref{ex: neighbourhoods of a geodesic}, the surface $X_2$ has a subsequence of boundary components $\{\gamma_n\}$ whose lengths tend to zero. Therefore, these surfaces $X_{0}$, $X_{0}'$ and $X_2$ do not satisfy the Shiga's condition while they satisfy the geometric condition $(\star)$.

       On the other hand, we can find a complete hyperbolic surface $Y_{0}$ of infinite type which has infinitely many geodesic boundary components and satisfies the Shiga's condition but does not satisfy the geometric condition $(\star)$. The surface $Y_{0}$ is constructed as follows:

                 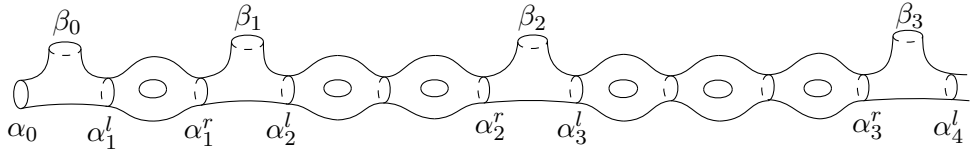
\begin{figure}[ht]
\begin{tikzpicture}[scale=0.3]
\draw (-4.7,0.5) .. controls (-4.7,0.9) and (-3.3,0.9) .. (-3.3,0.5);
\draw[dashed] (-4.7,0.5) .. controls (-4.7,0.1) and (-3.3,0.1) .. (-3.3,0.5);
\draw (-4.7,0.5) .. controls (-4.7,-0.3) and (-4.95,-0.85) .. (-5.95,-0.9);
\draw (-5.95,-0.9) .. controls (-6.35,-0.9) and (-6.35,-2.1) .. (-5.95,-2.15);
\draw (-5.95,-0.9) .. controls (-5.5,-0.9) and (-5.5,-2.15) .. (-5.95,-2.15);
\draw (-3.3,0.5) .. controls (-3.3,-0.35) and (-3.05,-0.8) .. (-2.1,-0.8);
\draw (-5.95,-2.15) .. controls (-4.75,-1.9) and (-3.15,-1.9) .. (-2.15,-2.05);
\draw (-2.1,-0.8) .. controls (-1.75,-0.85) and (-1.75,-2.05) .. (-2.15,-2.05);
\draw[dashed] (-2.1,-0.8) .. controls (-2.5,-0.85) and (-2.5,-2.05) .. (-2.15,-2.05);
\draw (-2.1,-0.8) .. controls (-1.25,-0.8) and (-1.15,0.1) .. (-0.1,0.15);
\draw (-2.15,-2.05) .. controls (-1.3,-2.1) and (-1.3,-2.7) .. (-0.1,-2.7);
\draw (-0.1,0.15) .. controls (1.15,0.1) and (1.15,-0.8) .. (1.95,-0.8);
\draw (-0.1,-2.7) .. controls (1.1,-2.65) and (1.1,-2.05) .. (1.95,-2);
\draw (-0.75,-1.3) .. controls (-0.75,-0.75) and (0.45,-0.7) .. (0.45,-1.3);
\draw (-0.75,-1.3) .. controls (-0.75,-1.8) and (0.45,-1.8) .. (0.45,-1.3);
\draw [dashed](1.95,-0.8) .. controls (1.6,-0.8) and (1.6,-1.95) .. (1.95,-2);
\draw (1.95,-0.8) .. controls (3.2,-0.6) and (3.25,-0.1) .. (3.3,0.8);
\draw (1.95,-2) .. controls (3.3,-1.75) and (4.55,-1.75) .. (5.75,-1.95);
\draw (3.3,0.8) .. controls (3.3,1.2) and (4.7,1.2) .. (4.7,0.8);
\draw[dashed] (3.3,0.8) .. controls (3.3,0.35) and (4.7,0.35) .. (4.7,0.8);
\draw (4.7,0.8) .. controls (4.7,-0.05) and (4.7,-0.65) .. (5.8,-0.7);
\draw (5.8,-0.7) .. controls (6.15,-0.7) and (6.15,-1.9) .. (5.75,-1.95);
\draw[dashed] (5.8,-0.7) .. controls (5.4,-0.7) and (5.4,-1.9) .. (5.75,-1.95);
\draw (5.8,-0.7) .. controls (6.7,-0.7) and (6.75,0.15) .. (7.95,0.2);
\draw (5.75,-1.95) .. controls (6.7,-1.95) and (6.75,-2.55) .. (7.95,-2.65);
\draw (7.95,0.2) .. controls (9.2,0.15) and (9.2,-0.7) .. (10,-0.75);
\draw (7.95,-2.65) .. controls (9.1,-2.6) and (9.1,-1.9) .. (9.95,-1.9);
\draw (7.35,-1.25) .. controls (7.35,-0.75) and (8.55,-0.75) .. (8.55,-1.25);
\draw (7.35,-1.25) .. controls (7.35,-1.75) and (8.55,-1.75) .. (8.55,-1.25);
\draw (10,-0.75) .. controls (10.9,-0.75) and (10.95,0.2) .. (12.25,0.2);
\draw (9.95,-1.9) .. controls (11,-1.9) and (11.1,-2.55) .. (12.2,-2.6);
\draw (11.6,-1.25) .. controls (11.6,-0.75) and (12.8,-0.75) .. (12.8,-1.25);
\draw (11.6,-1.25) .. controls (11.6,-1.75) and (12.8,-1.75) .. (12.8,-1.25);
\draw (10,-0.75) .. controls (10.35,-0.75) and (10.35,-1.9) .. (9.95,-1.9);
\draw[dashed] (10,-0.75) .. controls (9.65,-0.75) and (9.65,-1.9) .. (9.95,-1.9);
\draw (12.25,0.2) .. controls (13.3,0.2) and (13.35,-0.7) .. (14.35,-0.75);
\draw (12.2,-2.6) .. controls (13.25,-2.55) and (13.3,-1.9) .. (14.35,-1.9);
\draw (14.35,-0.75) .. controls (14.7,-0.75) and (14.7,-1.9) .. (14.35,-1.9);
\draw[dashed] (14.35,-0.75) .. controls (14,-0.75) and (14.05,-1.9) .. (14.35,-1.9);
\draw (14.35,-0.75) .. controls (15.2,-0.7) and (15.8,-0.1) .. (15.85,0.8);
\draw (14.35,-1.9) .. controls (15.9,-1.75) and (16.95,-1.75) .. (18.4,-1.9);
\draw (15.85,0.8) .. controls (15.85,1.2) and (17.15,1.2) .. (17.15,0.85);
\draw[dashed] (15.85,0.8) .. controls (15.85,0.4) and (17.15,0.4) .. (17.15,0.85);
\draw (17.15,0.85) .. controls (17.15,-0.15) and (17.65,-0.65) .. (18.4,-0.7);
\draw (18.4,-0.7) .. controls (18.8,-0.7) and (18.8,-1.9) .. (18.4,-1.9);
\draw [dashed](18.4,-0.7) .. controls (18.05,-0.75) and (18.05,-1.9) .. (18.4,-1.9);
\draw (18.4,-0.7) .. controls (19.35,-0.65) and (19.35,0.2) .. (20.5,0.25);
\draw (18.4,-1.9) .. controls (19.4,-1.9) and (19.4,-2.65) .. (20.5,-2.7);
\draw (20.5,0.25) .. controls (21.6,0.2) and (21.65,-0.65) .. (22.55,-0.7);
\draw (20.5,-2.7) .. controls (21.7,-2.65) and (21.65,-1.9) .. (22.55,-1.8);
\draw (19.85,-1.25) .. controls (19.85,-0.75) and (21.1,-0.75) .. (21.1,-1.25);
\draw (19.85,-1.25) .. controls (19.85,-1.8) and (21.1,-1.8) .. (21.1,-1.25);
\draw (1.95,-0.8) .. controls (2.35,-0.8) and (2.35,-1.95) .. (1.95,-2);
\draw (22.55,-0.7) .. controls (23,-0.7) and (22.95,-1.8) .. (22.55,-1.8);
\draw [dashed](22.55,-0.7) .. controls (22.15,-0.7) and (22.15,-1.8) .. (22.55,-1.8);
\draw (22.55,-0.7) .. controls (23.5,-0.65) and (23.55,0.15) .. (24.65,0.2);
\draw (22.55,-1.8) .. controls (23.55,-1.85) and (23.6,-2.65) .. (24.7,-2.7);
\draw (24.65,0.2) .. controls (25.85,0.2) and (25.85,-0.65) .. (26.85,-0.65);
\draw (24.7,-2.7) .. controls (25.95,-2.65) and (25.95,-1.85) .. (26.8,-1.8);
\draw (24,-1.3) .. controls (24,-0.75) and (25.25,-0.75) .. (25.25,-1.3);
\draw (24,-1.3) .. controls (24,-1.8) and (25.25,-1.8) .. (25.25,-1.3);
\draw (26.85,-0.65) .. controls (27.2,-0.65) and (27.15,-1.8) .. (26.8,-1.8);
\draw[dashed] (26.85,-0.65) .. controls (26.45,-0.65) and (26.45,-1.8) .. (26.8,-1.8);
\draw (26.85,-0.65) .. controls (27.95,-0.6) and (27.95,0.2) .. (29.05,0.3);
\draw (26.8,-1.75) .. controls (27.95,-1.8) and (28.05,-2.55) .. (29.1,-2.6);
\draw (29.05,0.3) .. controls (30.1,0.25) and (30.1,-0.6) .. (31,-0.65);
\draw (29.1,-2.6) .. controls (30.1,-2.55) and (30.1,-1.8) .. (31,-1.8);
\draw (28.4,-1.25) .. controls (28.4,-0.75) and (29.6,-0.75) .. (29.6,-1.25);
\draw (28.4,-1.25) .. controls (28.4,-1.75) and (29.6,-1.75) .. (29.6,-1.25);
\draw (31,-0.65) .. controls (31.35,-0.65) and (31.3,-1.8) .. (31,-1.8);
\draw [dashed](31,-0.65) .. controls (30.65,-0.6) and (30.65,-1.8) .. (31,-1.8);
\draw (31,-0.65) .. controls (32.1,-0.6) and (32.25,0) .. (32.3,1.05);
\draw (31,-1.8) .. controls (32.15,-1.7) and (33.8,-1.65) .. (35.65,-1.8);
\draw (32.3,1.05) .. controls (32.3,1.4) and (33.6,1.4) .. (33.6,1.05);
\draw[dashed] (32.3,1.05) .. controls (32.3,0.7) and (33.6,0.7) .. (33.6,1.05);
\draw (33.6,1.05) .. controls (33.6,0.05) and (33.75,-0.65) .. (35.55,-0.65);
\draw (34.85,-0.65) .. controls (35.2,-0.65) and (35.2,-1.75) .. (34.85,-1.75);
\draw [dashed](34.85,-0.65) .. controls (34.5,-0.65) and (34.5,-1.75) .. (34.85,-1.75);
\node at (-3.9,1.65) {$\beta_{0}$};
\node at (4.05,1.85) {$\beta_{1}$};
\node at (16.5,1.85) {$\beta_{2}$};
\node at (33.05,1.95) {$\beta_{3}$};
\node at (-5.85,-3.2) {$\alpha_{0}$};
\node at (-2.35,-3.2) {$\alpha^{l}_{1}$};
\node at (1.9,-3.2) {$\alpha^{r}_{1}$};
\node at (5.5,-3.05) {$\alpha^{l}_{2}$};
\node at (14.65,-2.95) {$\alpha^{r}_{2}$};
\node at (18.25,-3) {$\alpha^{l}_{3}$};
\node at (31.2,-2.9) {$\alpha^{r}_{3}$};
\node at (34.75,-2.95) {$\alpha^{l}_{4}$};
\end{tikzpicture}
\caption{\small{An example $Y_{0}$ which satisfies Shiga's condition but does not satisfy the geometric condition $(\star)$.}}
            \label{fig:EX2}
		\end{figure}

         Let $Z_{0}$ be a flute surface with $\alpha_{n}=\beta_{n}=1$ for $n\geq0$, where $\alpha_{n}$ and $\beta_{n}$ denote the same simple closed geodesics as in Proposition \ref{EX1}. Then we construct $Y_{0}$ by inserting a hyperbolic surface of genus $n$ with two geodesic boundary components $\alpha^{l}_{n}, \alpha^{r}_{n}$ (which admits a pair of pants decomposition with all decomposing curves of length 1) along both sides of $\alpha_{n}$ for $n\geq1$, as indicated in Figure \ref{fig:EX2}. 
          It follows easily that $Y_{0}$ satisfies the Shiga's condition. However, it follows from the construction of $Y_0$ and Lemma \ref{the cover of a removable set} that for any $L>0$ and any removable set $A$ for $\Gamma_{Y_0}$, the projection $\pi(A)$ fails to cover $Y_0\setminus B(\partial Y_0; L)$, where $B(\partial Y_0; L)$ consists of the points on $Y_0$ lying within the distance $L$ of $\partial Y_0$. This implies that $Y_0$ does not satisfy the geometric condition $(\star)$.

          In particular, there exist complete hyperbolic surfaces of infinite type which satisfy both the Shiga's condition and the geometric condition $(\star)$. The surface $Z_{0}$ mentioned above is such an example.


\begin{thebibliography}{CGQ}

\bibitem{ALPS1} D. Alessandrini, L. Liu, A. Papadopoulos, W. Su, The horofunction compactification of the arc metric on Teichm\"uller space. Topology Appl. 208 (2016), 160-191.

\bibitem{ALPSS} D. Alessandrini, L. Liu, A. Papadopoulos, W. Su, and Z. Sun, On Fenchel-Nielsen coordinates on Teichm\"uller spaces of surfaces of infinite type.  Ann. Acad. Sci. Fenn. Math. 36, no.2, 2011, 621-659.

\bibitem{Basmajian1} A. Basmajian, The orthogonal spectrum of a hyperbolic manifold. Amer. J. Math. 115 (1993), no.5, 1139-1159.

\bibitem{Basmajian2} A. Basmajian, Hyperbolic structures for surfaces of infinite type. Trans. Amer. Math. Soc. 336 (1993), no.1, 421-444.

\bibitem{Basmajian3} A. Basmajian, D. Saric, Geodesically complete hyperbolic structures. arXiv:1508.02280, 2015.

\bibitem{Bers} L. Bers, Nielsen extension of Riemann surfaces.  Ann. Acad. Sci. Fenn. Math. no.2, 1976, 29-34.

\bibitem{BA} A. Beurling, L. Ahlfors, The boundary correspondence under quasiconformal mappings.  Acta Math. 96 (1956), 125-142.

\bibitem{Bridgeman2} M. Bridgeman, S. P. Tan, Identities on hyperbolic manifolds.  arXiv:1309.3578, 2013.

\bibitem{Buser} P. Buser, Geometry and spectra of compact Riemann surfaces. Reprint of the 1992 edition. Modern Birkh\"auser Classics. Birkh\"auser Boston, Inc., Boston, MA, 2010.

\bibitem{Cao} J. Cao, The Bers-Nielsen kernels and souls of open surfaces with negative curvature. Michigan Math. J. Volume 41, Issue 1 (1994), 13-30.

\bibitem{Croke} C. Croke, A. Fathi, J. Feldman, The marked length-spectrum of a surface of nonpositive curvature. Topology 31 (1992), no. 4, 847-855.

\bibitem{Kassel} J. Danciger, F. Gueritaud, F. Kassel, Margulis spacetimes via the arc complex. Invent. Math. 204 (2016), no. 1, 133-193.

\bibitem{DLK} M. Duchin, C. J. Leininger, K. Rafi, Length spectra and degeneration of flat metrics. Invent. Math. (2010) 82: 231-277.

\bibitem{Earle} C. J. Earle, Reduced Teichm\"uller spaces.  Trans. Amer. Math. Soc. 126 (1967), no.1, 54-63.

\bibitem{Li}  Z. Li, Teichm\"uller metric and length spectrums of Riemann surfaces. Sci. Sinica Ser. A 29, no. 3, 1986, 265--274.

\bibitem{Liu2001} L. Liu, On the metrics of length spectrum in Teichm\"uller space. Chinese J. Cont. math. 22 (1), 2001, 23--34.

\bibitem{LP} L. Liu, A. Papadopoulos, Some metrics on Teichm\"uller spaces of surfaces of infinite type. Trans. Amer. Math. Soc. 363 (2011), no.8, 4109-4134.

\bibitem{LPST}  L. Liu, A. Papadopoulos, W. Su and G. Th\'eret,  On length spectrum metrics and weak metrics on Teichm\"uller spaces of surfaces with boundary. Ann. Acad. Sci. Fenn. Math. 35, no.1, 2010, 255-274.

\bibitem{LTBoundary}  L. Liu, A. Papadopoulos, W. Su and G. Th\'eret, Length spectra and the Teichm\"uller metric for surfaces with boundary. Monatsh. Math. 161, no.3, 2010, 295-311.	

\bibitem{LSZ} L. Liu, W. Su, Y. Zhong, On metrics defined by length spectra on Teichm\"uller spaces of surfaces with boundary. Ann. Acad. Sci. Fenn. Math.40, no.2, 2015, 617-644.

\bibitem{Matsuzaki1} K. Matsuzaki, Ergodic properties of discrete groups; in heritance to normal subgroups and invariance under quasiconformal deformations. J. Math. Kyoto Univ. 33-1 (1993) 205-226.

\bibitem{Matsuzaki} K. Matsuzaki, Hausdorff dimension of the limit sets of infinitely generated Kleinian groups. Math. Proc. Cambridge Philos. Soc. 128 (2000), no.1, 123-139.

\bibitem{Mirzakhani} M. Mirzakhani, Simple geodesics and Weil-Petersson volumes of moduli spaces of bordered Riemann surfaces. Invent. Math. 167 (2007), no.1, 179-222.

\bibitem{P} A. Papadopoulos, On Thurston's boundary of Teichm\"uller space and the extension of earthquakes. Topology Appl. 41 (1991), no.3, 147-177.

\bibitem{Papado-Th2007} A. Papadopoulos. and G. Th\'eret, On the topology defined by Thurston's asymmetric metric.  Math. Proc. Camb. Phil. Soc., 142, 2007, 487-496.

\bibitem{PG2010} A. Papadopoulos. and G. Th\'eret,  Shortening all the simple closed geodesics on surfaces with boundary. Proc. Amer. Math. Soc. 138 (2010), 1775-1784.

\bibitem{Parlier} H. Parlier, Lengths of geodesics on Riemann surfaces with boundary, Ann. Acad. Sci. Fenn. Math. 30 (2005), no. 2, 227-236.

\bibitem{Shiga} H. Shiga, On a distance defined by the length spectrum of Teichmuller space. Ann. Acad. Sci. Fenn. Math. 28 (2003), no. 2, 315-326.

\bibitem{Su} W. Su, Problems on Thurston metric. Handbook of Teichm\"uler theory, Vol. V, 55-72, IRMA Lect. Math. Theor. Phys., 26, Eur. Math. Soc., Z\"urich, 2016.

\bibitem{Taniguchi} M. Taniguchi, On certain boundary points on the Wiener's compactifications of open Riemann surfaces. J. Math. Kyoto Univ. 20 (1980), no. 4, 709-721.


 \bibitem{Thurston} W. P. Thurston, The geometry and topology of Three-manifolds. Mimeographed notes, Princeton University, 1976.

 \bibitem{Thurston-notes} W. P. Thurston, Three-dimensional geometry and topology, Vol.1, Princeton mathematical series 35, Princeton University Press, 1997.

\bibitem{Thurston1998}  W. P. Thurston, Minimal stretch maps between hyperbolic surfaces. preprint, 1986, Arxiv:math GT/9801039.

\bibitem{Wolpert} S. Wolpert, The length spectra as moduli for compact Riemann surfaces. Ann. Math. 109 (1979), 323-351.

		\end{thebibliography}
      \end{document}